\documentclass[9pt,a4paper]{amsart}
\usepackage{amsmath}
\usepackage{amsfonts}
\usepackage{amsthm}
\usepackage{amssymb}
\usepackage{url}
\usepackage{mathtools,bm}
\usepackage{verbatim}
\usepackage[usenames,dvips]{color}
\usepackage{graphicx}
\usepackage{psfrag, float}  
\usepackage{xcolor}
\usepackage{mathrsfs}
\usepackage{enumitem}
\usepackage[colorlinks=true]{hyperref}

\setlist{leftmargin=*,itemindent=10pt,align=left,itemsep=0.5\baselineskip}

\newtheorem{lemma}{Lemma}[section]
\newtheorem{theorem}[lemma]{Theorem}
\newtheorem{corollary}[lemma]{Corollary}
\newtheorem{definition}[lemma]{Definition}
\newtheorem{remark}[lemma]{Remark}
\newtheorem{remarks}[lemma]{Remarks}

\newtheorem{proposition}[lemma]{Proposition}

\newcommand{\bfalpha}{{\bm{\alpha}}}

\newcommand{\bflambda}{{\bm{\lambda}}}

\newcommand{\bfmu}{{\bm{\mu}}}

\newcommand{\bfgamma}{{\bm{\gamma}}}

\newcommand{\bfA}{{\bm{A}}}
\newcommand{\bfb}{{\bm{b}}}
\newcommand{\bfB}{{\bm{B}}}
\newcommand{\bfc}{{\bm{c}}}
\newcommand{\bfC}{{\bm{C}}}

\newcommand{\bfD}{{\bm{D}}}
\newcommand{\bfe}{{\bm{e}}}
\newcommand{\bff}{{\bm{f}}}
\newcommand{\bfF}{{\bm{F}}}
\newcommand{\bfg}{{\bm{g}}}
\newcommand{\bfG}{{\bm{G}}}

\newcommand{\bfk}{{\bm{k}}}
\newcommand{\bfK}{{\bm{K}}}
\newcommand{\bfl}{{\bm{\ell}}}
\newcommand{\bfL}{{\bm{L}}}
\newcommand{\bfM}{{\bm{M}}}
\newcommand{\bfn}{{\bm{n}}}
\newcommand{\bfN}{{\bm{N}}}

\newcommand{\bfS}{{\bm{S}}}
\newcommand{\bfT}{{\bm{T}}}
\newcommand{\bfu}{{\bm{u}}}
\newcommand{\bfU}{{\bm{U}}}
\newcommand{\bfv}{{\bm{v}}}

\newcommand{\bfx}{{\bm{x}^\prime}}

\newcommand{\bfz}{{\bm{z}}}
\newcommand{\bfr}{{\bm{r}}}
\newcommand{\bfzero}{{\bm{0}}}

\DeclareMathOperator{\curl}{curl}
\DeclareMathOperator{\Div}{div}
\DeclareMathOperator{\Grad}{grad}
\DeclareMathOperator{\Op}{Op}
\DeclareMathOperator{\ran}{ran}

\newcommand{\nablap}{\nabla^\perp}
\newcommand{\nablac}{\nabla \cdot}
\newcommand{\nablapc}{\nablap\! \cdot}

\newcommand{\In}{\qquad \text{in} \;}
\newcommand{\at}{\qquad \text{at} \; }

\newcommand{\ee}{\mathrm{e}}
\newcommand{\ii}{\mathrm{i}}

\title[A generalised Dirichlet--Neumann operator]{Analytical study of a generalised Dirichlet--Neumann operator and application to three-dimensional water waves on Beltrami flows}

\author{M. D. Groves}
\address{M. D. Groves, Fachrichtung Mathematik, Universit\"at des Saarlandes, Postfach 15 11 50, 66041 Saarbr\"ucken, Germany}
\email{groves@math.uni-sb.de}

\author{D. Nilsson}
\address{D. Nilsson, Centre for Mathematical Sciences, Lund University, P.O. Box 118, 22100 Lund, Sweden}
\email{dag.nilsson@math.lu.se}

\author{S. Pasquali}
\address{S. Pasquali, Universit\'e Paris-Saclay, CNRS, Laboratoire de math\'ematiques d'Orsay, 91405, Orsay, France}
\email{stefano.pasquali@universite-paris-saclay.fr}

\author{E. Wahl\'en}
\address{E. Wahl\'en, Centre for Mathematical Sciences, Lund University, P.O. Box 118, 22100 Lund, Sweden}
\email{erik.wahlen@math.lu.se}

\numberwithin{equation}{section}

\begin{document}

\begin{abstract}
We consider three-dimensional doubly periodic steady water waves with vorticity, under the action of gravity
and surface tension; in particular we consider so-called Beltrami flows, for which the velocity field and the vorticity are collinear. 
We adapt a recent formulation of the corresponding problem for localised waves which involves a generalisation of the classical
Dirichlet--Neumann operator. We study this operator in detail, extending some well-known results for the classical Dirichlet--Neumann operator, such as the Taylor expansion in homogeneous powers of the wave profile, the computation of its differential and the asymptotic expansion of its associated symbol. A new formulation of the problem as a single equation for the wave profile is also presented and discussed in a similar vein.
As an application of these results we prove existence of doubly periodic gravity-capillary steady waves and construct approximate doubly periodic gravity steady waves.
\end{abstract}
\maketitle

\allowdisplaybreaks

\section{Introduction} \label{sec:intro}

This paper is concerned with three-dimensional doubly periodic steady water waves with vorticity, under the action of gravity and surface tension. Irrotational water waves have been studied extensively, both in two and three dimensions (see the survey paper by Haziot \emph{et al.} \cite{HaziotHurStraussTolandWahlenWalshWheeler22} and references therein); fewer results are available for non-zero vorticity, although it may be significant for modelling the interaction of three-dimensional waves with non-uniform currents. We restrict ourselves to Beltrami fields, in which the velocity field $\bfu$ and the vorticity $\curl \bfu$ are collinear, so that $\curl \bfu = \alpha \bfu$; more precisely, we consider the so-called strong Beltrami fields, for which the proportionality factor $\alpha$ is a constant (this case appears to be the most relevant, since Enciso and Peralta-Salas \cite{EncisoPeraltaSalas16} proved that Beltrami fields with non-constant proportionality factors are `rare' in a topological sense).

The importance of Beltrami fields in the context of ideal fluids, and more precisely in the context of stationary Euler flows, was highlighted by Arnold \cite{Arnold66} and Arnold and Khesin \cite{ArnoldKhesin}: indeed, Arnold's structure theorem ensures that, under suitable technical assumptions, a smooth stationary solution to the three-dimensional Euler equation is either integrable or a Beltrami field. It is thus natural to expect that more complex dynamics (usually associated to turbulent flows in physical literature) in stationary fluids are related to Beltrami fields (see Monchaux \emph{et al.} \cite{MonchauxRaveletDubrulleChiffaudelDaviaud06}). 
The dynamics of Beltrami fields, and in particular the dynamics of the so-called ABC flows, have been numerically studied by H\'{e}non \cite{Henon66} and Dombre \emph{et al.} \cite{DombreFrischGreeneHenonMehrSoward86}. Such studies lead to the conjecture that  Beltrami fields should exhibit chaotic dynamics together with a positive measure set of invariant tori, much like the restriction to an energy level of a typical mechanical system with two degrees of freedom; recently Enciso, Peralta-Salas and Romaniega \cite{EncisoPeraltaSalasRomaniega23} proved that with probability one a random Beltrami field in ${\mathbb R}^3$ exhibits chaotic regions that coexist with invariant tori of complicated topology.

There has recently been some interest in variational formulations of the three-dimensional steady water-wave problem 
with relative velocities given by Beltrami fields. We mention a recent variational formulation by Lokharu and Wahlén \cite{LokharuWahlen19} for doubly periodic waves which is valid under general assumptions on the wave profile (including for example the case of overhanging wave profiles). More recently, Groves and Horn  \cite{GrovesHorn20} gave another variational formulation for localised waves (solitary waves)
under the more classical assumption that the free surface is given by the graph of an unknown function $\eta$ depending only on the horizontal directions. Their formulation, which can be considered as a generalisation of an alternative variational framework for three-dimensional irrotational water waves by Benjamin \cite[\S6.6]{Benjamin84}, is not only more explicit, but it allows one to recover the classical  Zakharov--Craig--Sulem formulation of steady water waves in the irrotational case $\alpha =0$. Moreover, this formulation leads naturally to the definition of a generalised Dirichlet--Neumann operator $H(\eta)$ which reduces to the classical Dirichlet--Neumann operator in the irrotational case.

In this paper we perform an analytical study of the generalised Dirichlet--Neumann operator (whose definition is subtly different in the present context of doubly periodic waves)
and of a related operator appearing in a new single equation formulation of the problem, extending some well-known results for the classical Dirichlet--Neumann operator, such as the Taylor expansion in homogeneous powers of the profile $\eta$ by Craig and Sulem \cite{CraigSulem93}, the computation of its differential by Lannes \cite[\S3.3]{Lannes}, and the asymptotic expansion of its associated symbol (see Alazard and M{\'e}tivier
\cite[\S2.4]{AlazardMetivier09}).

As an application of the above results, we prove the existence of doubly periodic gravity-capillary waves by Lyapunov--Schmidt reduction, recovering a result recently given by Lokharu, Seth and Wahlén \cite{LokharuSethWahlen20}. We also show how the reduction can be formally carried out in the absence of surface tension and thus compute approximate doubly periodic gravity waves in the form of formal power series.
The failure of the Lyapunov--Schmidt reduction for gravity waves is due to the presence of small divisors when attempting to invert the
relevant linear operator. This problem has been overcome for irrotational waves by Iooss and Plotnikov \cite{IoossPlotnikov09,IoossPlotnikov11} using Nash-Moser theory; its treatment for Beltrami flows is deferred to a future article.

\subsection{The hydrodynamic problem} \label{subsec:model}

We consider an incompressible inviscid fluid occupying a three-dimensional domain with flat bottom, under the action of gravity and surface tension. We study steady water waves, namely a fluid flow in which the velocity field and the free-surface profile are stationary with respect to a uniformly translating frame. In this moving frame, the fluid domain can be parametrized by
\begin{align*}
D_\eta &\coloneqq  \{ (\bfx,z) \in \mathbb{R}^2 \times \mathbb{R}\colon -h < z < \eta(\bfx) \},
\end{align*}
so that the free surface is given by the graph of an unknown function $\eta\colon\mathbb{R}^2 \to (-h,\infty)$, and $h>0$ is the depth of the fluid.
We consider a so-called strong Beltrami flow, in which the velocity field $\bfu\colon\overline{D_\eta} \to \mathbb{R}^3$ and the vorticity $\curl \bfu$ are collinear, that is $\curl \bfu = \alpha \bfu$ for some constant $\alpha$. The equations describing the flow are given by
\begin{alignat}{2}
\Div \bfu &= 0 & & \In D_\eta, \label{eq:Solenoidal} \\
\curl \bfu &= \alpha \bfu & & \In D_\eta, \\
\bfu \cdot \bfe_3 &= 0 & & \at z=-h, \label{eq:Imperm} \\
\bfu \cdot \bfn &= 0 & & \at z=\eta, \label{eq:KinFree} \\
\frac{1}{2} |\bfu|^2 + g \eta & -\beta \; \left( \frac{\eta_x}{ (1+|\nabla\eta|^2)^{1/2} } \right)_x - \beta \; \left( \frac{\eta_y}{ (1+|\nabla\eta|^2)^{1/2} }\right)_y = \frac{1}{2} |\bfc|^2 & & \at z=\eta, \label{eq:DynFree}
\end{alignat}
where $\nabla \eta\coloneqq(\eta_x,\eta_y)^T$, $g$ is the acceleration due to gravity, $\beta$ is the coefficient of surface tension, $\bfc\coloneqq(c_1,c_2)^T$ is the wave velocity, $\bfe_3\coloneqq(0,0,1)^T$ and
$$\bfn\coloneqq\frac{1}{1+|\nabla \eta|^2}\bfN, \qquad \bfN\coloneqq(-\eta_x,-\eta_y,1)^T$$
denotes the outward unit normal vector. We discuss doubly periodic solutions to
\eqref{eq:Solenoidal}--\eqref{eq:DynFree}, that is solutions which satisfy
$$\eta(\bfx+\bflambda)=\eta(\bfx), \qquad \bfu(\bfx+\bflambda,z)=\bfu(\bfx,z)$$
for every $\bflambda \in \Lambda$, where $\Lambda$ is the lattice given by
$$
\Lambda \coloneqq  \{ \bflambda = m_1 \bflambda_1 + m_2 \bflambda_2 \colon m_1, m_2 \in \mathbb{Z}\}
$$
for two linearly independent vectors $\bflambda_1$, $\bflambda_2$. The functions $\eta$ and $\bfu$ are thefore defined on the
periodic domains ${\mathbb R}^2/\Lambda$ and (with a slight abuse of notation) $D_\eta/\Lambda$.

A `trivial solution' of \eqref{eq:Solenoidal}--\eqref{eq:DynFree} is given by $(0,\bfu^\star)$, where $\bfu^\star$ is the two-parameter family of laminar flows
\begin{align*} 
\bfu^\star &\coloneqq  c_1 \bfu^{(1)} + c_2 \bfu^{(2)}, \; \; c_1,c_2 \in \mathbb{R}, \\
\bfu^{(1)} &\coloneqq  (\cos(\alpha z), -\sin(\alpha z),0)^T, \nonumber \\
\bfu^{(2)} &\coloneqq  (\sin(\alpha z), \cos(\alpha z),0)^T. \nonumber
\end{align*}
We consider solutions $(\eta,\bfu)$ of \eqref{eq:Solenoidal}--\eqref{eq:DynFree} which are small perturbations of $(0,\bfu^\star)$; setting $\bfv= \bfu - \bfu^\star$ and representing the velocity field $\bfv$ by a solenoidal vector potential $\bfA$, we seek solutions $(\eta,\bfA)$ of the equations
\begin{alignat}{2}
\Div \bfA &= 0 & & \In D_\eta, \label{eq:Solenoidal2} \\
\curl \curl \bfA &= \alpha \curl \bfA & & \In D_\eta, \\
\bfA \times \bfe_3 &= \bfzero & & \at z=-h, \label{eq:Imperm2} \\
\bfA \cdot \bfn &= 0 & & \at z=\eta, \\
\curl \bfA \cdot \bfn + \bfu^\star \cdot \bfn &= 0 & & \at z=\eta, \label{eq:KinFree2} \\
\frac{1}{2} |\curl \bfA|^2 + \curl \bfA \cdot \bfu^\star + g  \eta & -\beta \; \left( \frac{\eta_x}{ (1+|\nabla\eta|^2)^{1/2} } \right)_x - \beta \; \left( \frac{\eta_y}{ (1+|\nabla\eta|^2)^{1/2} }\right)_y = 0 & & \at z=\eta. \label{eq:DynFree2}
\end{alignat}
Note that \eqref{eq:Solenoidal}--\eqref{eq:Imperm} are implied by \eqref{eq:Solenoidal2}--\eqref{eq:Imperm2},
while \eqref{eq:KinFree}, \eqref{eq:DynFree} are equivalent to \eqref{eq:KinFree2}, \eqref{eq:DynFree2};
furthermore $\bfu^\star = \curl \bfA^\star$, where
\begin{align*}
\bfA^\star &\coloneqq  \frac{c_1}{\alpha} \bfA^{(1)} + \frac{c_2}{\alpha} \bfA^{(2)}, \\
\bfA^{(1)} &\coloneqq  (\cos(\alpha z)-1, -\sin(\alpha z),0)^T, \\
\bfA^{(2)} &\coloneqq  (\sin(\alpha z), \cos(\alpha z)-1,0)^T. 
\end{align*}

\begin{remark}
In the irrotational case $\alpha =0$ we can write $\curl \bfA = \mathrm{grad}\,\varphi$ for a scalar potential $\varphi$, so that
\eqref{eq:Solenoidal2}--\eqref{eq:DynFree2} becomes the classical steady water-wave problem
\begin{alignat*}{2}
\Delta \varphi & =0 & & \In D_\eta, \\
\partial_n \varphi & =0 & & \at z=-h, \\
(1+|\nabla\eta|^2)^\frac{1}{2}\partial_n \varphi & = \bfc \cdot \nabla \eta & & \at z = \eta, \\
\frac{1}{2} |\Grad \varphi|^2 + \bfc \cdot (\varphi_x,\varphi_y)^T + g  \eta & -\beta \; \left( \frac{\eta_x}{ (1+|\nabla\eta|^2)^{1/2} } \right)_x - \beta \; \left( \frac{\eta_y}{ (1+|\nabla\eta|^2)^{1/2} }\right)_y = 0 & & \at z=\eta.
\end{alignat*}
\end{remark}

\subsection{The formulation} \label{subsec:Lagrangian}

Let $\bfF=(F_1,F_2,F_3)^T$ be a three-dimensional vector field, and denote by $\bfF_\mathrm{h}=(F_1,F_2)^T$ its horizontal component and by
$\bfF_{\parallel}=\bfF_\mathrm{h}+F_3\nabla\eta|_{z=\eta}$ the horizontal component of its tangential part at $z=\eta$. Let $\bff=(f_1,f_2)^T$
be a two-dimensional vector field and write $\bff^\perp =(f_2,-f_1)^T$. According to the Hodge--Weyl decomposition for doubly periodic vector fields on $\mathbb{R}^2$ (see Majda and Bertozzi \cite[Proposition 1.18]{MajdaBertozzi}) we have
\begin{equation}
\bff = \bfgamma+\nabla\Phi+\nablap\Psi, \label{eq:HW decomp}
\end{equation}
$$
\bfgamma \coloneqq  \langle \bff \rangle, \quad
\Phi \coloneqq  \Delta^{-1}(\nablac \bff), \quad \Psi\coloneqq \Delta^{-1}(\nablapc \bff),
$$
where
$\langle \bff \rangle$ denotes the mean value of $\bff$ over one periodic cell,
$\nabla \coloneqq  (\partial_x,\partial_y)^T,$ $\nablap \coloneqq  (\partial_y,-\partial_x)^T$
and $\Delta^{-1}$ is the two-dimensional periodic Newtonian potential.

Equations \eqref{eq:Solenoidal2}--\eqref{eq:DynFree2} can
be reformulated in terms of $\eta$ and the mean-value and gradient-potential parts of $(\curl \bfA)_\parallel$
using the following procedure. Fix $\bfgamma$ and $\Phi$, let $\bfA$ be the unique solution of the boundary-value problem
\begin{alignat}{2}
\Div \bfA &= 0 & & \In D_\eta, \label{Intro A BVP 1} \\
\curl \curl \bfA &= \alpha  \curl \bfA & & \In D_\eta, \label{Intro A BVP 2} \\
\bfA \times \bfe_3 &= \bfzero & & \at z=-h, \label{Intro A BVP 3} \\
\bfA \cdot \bfn &= 0 & & \at z=\eta, \label{Intro A BVP 4} \\
(\curl \bfA)_\parallel &= \bfgamma+\nabla\Phi - \alpha \nablap\Delta^{-1}(\nablac \bfA_\parallel^\perp) & & \at z=\eta, \label{Intro A BVP 5}
\end{alignat}
and define  the \emph{generalised Dirichlet--Neumann operator} by the formula
\begin{align} \label{eq:GenDNOp}
H(\eta)(\bfgamma,\Phi) &\coloneqq  \curl \bfA \cdot \bfN|_{z=\eta}=\nablac \bfA_\parallel^\perp.
\end{align}
(Note that $\Psi=\Delta^{-1}(\nablapc (\curl \bfA)_{\parallel})$ is necessarily given by $\Psi = -\alpha\, \Delta^{-1}(\nablac\bfA_\parallel^\perp)$ because
\begin{equation}
\Psi = -\Delta^{-1}(\nablac\curl \bfA^{\perp}_\parallel)=-\Delta^{-1}(\curl\curl \bfA\cdot\bfN\big|_{z=\eta})
=-\alpha\, \Delta^{-1}( \curl \bfA\cdot\bfN\big|_{z=\eta})=- \alpha\, \Delta^{-1} (\nablac\bfA^{\!\perp}_\parallel),
\label{eq:compatibility}
\end{equation}
in which the vector identity $\curl \bfF\cdot\bfN\big|_{z=\eta}=\nablac \bfF^\perp_\parallel$ has been used.)

\begin{proposition}
Equations \eqref{eq:KinFree2} and \eqref{eq:DynFree2} are equivalent to
\begin{align}
& H(\eta)(\bfgamma,\Phi) + \bfu^\star \cdot \bfN|_{z=\eta} =0, \label{eq:GZCS1} \\
& \frac{1}{2} |\bfK(\eta)(\bfgamma,\Phi)|^2\! -\! \frac{ ( H(\eta)(\bfgamma,\Phi) + \bfK(\eta)(\bfgamma,\Phi) \!\cdot\! \nabla\eta)^2 }{2(1+|\nabla\eta|^2)}  \nonumber \\
&\qquad\quad\mbox{} + \bfK(\eta)(\bfgamma,\Phi) \cdot \bfu_\mathrm{h}^\star|_{z=\eta}+ g\eta - \beta \left( \frac{\eta_x}{ (1+|\nabla\eta|^2)^{1/2} } \right)_{\!\!x}\!\!\! - \beta \left( \frac{\eta_y}{ (1+|\nabla\eta|^2)^{1/2} } \right)_{\!\!y}\!\!\!= 0, \label{eq:GZCS2}\hspace{-0.35cm}
\end{align}
where
$$
\bfK(\eta)(\bfgamma,\Phi) \coloneqq  \bfgamma+\nabla\Phi - \alpha  \nablap \Delta^{-1}(H(\eta)(\bfgamma,\Phi)).
$$
\end{proposition}

This proposition, which is established by an elementary calculation, shows that the mathematical problem reduces to
solving \eqref{eq:GZCS1} and \eqref{eq:GZCS2} for $\eta$ and $\Phi$ (with an arbitrary choice of $\bfgamma$); the velocity field $\bfv=\curl \bfA$ is recovered
by solving \eqref{Intro A BVP 1}--\eqref{Intro A BVP 5}. The method was first given in the context of solitary waves
(with a slightly different Hodge--Weyl decomposition for spatially extended functions) by Groves and Horn \cite{GrovesHorn20}; note however the
spurious extra term in the statement of the equations in that reference.

\begin{remark}
In the irrotational case $\alpha =0$ one finds that $\curl \bfA = \mathrm{grad}\,\varphi$, where $\varphi$ is the unique harmonic function such that $\varphi_n|_{z=-h}=0$ and $\varphi|_{z=\eta}=\Phi$,
so that $\bfgamma=\bfzero$ (because $(\Grad \phi)_\parallel = \nabla (\phi|_{z=\eta})$) and
\begin{align*}
H(\eta)(\bfzero,\Phi) &= \nabla \varphi \cdot \bfN|_{z=\eta} = G(\eta)\Phi,
\end{align*}
where $G(\eta)$ is the classical Dirichlet--Neumann operator. Furthermore, 
equations \eqref{eq:GZCS1}, \eqref{eq:GZCS2} reduce to
\begin{align*}
& G(\eta)\Phi +\bfc\cdot\nabla\eta=0, \\
& \tfrac{1}{2}|\nabla\Phi|^2 -\frac{(G(\eta)\Phi+\nabla\eta\cdot\nabla\Phi)^2}{2(1+|\nabla\eta|^2)} 
-\bfc\cdot\nabla \Phi
+g\eta- \beta\!\left(\frac{\eta_x}{(1+|\nabla\eta|^2)^\frac{1}{2}}\right)_{\!\!\!x}-\beta\!\left(\frac{\eta_z}{(1+|\nabla\eta|^2)^{1/2}\frac{1}{2}}\right)_{\!\!\!z}= 0,
\end{align*}
so that we recover the Zakharov--Craig--Sulem formulation of the steady water-wave problem (see Zakharov \cite{Zakharov68} and Craig and Sulem \cite{CraigSulem93}).
\end{remark}

We proceed by specialising to $\bfgamma=\bfzero$, writing $\bfc=\bfc_0+\bfmu$, where $\bfc_0=(c_{10},c_{20})^T$ is a reference wave velocity to be chosen later, so that
$$\bfu^\star = (c_{10}+\mu_1) \bfu^{(1)} +  (c_{20}+\mu_2) \bfu^{(2)},$$
and reducing equations \eqref{eq:GZCS1}, \eqref{eq:GZCS2} to a single equation for $\eta$ (see Oliveras and Vasan \cite{OliverasVasan13} for a derivation of the corresponding single-equation formulation for
irrotational water waves). Eliminating $\Phi$ from \eqref{eq:GZCS2} using \eqref{eq:GZCS1}, we find that
$$
J(\eta,\bfmu) \coloneqq   \frac{1}{2} |\bfT(\eta)|^2 - \frac{ ( -\underline{\bfu}^\star \cdot \bfN + \bfT(\eta) \cdot \nabla\eta)^2 }{2(1+|\nabla\eta|^2)}
+ \bfT(\eta) \cdot \underline{\bfu}_\mathrm{h}^\star + g\eta- \beta \left( \frac{\eta_x}{ (1+|\nabla\eta|^2)^{1/2} } \right)_{x} -\beta \left( \frac{\eta_y}{ (1+|\nabla\eta|^2)^{1/2} } \right)_{y}  = 0,
$$
where
$$
\bfT(\eta) \coloneqq  - \nabla \left( H(\eta)(\bfzero,\cdot)^{-1} (\underline{\bfu}^\star \cdot \bfN) \right) + \alpha \, \nablap \Delta^{-1} (\underline{\bfu}^\star \cdot \bfN)
$$
and the underscore denotes evaluation at $z=\eta$.

\begin{remark} \label{rem:symmetries}
Let $S_0$ be the reflection
\begin{equation*}
S_0\eta(\bfx)\coloneqq \eta(-\bfx),
\end{equation*}
and $T_{\bfv^\prime}$ be the translation
$$T_{\bfv^\prime}\eta(\bfx)\coloneqq \eta(\bfx+\bfv^\prime).$$
The mapping $ J$ is equivariant with respect to both $S_0$ and $T_{\bfv^\prime}$, that is
$$
  J(T_{\bfv^\prime}\eta,\bfmu)=T_{\bfv^\prime} J(\eta,\bfmu),\qquad
  J(S_0\eta,\bfmu)=S_0 J(\eta,\bfmu).
$$
\end{remark}

The operator $\bfT(\eta)$ can be defined more rigorously in terms of a boundary-value problem. Noting that\linebreak
$\underline{\bfu}^\star \cdot \bfN = \nablac \bfS(\eta)^\perp$, where
\begin{align}
\bfS(\eta) &\coloneqq  \frac{c_1}{\alpha} 
\begin{pmatrix}
\cos (\alpha \, \eta) -1 \\
-\sin (\alpha \, \eta)
\end{pmatrix}
+\frac{c_2}{\alpha} 
\begin{pmatrix}
\sin (\alpha \, \eta) \\
\cos (\alpha \, \eta) -1
\end{pmatrix},
\end{align}
we can define
$$
\bfT(\eta) \coloneqq  \bfM(\eta)(\bfzero, \bfS(\eta)), 
$$
where
$$
\bfM(\eta)(\bfgamma,\bfg) \coloneqq  -(\curl \bfB)_{\parallel}, \label{eq:MOpDef}
$$
and $\bfB$ solves the boundary-value problem
\begin{alignat}{2}
\curl \curl \bfB &= \alpha\curl \bfB & & \In D_\eta, \label{Intro B BVP 1} \\
\Div \bfB &= 0 & & \In D_\eta, \label{Intro B BVP 2} \\
\bfB \times \bfe_{3} &= \bfzero & & \at z=-h, \label{Intro B BVP 3} \\
\bfB \cdot \bfn &= 0 & & \at z=\eta, \label{Intro B BVP 4} \\
\nablac \bfB_{\parallel}^{\perp} &= \nablac \bfg^\perp & & \at z=\eta, \label{Intro B BVP 5} \\
\langle (\curl \bfB)_\parallel \rangle & = \bfgamma. \label{Intro B BVP 6}
\end{alignat}
Any solution to this boundary-value problem satisfies
$$(\curl \bfB)_\parallel = \bfgamma+\nabla \Phi - \alpha \nablap \Delta^{-1}(\nablac \bfB_\parallel^\perp)$$
 for some $\Phi$ (see equation \eqref{eq:compatibility}), so that $\Phi=H(\eta)(\bfgamma,\cdot)^{-1}\nablac \bfg^\perp$
and
$$-(\curl \bfB)_{\parallel}=-\bfgamma - \nabla (H(\eta)(\bfgamma,\cdot)^{-1}\nablac \bfg^\perp) + \alpha \nablap \Delta^{-1}(\nablac \bfg^\perp).$$

A rigorous treatment of the boundary-value problems \eqref{Intro A BVP 1}--\eqref{Intro A BVP 5} and
\eqref{Intro B BVP 1}--\eqref{Intro B BVP 6} is given in Section \ref{sec:WeakStrong} using a traditional weak/strong-solution approach.

\subsection{Analytical results for the operators $H$ and $\bfM$} \label{subsec:GenDNOp}

We write functions $f\colon {\mathbb R}^2/\Lambda \rightarrow {\mathbb R}$ as Fourier series
\begin{equation*}
f(\bfx)=\sum_{\bfk\in\Lambda^\prime}\hat{f}_{\bfk}\ee^{\ii\bfk\cdot\bfx},
\end{equation*}
where $\Lambda^\prime$ is the dual lattice to $\Lambda$; the Fourier coefficients $\hat{f}_\bfk$ are given by
\begin{align*}
\hat{f}_\bfk &= \frac{1}{|\Omega|} \int_\Omega f(\bfx) \ee^{-\ii\bfk \cdot \bfx} \; \mathrm{d}\bfx,
\end{align*}
where $\Omega$ is the parallelogram built with $\bflambda_1$, $\bflambda_2$. We write $\bfk=(k_1,k_2)^T$ and
work in the Sobolev spaces
$$H^s({\mathbb R}^2/\Lambda)
\coloneqq
\left\{f \in L^2({\mathbb R}^2/\Lambda)\colon
\| f \|_s^2 \coloneqq  \sum_{\bfk \in \Lambda^\prime} \left( 1 + |\bfk|^2 \right)^{s} |\hat f_\bfk|^2 < \infty\right\}, \qquad s \geq 0,
$$
and their subspaces
$$\mathring{H}^s({\mathbb R}^2/\Lambda)\coloneqq \{f \in H^s({\mathbb R}^2)\colon \hat{f}_{\bfzero}=0\}$$
of functions with zero mean, noting that the Hodge--Weyl decomposition \eqref{eq:HW decomp} of a function $\bff \in H^s({\mathbb R}^2/\Lambda)^2$ is given by
\begin{align*}
{\mathbb R}^2 \ni \bfgamma &= (\hat{f}_{1\bfzero},\hat{f}_{2\bfzero})^T, \\
 \mathring{H}^{s+1}({\mathbb R}^2/\Lambda)^2 \ni \Phi &= -\sum_{\bfk\in\Lambda^\prime \atop \bfk\neq \bfzero}
 \left(\frac{\ii k_1 \hat{f}_{1\bfk}+\ii k_2 \hat{f}_{2\bfk}}{|\bfk|^2}\right)\ee^{\ii\bfk\cdot\bfx}, \\
\mathring{H}^{s+1}({\mathbb R}^2/\Lambda)^2 \ni \Psi &= -\sum_{\bfk\in\Lambda^\prime \atop \bfk\neq \bfzero}
\left(\frac{\ii k_2 \hat{f}_{1\bfk} - \ii k_1 \hat{f}_{2\bfk}}{|\bfk|^2}\right)\ee^{\ii\bfk\cdot\bfx}.
\end{align*}

In Section \ref{Analflat} we show that the solutions to the boundary-value problems \eqref{Intro A BVP 1}--\eqref{Intro A BVP 5} and \eqref{Intro B BVP 1}--\eqref{Intro B BVP 6}
depend analytically upon $\eta$ and use this result to deduce that the same is true of $H(\eta)$ and $\bfM(\eta)$.
We proceed by `flattening' the fluid domain by means
of the transformation $\Sigma\colon D_0 \to D_\eta$ given by
$$
\Sigma\colon(\bfx,v) \mapsto (\bfx,v+\sigma(\bfx,v)), \qquad \sigma(\bfx,v)\coloneqq  \eta(\bfx)(1+v/h)
$$
which transforms the boundary-value problems for $\bfA$ and $\bfB$ into equivalent problems
for $\tilde{\bfA}\coloneqq \bfA \circ \Sigma$ and $\tilde{\bfB}\coloneqq  \bfB \circ \Sigma$ in the fixed domain $D_0$ (equations \eqref{Flattened A BVP 1}--\eqref{Flattened A BVP 5} and
\eqref{Flattened B BVP 1}--\eqref{Flattened B BVP 6} respectively).
The spatially extended version of the boundary-value problem for $\tilde{\bfA}$
was studied by Groves and Horn \cite[\S4]{GrovesHorn20} under the following non-resonance condition.
\begin{itemize}
\item[(NR)]
The restrictions
$$\begin{cases}
|\bfk| \neq |\alpha|, \\[2mm]
h \sqrt{\alpha^2-|\bfk|^2} \notin \frac{\pi}{2} \, \mathbb{N} , & \text{if} \; |\bfk| < |\alpha|, 
\end{cases}$$
hold for each $\bfk \in \Lambda^\prime$.
\end{itemize}

Their analysis in the present context leads to the first statement in the following theorem; the second is deduced from it.
Condition (NR) is a blanket hypothesis in Sections \ref{sec:diffNHOp}, \ref{sec:Taylor}, \ref{sec:asympexpNHOp} and \ref{sec:approx}, which rely upon these theorems.

\begin{theorem} \label{thm:analHM}
Suppose that $s \geq 2$, and assume that the non-resonance condition (NR) holds.
There exists an open neighbourhood $U$ of the origin in $H^{s+\frac{1}{2}}({\mathbb R}^2/\Lambda)$ such that
\begin{itemize}
\item[(i)]
the boundary-value problem \eqref{Flattened A BVP 1}--\eqref{Flattened A BVP 5}
has a unique solution $\tilde{\bfA}=\tilde{\bfA}(\eta,\bfgamma,\Phi)$ in $H^s(D_0/\Lambda)^3$ which depends analytically upon
$\eta \in U$, $\bfgamma \in {\mathbb R}^2$ and $\Phi \in \mathring{H}^{s-\frac{1}{2}}({\mathbb R}^2/\Lambda)$ (and linearly upon $(\bfgamma,\Phi)$);
\item[(ii)]
the boundary-value problem \eqref{Flattened B BVP 1}--\eqref{Flattened B BVP 6} has a unique solution
$\tilde{\bfB}=\tilde{\bfB}(\eta,\bfgamma,\bfg)$ in $H^s(D_0/\Lambda)^3$ which depends analytically upon
$\eta \in U$ and $\bfg \in H^{s-\frac{3}{2}}({\mathbb R}^2)^2$ (and linearly upon $(\bfgamma,\bfg)$).
\end{itemize}
\end{theorem}
The analyticity of $H$, $\bfM$ and $\bfT$ follows from Theorem \ref{thm:analHM} and the facts that
\begin{equation}
H(\eta)(\bfgamma,\Phi)=\nablac \tilde\bfA_\parallel^\perp,
\qquad
\bfM(\eta)(\bfgamma,\bfg)=-(\curl^\sigma \tilde\bfB)_\parallel,
\quad
\label{Flat ops}
\end{equation}
and $\bfT(\eta)=\bfM(\eta)(\bfzero,\bfS(\eta))$, where
$$
\curl^\sigma \tilde{\bfB}(\bfx,v) \coloneqq (\curl \, \bfB)\circ\Sigma(\bfx,v).
$$

\begin{theorem} \label{thm:analGDNO}
Suppose that $s \geq 2$, and assume that the non-resonance condition (NR) holds.
There exists an open neighbourhood $U$ of the origin in $H^{s+\frac{1}{2}}(\mathbb{R}^2/\Lambda)$ such that
$\eta \mapsto H(\eta)$, $\eta \mapsto \bfM(\eta)$ and $\eta \mapsto \bfT(\eta)$ are analytic mappings $U \to L({\mathbb R}^2 \times \mathring{H}^{s-\frac{1}{2}}({\mathbb R}^2/\Lambda),\mathring{H}^{s-\frac{3}{2}}({\mathbb R}^2/\Lambda))$,
$U \to L({\mathbb R}^2 \times H^{s-\frac{1}{2}}(\mathbb{R}^2/\Lambda)^2, H^{s-\frac{3}{2}}(\mathbb{R}^2/\Lambda)^2)$ and
$U \to H^{s-\frac{3}{2}}(\mathbb{R}^2/\Lambda)^2$ respectively.
\end{theorem}

In Section \ref{sec:diffNHOp} we turn to the differentials of $H(\eta)$ and $\bfM(\eta)$. Applying the operator $\mathrm{d}- \mathrm{d}\sigma \partial_v^\sigma$, where\linebreak $\partial_v^\sigma = (1+\partial_v\sigma)^{-1} \partial_v$,
to equations \eqref{Flat ops} shows that
\begin{align*}
\mathrm{d}H[\eta](\delta\eta)(\bfgamma,\Phi) &= \nablac \tilde{\bfC}_\parallel^\perp + \partial_v^\sigma \curl^\sigma \tilde{\bfA} \cdot \bfN|_{v=0}\delta\eta - (\curl^\sigma \tilde{\bfA})_\mathrm{h}\cdot\nabla \delta\eta, \\
\mathrm{d}\bfM[\eta](\delta\eta)(\bfgamma,\bfg) &= -(\curl^\sigma \tilde{\bfD})_\parallel -\delta\eta(\partial_v^\sigma \curl^\sigma \tilde{\bfB})_\parallel - (\curl^\sigma \tilde{\bfB})_3|_{v=0}\nabla\delta\eta,
\end{align*}
where $\tilde\bfC=(\mathrm{d} \tilde\bfA -\mathrm{d}\sigma \partial_v^\sigma\tilde\bfA)$ and $\tilde\bfD=(\mathrm{d} \tilde\bfB-\mathrm{d}\sigma  \partial_v^\sigma\tilde\bfB)$. Careful inspection of the boundary-value problems
for $\tilde\bfC$ and $\tilde\bfD$ (which are obtained by applying $\mathrm{d}- \mathrm{d}\sigma \partial_v^\sigma$ to the boundary-value problems for $\tilde\bfA$ and $\tilde{\bfB}$) yields the following result. Note the increased regularity requirement
due to the double application of $H(\eta)$ and $\bfM(\eta)$ in the formulae.

\begin{theorem} \label{thm:diffs}
Suppose that $s\geq3$.
\begin{itemize}
\item[(i)]
The differential of the operator
$H(\cdot)\colon U \to L({\mathbb R}^2 \times \mathring{H}^{s-\frac{1}{2}}({\mathbb R}^2/\Lambda),\mathring{H}^{s-\frac{3}{2}}({\mathbb R}^2/\Lambda))$ is given by
\begin{align*}
\mathrm{d}H&[\eta](\delta\eta)(\bfgamma,\Phi)\\
&=
H(\eta)\left(-\alpha\langle(\bfK(\eta)(\bfgamma,\Phi)-u\nabla\eta)^\perp\delta\eta\rangle,-\alpha\Delta^{-1}\nablac((\bfK(\eta)(\bfgamma,\Phi)-u\nabla\eta)^\perp\delta\eta)-u\delta\eta + \langle u \delta\eta\rangle\right)\\
& \qquad\quad\mbox{}
-\nablac((\bfK(\eta)(\bfgamma,\Phi)-u\nabla\eta)\delta\eta),
\end{align*}
where
$$u=\frac{\bfK(\eta)(\bfgamma,\Phi)\cdot \nabla \eta+H(\eta)(\bfgamma,\Phi)}{1+|\nabla \eta|^2}.$$
\item[(ii)]
The differential of the operator
$\bfM(\cdot)\colon U \to L({\mathbb R}^2 \times H^{s-\frac{1}{2}}(\mathbb{R}^2/\Lambda)^2, H^{s-\frac{3}{2}}(\mathbb{R}^2/\Lambda)^2)$ is given by
\begin{align*}
\mathrm{d}\bfM&[\eta](\delta\eta)(\bfgamma,\bfg) \\
&=
\bfM(\eta)\left(\alpha\langle(\bfM(\eta)(\bfgamma,\bfg)+u\nabla\eta)^\perp\delta\eta\rangle, (\bfM(\eta)(\bfgamma,\bfg)+u\nabla\eta)^\perp\delta\eta\right)-\nabla(u\delta\eta)
+\alpha(\bfM(\eta)(\bfgamma,\bfg)+u\delta\eta)^\perp\delta\eta,
\end{align*}
where
$$u = \frac{ \nablac \bfg^\perp - \bfM(\eta)(\bfgamma,\bfg) \cdot \nabla\eta }{1+|\nabla\eta|^2}.$$
\end{itemize}
\end{theorem}

In Section \ref{sec:Taylor} we show how to use recursion formulae to
compute the terms in the Taylor expansions
\begin{align} \label{eq:TaylorHMeta}
H(\eta) = \sum_{j=0}^\infty H_j(\eta), \qquad \bfM(\eta) = \sum_{j=0}^\infty \bfM_j(\eta)
\end{align}
of $H(\eta)$ and $\bfM(\eta)$ at $\eta=0$ systematically, 
where $H_j(\eta)$ and $\bfM_j(\eta)$ are homogeneous of degree $j$ in $\eta$
(compare with the recursion formulae for the Taylor expansion of the Dirichlet--Neumann
operator appearing in the irrotational case given by Craig and Sulem \cite{CraigSulem93}). The recursion formulae are derived by
substituting the expansions \eqref{eq:TaylorHMeta} into the expressions for $\mathrm{d}H[\eta](\eta)(\bfgamma,\Phi)$
and $\mathrm{d}\bfM[\eta](\eta)(\bfgamma,\bfg)$ given by Theorem \ref{thm:diffs}, and equating terms of equal homogeneity in $\eta$.
The individual terms in the series are computed as functions of $H_0$ and $\bfM_0$
using the recursion formulae, and straightforward calculations using Fourier series show that
$$H_0(\bfgamma,\Phi) = D^2 \, \mathtt{t}(D) \, \Phi, \qquad
\bfM_0(\bfgamma,\bfg) = -\bfgamma + \frac{1}{D^2} \, \left( \alpha \, \bfD^{\perp} + \bfD \, \mathtt{c}(D) \right) \, \bfD \cdot \bfg^{\perp},$$
where
$$
\mathtt{c}(|\bfk|) \coloneqq  
\begin{cases}
\sqrt{\alpha^2-|\bfk|^2} \, \cot(h\sqrt{\alpha^2-|\bfk|^2}),   & \mbox{if $|\bfk| < |\alpha|$,} \\
\sqrt{|\bfk|^2-\alpha^2} \, \coth(h\sqrt{|\bfk|^2-\alpha^2}), & \mbox{if $|\bfk| > |\alpha|$,}
\end{cases}
\qquad
\mathtt{t}(|\bfk|) \coloneqq  
\begin{cases}
\frac{ \tan(h \sqrt{\alpha^2-|\bfk|^2}) }{\sqrt{\alpha^2-|\bfk|^2}},   & \mbox{if $|\bfk| < |\alpha|$,} \\
\frac{ \tanh(h \sqrt{|\bfk|^2-\alpha^2}) }{\sqrt{|\bfk|^2-\alpha^2}},  & \mbox{if $|\bfk| > |\alpha|$.}
\end{cases}$$
and
$$\bfD=(D_1,D_2)^T=-\ii\nabla, \qquad D=|\bfD|.$$
Explicit formulae for $H_0$, $H_1$, $H_2$ and $\bfM_0$, $\bfM_1$, $\bfM_2$ are  are computed in Section \ref{sec:Taylor}.

\begin{remark}
This method leads to formulae involving ever more derivatives of $\eta$ in the individual terms in the formulae for
$H_j(\eta)$ and $\bfM_j(\eta)$; the overall validity of the formulae arises from subtle cancellations between the terms
(see Nicholls and Reitich \cite[\S2.2]{NichollsReitich01a} for a discussion of this phenomenon in the context of the classical Dirichlet--Neumann operator). 
\end{remark}

\subsection{Pseudodifferential calculus for the operators $H$ and $\bfM$}

In Section \ref{sec:asympexpNHOp} we fix $\eta \in C^\infty({\mathbb R}^2/\Lambda)$, prove that $H(\eta)(\bfzero,\cdot)$ and $\bfM(\eta)(\bfzero,\cdot)$ are smooth perturbations of properly supported pseudodifferential
operators, and compute their asymptotic expansions.

Following Alazard, Burq and Zuily \cite{AlazardBurqZuily11}, we begin by introducing a localising transform (which differs from the flattening transform
used in Section \ref{sec:operators}). Choose $\delta>0$ so that the fluid domain $D_\eta$ contains the strip
\begin{align*}
	\Omega_\delta&\coloneqq  \{ (\bfx,z) \in \mathbb{R}^2 \times \mathbb{R}\colon \eta(\bfx)-\delta  h \leq z < \eta(\bfx) \}
\end{align*}
for $\eta \in U$ and define $\hat\Sigma\colon D_0 \to \Omega_\delta$ by
$$
	\hat\Sigma\colon (\bfx,w) \mapsto (\bfx,\varrho(\bfx,w)), \qquad \varrho(\bfx,w)\coloneqq \delta  w + \eta(\bfx).
$$
This transform converts the equation
$$-\Delta \bfU = \alpha  \curl \bfU \In \Omega_\delta$$
into
$$
-\Delta^\varrho \hat{\bfU} - \alpha  \curl^\varrho \hat{\bfU} = \bfzero \In D_0,
$$
where
$$\curl^\varrho \hat{\bfU}(\bfx,w) =(\curl \, \bfU)\circ\hat\Sigma(\bfx,w), \qquad \Delta^\varrho \hat{\bfU} = (\Delta \, \bfU) \circ \hat\Sigma(\bfx,w),$$
which we write as
\begin{equation}
L\hat{\bfU} =\bfzero \label{eq:partiallyflattened}
\end{equation}
(the explicit formula for $L$ is given in Section \ref{subsec:Flatfac}). We proceed by implementing Treves's factorisation method
(Treves \cite[Ch.\ III, \S3]{Treves}) and examining its consequences for solutions of equation \eqref{eq:partiallyflattened}.

\begin{lemma} 
There are properly supported operators $M$, $N \in \Psi^1({\mathbb R}^2/\Lambda)$ such that
\begin{itemize}
\item[(i)]
$L-a(\partial_w I - N)(\partial_w I -M ) \in \Psi^{-\infty}({\mathbb R}^2/\Lambda)$, where $a=(1+|\nabla\eta|^2)/\delta^2$,
\item[(ii)]
the principal symbols $\mathtt{M}^{(1)}$, $\mathtt{N}^{(1)}$ of $M$, $N$ take the form $\mathtt{M}^{(1)} = \mathtt{m}^{(1)}{\mathbb I}_3$,
$\mathtt{N}^{(1)} = \mathtt{n}^{(1)}{\mathbb I}_3$, where the scalar-valued symbols $\mathtt{m}^{(1)}$, $-\mathtt{n}^{(1)} \in S^1({\mathbb R}^2/\Lambda)$
are strongly elliptic.
\end{itemize}
\end{lemma}
\begin{lemma} \label{lem:regtobd intro}
Any function $\hat{\bfU} \in H^2(D_0 / \Lambda)^3$ with $L\hat{\bfU}=0$ in $D_0$ satisfies
$$\partial_w \hat{\bfU} = M\hat{\bfU} + R_\infty \hat{\bfU} \at w=0,$$
where the symbol $R_\infty$ denotes a linear function of its argument whose range lies in $C^\infty({\mathbb R^2}/\Lambda)^3$.
\end{lemma}

Let $s \geq 2$, $\Phi \in \mathring{H}^{s-\frac{1}{2}}({\mathbb R}^2 / \Lambda)$ and $\tilde{\bfA} \in H^s(D_0/\Lambda)^3$ be
the function defining $H(\eta)(\bfzero,\Phi)$ (see equation \eqref{Flat ops}). The variable
$$
\hat{\bfA}(\bfx, w) \coloneqq \tilde{\bfA}(\bfx,v),\qquad w \coloneqq \frac{1}{\delta h}(h+\eta)v$$
satisfies \eqref{eq:partiallyflattened}
and hence
\begin{equation}
\partial_w \hat{\bfA}|_{w=0} = M\hat{\bfA}|_{w=0} + R_\infty\Phi \label{eq:partial flat 1}
\end{equation}
(see Lemma \ref{lem:regtobd intro}, noting that $\hat{\bfA}$ is a linear function of $\Phi$), together with
\begin{alignat}{2}
\hat{A}_3 &=\eta_x \hat{A}_1+\eta_y \hat{A}_2 & & \at  w=0, \label{eq:partial flat 2} \\
(\text{curl}^\varrho\hat{\bfA})_\parallel&=\nabla\Phi-\alpha\nablap\Delta^{-1}(\nablac\hat{\bfA}_\parallel^\perp) & & \at  w=0. \label{eq:partial flat 3}
\end{alignat}
Eliminating $\partial_w\hat{\bfA}$ using \eqref{eq:partial flat 1}, we find from \eqref{eq:partial flat 2}, \eqref{eq:partial flat 3} that
\begin{equation}
\hat{\bfA}|_{w=0} = Z \Phi + R_\infty\Phi, \label{eq:partial flat 4}
\end{equation}
where $\mathtt{Z} \in S^0({\mathbb R}^2 / \Lambda)$ and $Z=\Op\mathtt{Z}$.
Finally, inserting $\hat{\bfA}|_{w=0}$ and $\partial_w\hat{\bfA}|_{w=0}$ from \eqref{eq:partial flat 1}, \eqref{eq:partial flat 3}
into
$$
H(\eta)(\bfzero,\Phi) = \hat{A}_{2x}+\eta_y\hat{A}_{3x}-\hat{A}_{1y}-\eta_x\hat{A}_{3y}\big\vert_{w=0}
$$
shows that
$$H(\eta)(\bfzero,\Phi)=\Op\lambda_{\alpha}\Phi+R_\infty\Phi,$$
where $\lambda_{\alpha}\in S^1({\mathbb R}^2/\Lambda)$. The asymptotic expansions
$$\mathtt{Z} \sim \sum_{j \leq 0} \mathtt{Z}^{(j)}, \qquad \lambda_\alpha \sim\sum_{j \leq 1} \lambda_\alpha^{(j)}$$
can be determined recursively by substituting
$$
\hat{\bfA}|_{w=0} = Z \Phi + R_\infty\Phi, \qquad \partial_w\hat{\bfA}|_{w=0} = MZ\Phi + R_\infty\Phi$$
into
\eqref{eq:partial flat 2}, \eqref{eq:partial flat 3}. These calculations are performed in Section \ref{subsec:asympH} and summarised in the following theorem.

\begin{theorem} \label{thm:ExpHOp}
Suppose $s\geq2$, $\eta \in C^\infty({\mathbb R}^2/\Lambda)$ and $\Phi \in  \mathring{H}^{s-1/2}({\mathbb R}^2/\Lambda)$.
We have that 
$$
H(\eta)(\bfgamma,\Phi)=H(\eta)(\bfgamma,0)+H(\eta)(\bfzero,\Phi),
$$
the first term of which belongs to $C^\infty({\mathbb R}^2/\Lambda)$, and
$$H(\eta)(\bfzero,\Phi)=\Op\lambda_{\alpha}\Phi+R_\infty\Phi,$$
where $\lambda_{\alpha} \in S^1({\mathbb R}^2/\Lambda)$
and $R_\infty \Phi \in C^\infty({\mathbb R}^2/\Lambda)$. The symbol $\lambda_\alpha$ admits the asymptotic expansion
\begin{align*}
\lambda_\alpha &\sim \lambda^{(1)}_{\alpha} + \lambda^{(0)}_{\alpha} + \cdots,
\end{align*}
in which $\lambda^{(j)}_{\alpha}(\bfx,\bfk)$ is homogeneous of degree $j$ in $\bfk$. Moreover
\begin{align*}
\lambda^{(1)}_{\alpha}(\bfx,\bfk) &= \lambda^{(1)}(\bfx,\bfk) , \\
\lambda^{(0)}_{\alpha}(\bfx,\bfk) &\coloneqq  \lambda^{(0)}(\bfx,\bfk) + \alpha\frac{(\bfk\cdot\nabla\eta)(\bfk\cdot\nablap\eta)}{|\bfk|^2},
\end{align*}
where
\begin{align*}
 \lambda^{(1)}(\bfx,\bfk) &\coloneqq  \sqrt{ (1+|\nabla\eta|^2) |\bfk|^2 - (\bfk \cdot \nabla\eta)^2 },\\
\lambda^{(0)}(\bfx,\bfk) &\coloneqq  \frac{1+|\nabla\eta|^2}{2\lambda^{(1)}} \left(\nablac (\mathtt{m}^{(1)} \, \nabla\eta )+ \ii \nabla_{\bfk} \lambda^{(1)}\cdot \nabla \mathtt{m}^{(1)}  \right), \qquad \mathtt{m}^{(1)}(\bfx,\bfk) \coloneqq  \frac{ \ii  \bfk \cdot \nabla\eta + \lambda^{(1)} }{ 1+|\nabla\eta|^2 }, 
\end{align*}
are the principal and sub-principal symbols of the classical Dirichlet--Neumann operator.
\end{theorem}

The corresponding result for $M(\eta)$ is obtained in a similar fashion in Section \ref{subsec:asympM}.

\begin{theorem} \label{thm:ExpMOp}
Suppose $s\geq2$, $\eta \in C^\infty({\mathbb R}^2/\Lambda)$ and $\bfg \in H^{s-1/2}(\mathbb{R}^2/\Lambda)^2$.
We have that 
$$
\bfM(\eta)(\bfgamma,\bfg)=\bfM(\eta)(\bfgamma,\bfzero)+\bfM(\eta)(\bfzero,\bfg),
$$
the first term of which belongs to $C^\infty({\mathbb R}^2/\Lambda)^2$, and
$$\bfM(\eta)(\bfzero,\bfg)=\Op\nu_{\alpha}\bfg+R_\infty\bfg,$$
where $\nu_{\alpha} \in S^1({\mathbb R}^2/\Lambda)$
and $R_\infty \bfg \in C^\infty({\mathbb R}^2/\Lambda)^2$. The symbol $\nu_\alpha$ admits the asymptotic expansion
\begin{align*}
\nu_\alpha &\sim \nu^{(1)}_{\alpha} + \nu^{(0)}_{\alpha} + \cdots,
\end{align*}
in which $\nu^{(j)}_{\alpha}(\bfx,\bfk)$ is homogeneous of degree $j$ in $\bfk$.
Moreover
$$
\nu_\alpha^{(1)}(\bfx,\bfk)\bfg=\bfk\frac{(\bfk\cdot\bfg^\perp)}{\lambda^{(1)}}, \qquad
\nu_{\alpha}^{(0)}(\bfx,\bfk)\bfg=\begin{pmatrix}
\zeta_1(\bfx,\bfk)\\
\zeta_2(\bfx,\bfk)
\end{pmatrix}(\bfk\cdot\bfg^\perp),
$$
where 
\begin{align*}
\zeta_1(\bfx,\bfk)&=\frac{\ii}{2(\lambda^{(1)})^5}\bigg( k_1^2(-1+2\eta_y^2)\eta_x- k_1 k_2\eta_y(3+4\eta_x^2)+2 k_2^2\eta_x(1+\eta_x^2)+\ii k_1\lambda^{(1)}\bigg)\\
&\qquad \times \bigg( k_1^2\eta_{yy}-2 k_1 k_2\eta_{xy}+ k_2^2\eta_{xx}\bigg)+\frac{\alpha}{(\lambda^{(1)})^2}\left( k_2(1+\eta_x^2)- k_1\eta_x\eta_y\right),\\
\zeta_2(\bfx,\bfk)&=\frac{\ii}{2(\lambda^{(1)})^5}\bigg(2 k_1^2\eta_y(1+\eta_y^2)- k_1 k_2\eta_x(3+4\eta_y^2)+ k_2^2\eta_y(-1+2\eta_x^2)+\ii k_2\lambda^{(1)}\bigg)\\
&\qquad \times \bigg( k_1^2\eta_{yy}-2 k_1 k_2\eta_{xy}+ k_2^2\eta_{xx}\bigg)+\frac{\alpha}{(\lambda^{(1)})^2}\left(- k_1(1+\eta_y^2)+ k_2\eta_x\eta_y\right).
\end{align*}
\end{theorem}

\subsection{Construction of approximate solutions} \label{subsec:ApprSol}
In Section \ref{sec:approx} we construct approximate solutions of
\begin{equation}
J(\eta,\bfmu)=0 \label{eq:hydroeq with mu - intro}
\end{equation}
for $\beta\geq 0$ in the form of power series and moreover prove their convergence for
$\beta>0$. The solutions have wave velocity $\bfc$ close to a reference value $\bfc_0$ chosen such that the following transversality condition holds;
we refer to Lokharu, Seth and Wahl\'{e}n \cite{LokharuSethWahlen20} for a detailed geometrical investigation of this condition (see in particular condition (3.7) and Proposition 3.3 in that
reference).
\begin{itemize}
\item[(T)] The only solutions $\bfk \in \Lambda^\prime$ of the \emph{dispersion relation}
\begin{align*} 
	\rho(\bfk,\bfc,\beta) &\coloneqq  \left[ g + \beta \; |\bfk|^2 - \frac{\alpha}{ |\bfk|^2 } \; (\bfc \cdot \bfk)(\bfk^\perp \cdot \bfc) \right] |\bfk|^2\mathtt{t}(|\bfk|) -  \; (\bfc \cdot \bfk)^2=0
\end{align*}
 are $\bfk=\bfzero$, $\pm \bfk_1$, $\pm \bfk_2$, where $\bfk_1$ and $\bfk_2$ are the generators of the lattice $\Lambda^\prime$. Furthermore,
the vectors $\nabla_{\!\bfc}\,\rho(\bfk_1,\bfc_0,\beta)$ and $\nabla_{\!\bfc}\,\rho(\bfk_2,\bfc_0,\beta)$ are linearly independent.
\end{itemize}

We consider $ J$ as a locally analytic mapping $X_s^\beta \times {\mathbb R}^2 \rightarrow H^s({\mathbb R}^2/\Lambda)$ for a sufficiently large value of $s$, where
\begin{align*}
 X_s^\beta&\coloneqq\begin{cases}
H^{s+2}(\mathbb{R}^2/\Lambda), &  \mbox{if $\beta>0$,}\\
H^{s+1}(\mathbb{R}^2/\Lambda),  &  \mbox{if $\beta=0$,}
\end{cases}
\end{align*}
and proceed to investigate the kernel and range of 
$$J_{10}(\eta)\coloneqq \mathrm{d}_1J[0,\bfzero](\eta)=\bfT_1(\eta)\cdot\bfc_0+g\eta-\beta\Delta\eta.$$
Writing
\begin{equation*}
\eta(\bfx)=\sum_{\bfk\in\Lambda^\prime}\hat{\eta}_{\bfk}\ee^{\ii\bfk\cdot\bfx},
\end{equation*}
so that
$$(J_{10}\eta)(\bfx) =g\hat\eta_{\bfzero}+\sum_{\bfk\in\Lambda^\prime\setminus\{\bfzero\}}\frac{\mathtt{c}(|\bfk|)}{|\bfk|^2}\rho(\bfk,\bfc_0,\beta)\hat{\eta}_{\bfk}\ee^{\ii\bfk\cdot\bfx},$$
we find that
\begin{equation*}
\ker(J_{10})=\{A\ee^{\ii\bfk_1\cdot\bfx}+B\ee^{\ii\bfk_2\cdot\bfx}+\bar{A}\ee^{-\ii\bfk_1\cdot\bfx}+\bar{B}\ee^{-\ii\bfk_2\cdot\bfx}\colon A, B \in {\mathbb C}\},
\end{equation*}
since $\rho(\bfk,\bfc_0,\beta)=0$ if and only if $\bfk=\bfzero,\pm\bfk_1,\pm\bfk_2$. The operator $J_{10}$ is formally invertible
if $\hat{f}_{\pm\bfk_1}=\hat{f}_{\pm\bfk_2}=0$ with formal inverse given by
$$(J_{10}^{-1} f)(\bfx)=\frac{1}{g}\hat{f}_{\bfzero}+\hspace{-0.5cm}\sum_{\bfk\in\Lambda^\prime \atop \bfk\neq \bfzero,\pm\bfk_1,\pm\bfk_2}\hspace{-5mm}
\frac{|\bfk|^2}{\mathtt{c}(|\bfk|)\rho(\bfk,\bfc_0,\beta)}\hat{f}_{\bfk}
\ee^{\ii\bfk\cdot\bfx}.
$$
For $\beta>0$ we find that $\rho(\bfk,\bfc_0,\beta)\gtrsim |\bfk|^3$ for sufficiently large $|\bfk|$, so that the above series
converges in $H^{s+2}(\mathbb{R}^2/\Lambda)$ for $f \in H^{s+2}(\mathbb{R}^2/\Lambda)$; it follows that
$J_{10}\colon H^{s+2}(\mathbb{R}^2/\Lambda) \rightarrow H^s(\mathbb{R}^2/\Lambda) $ is Fredholm with index $0$.
In contrast $\rho(\bfk,\bfc_0,0)$ is not bounded from below as $|\bfk|\rightarrow \infty$, so that the above formula does not define a bounded operator from $H^s(\mathbb{R}^2/\Lambda)$ to $H^{s+1}(\mathbb{R}^2/\Lambda)$
for any $s$. We therefore proceed formally, noting that the procedure is rigorously valid for $\beta>0$.

To apply the Lyapunov--Schmidt reduction to equation \eqref{eq:hydroeq with mu - intro} we write $\eta=\eta_1+\eta_2$, where
\begin{equation*}
\eta_1=A\ee^{\ii\bfk_1\cdot\bfx}+B\ee^{\ii\bfk_2\cdot\bfx}+\bar{A}\ee^{-\ii\bfk_1\cdot\bfx}+\bar{B}\ee^{-\ii\bfk_2\cdot\bfx}
\end{equation*}
and $\eta_2 \in \ker(J_{10})^\perp$.
Noting that $S_0$ and $T_{\bfv^\prime}$ act on the coordinates $(A,B,\bar{A},\bar{B})$ as 
$$S_0(A,B,\bar{A},\bar{B})=(\bar{A},\bar{B},A,B), \qquad T_{\bfv^\prime}(A,B,\bar{A},\bar{B})=(A\ee^{\ii\bfk_1\cdot\bfv^\prime},B\ee^{\ii\bfk_2\cdot\bfv^\prime},\bar{A}\ee^{-\ii\bfk_1\cdot\bfv^\prime},\bar{B}\ee^{-\ii\bfk_2\cdot\bfv^\prime}),$$
and that the reduced equation remains equivariant under these symmetries, we show that \eqref{eq:hydroeq with mu - intro} is locally equivalent to
\begin{align*}
Ag_1(|A|^2,|B|^2,\bfmu)=0, \\
Bg_2(|A|^2,|B|^2,\bfmu)=0,
\end{align*}
where $g_1,g_2$ are real-valued locally analytic functions which vanish at the origin. The following result is obtained from the analytic implicit-function theorem and the
transversality condition (T).

\begin{proposition}
There exist $\varepsilon>0$ and analytic functions $\mu_i\colon B_\varepsilon ({\bf 0},\mathbb{R}^2)\rightarrow {\mathbb R}$, $i=1,2$ such that $\mu_i(0,0)=0$ and $(|A|^2,|B|^2,\mu_1(|A|^2,|B|^2),\mu_2(|A|^2,|B|^2))$ is the unique local solution of $g_i(|A|^2,|B|^2,\bfmu)=0$, $i=1,2$.
\end{proposition}

Our main result now follows by substituting $\bfmu = \bfmu(|A|^2,|B|^2)$ into $\eta=\eta_1+\eta_2(\eta_1,\bfmu)$.

\begin{theorem}
Suppose that $\beta>0$.
There exist $\varepsilon>0$, a neighbourhood $V$ of the origin in $X_s^\beta \times {\mathbb R}^2$ and analytic functions
$\mu_1,\mu_2 \colon B_\varepsilon ({\bf 0},{\mathbb R}^2)\rightarrow {\mathbb R}$ and $\eta \colon B_\varepsilon ({\bf 0},{\mathbb C}^4) \rightarrow X_s^\beta$ such that
$$\{(\eta,\bfmu) \in X_s^\beta \times {\mathbb R}^2 \colon  J(\eta,\bfmu)=0,\ \eta \neq 0\} \cap V = \{ (\eta(A,B,\bar{A},\bar{B}),\bfmu(|A|^2,|B|^2)) \colon (A,B,\bar{A},\bar{B}) \in B_\varepsilon^\prime ({\bf 0},{\mathbb C}^4)\};$$
furthermore $\bfmu(0,0)=\bfzero$ and
$$\eta(\bfx)=A\ee^{\ii\bfk_1\cdot\bfx}+B\ee^{\ii\bfk_2\cdot\bfx}+\bar{A}\ee^{-\ii\bfk_1\cdot\bfx}+\bar{B}\ee^{-\ii\bfk_2\cdot\bfx}
+O(|(A,B,\bar{A},\bar{B})|^2).$$
\end{theorem}

The terms in the expansions
$$\eta=A\ee^{\ii\bfk_1\cdot\bfx}+B\ee^{\ii\bfk_2\cdot\bfx}+\bar{A}\ee^{-\ii\bfk_1\cdot\bfx}+\bar{B}\ee^{-\ii\bfk_2\cdot\bfx}
+\hspace{-4mm}\sum_{i+j+k+l \geq 2}\hspace{-2mm}\eta_{ijkl}A^iB^j\bar{A}^k\bar{B}^l
$$
and
$$\mu_i = \sum_{j+k \geq 1} \mu_{i,jk} |A|^{2j} |B|^{2k}, \qquad i=1,2,$$
can be determined recursively by substituting these expressions into \eqref{eq:hydroeq with mu - intro}
and equating monomials in $(A,B,\bar{A},\bar{B})$. Note that the series can be computed to any order for $\beta \geq 0$
but their convergence has been established only for $\beta>0$.
The coefficients $\eta_{ijkl}$ for $i+j+k+\ell=2$ and $\mu_{1,jk}$, $\mu_{2,jk}$ for $j+k=1$ are computed in Section \ref{sec:approx}.

\section{The operators $H(\eta)$ and $\bfM(\eta)$} \label{sec:operators}
In this section we study the operators $H(\eta)$ and $\bfM(\eta)$ defined by
\begin{equation}
H(\eta)(\bfgamma,\Phi) \coloneqq  \nablac\bfA_\parallel^\perp, \qquad
\bfM(\eta)(\bfgamma,\bfg) \coloneqq  -(\curl \bfB)_{\parallel},\label{eq:alt defn of H and M}
\end{equation}
where $\bfA$ and $\bfB$ are the solutions to the boundary-value problems
\begin{alignat}{2}
\curl \curl \bfA &= \alpha  \curl \bfA & & \In D_\eta, \label{A BVP 1} \\
\Div \bfA &= 0 & & \In D_\eta, \label{A BVP 2} \\
\bfA \times \bfe_3 &= \bfzero & & \at z=-h, \label{A BVP 3} \\
\bfA \cdot \bfn &= 0 & & \at z=\eta, \label{A BVP 4} \\
(\curl \bfA)_\parallel &= \bfgamma+\nabla\Phi - \alpha \nablap\Delta^{-1}(\nablac \bfA_\parallel^\perp) & & \at z=\eta \label{A BVP 5}
\end{alignat}
and
\begin{alignat}{2}
\curl \curl \bfB &= \alpha  \curl \bfB & & \In D_\eta, \label{B BVP 1}\\
\Div \bfB &= 0 & & \In D_\eta, \label{B BVP 2}\\
\bfB \times \bfe_{3} &= \bfzero & & \at z=-h, \label{B BVP 3}\\
\bfB \cdot \bfn &= 0 & & \at z=\eta, \label{B BVP 4}\\
\nablac \bfB_{\parallel}^{\perp} &= \nablac \bfg^\perp, \label{B BVP 5} \\
\langle (\curl \bfB)_\parallel \rangle &= \bfgamma. \label{B BVP 6}
\end{alignat}

\subsection{Weak and strong solutions} \label{sec:WeakStrong}
We first suppose that $\eta$ is a fixed function in $W^{2,\infty}({\mathbb R}^2/\Lambda)$ with $\inf \eta > -h$ and present a
traditional weak/strong-solution approach to the
boundary-value problems \eqref{A BVP 1}--\eqref{A BVP 5} and \eqref{B BVP 1}--\eqref{B BVP 6}, working with the standard 
spaces ${\mathscr D}(D_\eta/\Lambda)^3$ and ${\mathscr D}(\overline{D_\eta/\Lambda})^3$ of periodic test functions, the
Sobolev spaces $L^2(D_\eta/\Lambda)^3$ and $H^1(D_\eta/\Lambda)^3$, and the closed subspace
\begin{align*}
{\mathcal X}_\eta & = \{\bfF \in H^1(D_\eta/\Lambda)^3\colon \bfF \times \bfe_3 |_{z=-h}= \mathbf{0},\, \bfF\cdot\bfn|_{z=\eta} = 0\}
\end{align*}
of $H^1(D_\eta/\Lambda)^3$.

\begin{definition}$ $
\begin{itemize}
\item[(i)]
A \underline{weak solution} of \eqref{A BVP 1}--\eqref{A BVP 5} is a function
$\bfA \in {\mathcal X}_\eta$ such that
\begin{equation}
\int_\Omega \int_{-h}^\eta (\curl \bfA\cdot\curl \bfC -\alpha \curl \bfA \cdot\bfC + \Div \bfA\, \Div \bfC)
-\alpha\int_\Omega \nabla \Delta^{-1}( \nabla \cdot \bfA^{\!\perp}_\parallel) \cdot \bfC_\parallel
=
\int_\Omega (\bfgamma^\perp+\nablap \Phi)\cdot\bfC_\parallel
\label{Weak solution A}
\end{equation}
for all $\bfC \in {\mathcal X}_\eta$,
while a \underline{strong solution} has the additional regularity requirement that $\bfA \in H^2(D_\eta/\Lambda)^3$, is solenoidal
and satisfies
\eqref{A BVP 1} in $L^2(D_\eta/\Lambda)^3$ and \eqref{A BVP 5} in $H^\frac{1}{2}({\mathbb R}^2/\Lambda)^2$.
\item[(ii)]
A \underline{weak solution} of \eqref{B BVP 1}--\eqref{B BVP 6} is a function
$\bfB \in {\mathcal X}_\eta$ which satisfies \eqref{B BVP 5} and
\begin{equation}
\int_\Omega \int_{-h}^\eta (\curl \bfB\cdot\curl \bfD -\alpha \curl \bfB \cdot\bfD + \Div \bfB\, \Div \bfD)
=\int_\Omega (\bfgamma^\perp+\alpha\nabla \Delta^{-1}( \nablac \bfg^\perp) )\cdot \bfD_\parallel
\label{Weak solution B}
\end{equation}
for all $\bfD \in {\mathcal X}_\eta^0$, where
$${\mathcal X}_\eta^0\coloneqq \{\bfF \in {\mathcal X}_\eta\colon \nablac \bfF_{\parallel}^{\perp} =0\},$$
while a \underline{strong solution} has the additional regularity requirement that $\bfB$ lies in $H^2(D_\eta/\Lambda)^3$, is solenoidal,
satisfies \eqref{B BVP 6} and satisfies
\eqref{B BVP 1} in $L^2(D_\eta/\Lambda)^3$.
\end{itemize}
\end{definition}
 The existence of weak and strong solutions is established in Lemmata \ref{weak exist} and \ref{strong exist} below,
whose proofs rely upon the following technical results (see Groves and Horn \cite[\S4(b)]{GrovesHorn20}).

\begin{proposition} \label{About spaces}
$ $
\begin{itemize}
\item[(i)]
The space ${\mathcal X}_\eta$ coincides with
\begin{align*}
& \{\bfF \in L^2(D_\eta/\Lambda)^3\colon \curl \bfF\in L^2(D_\eta/\Lambda)^3,\, \Div \bfF \in L^2(D_\eta/\Lambda),\,  \bfF \times \bfe_3 |_{z=-h}= \mathbf{0},\, \bfF\cdot\bfn|_{z=\eta}  = 0\}
\end{align*}
and the function
$\bfF \mapsto (\|\curl \bfF\|_{L^2(D_\eta/\Lambda)^3}^2+\|\Div \bfF\|_{L^2(D_\eta/\Lambda)}^2)^\frac{1}{2}$ is equivalent to its usual norm.

\item[(ii)]
The spaces
\begin{align*}
& \{\bfF \!\in\! L^2(D_\eta/\Lambda)^3\colon \curl \bfF\!\in\! L^2(D_\eta/\Lambda)^3,\, \Div \bfF \!\in\! L^2(D_\eta/\Lambda),\, \bfF\times \bfe_3|_{z=-h} \!\in\! H^\frac{1}{2}({\mathbb R}^2/\Lambda)^3,\, \bfF\cdot\bfN|_{z=\eta} \!\in\! H^\frac{1}{2}({\mathbb R}^2/\Lambda)\}, \\
& \{\bfF \!\in\! L^2(D_\eta/\Lambda)^3\colon \curl \bfF\!\in\! L^2(D_\eta/\Lambda)^3,\, \Div \bfF \!\in\! L^2(D_\eta/\Lambda),\, \bfF\times \bfe_3|_{z=-h} \!\in\! H^\frac{1}{2}({\mathbb R}^2/\Lambda)^3,\, \bfF_\parallel^\perp \!\in\! H^\frac{1}{2}({\mathbb R}^2/\Lambda)^2\}, \\
& \{\bfF \!\in\! L^2(D_\eta/\Lambda)^3\colon \curl \bfF\!\in\! L^2(D_\eta/\Lambda)^3,\, \Div \bfF \!\in\! L^2(D_\eta/\Lambda),\, \bfF\cdot\bfe_3|_{z=-h} \!\in\! H^\frac{1}{2}({\mathbb R}^2/\Lambda),\, \bfF\cdot\bfN|_{z=\eta} \!\in\! H^\frac{1}{2}({\mathbb R}^2/\Lambda)\}, \\
& \{\bfF \!\in\! L^2(D_\eta/\Lambda)^3\colon \curl \bfF\!\in\! L^2(D_\eta/\Lambda)^3,\, \Div \bfF \!\in\! L^2(D_\eta/\Lambda),\, \bfF\cdot\bfe_3|_{z=-h} \!\in\! H^\frac{1}{2}({\mathbb R}^2/\Lambda),\, \bfF_\parallel^\perp \!\in\! H^\frac{1}{2}({\mathbb R}^2/\Lambda)^2\}
\end{align*}
coincide with $H^1(D_\eta/\Lambda)^3$.
\item[(iii)]
The space
$$\{\bfF \in L^2(D_\eta/\Lambda)^3\colon \curl \bfF\in H^1(D_\eta/\Lambda)^3,\, \Div \bfF \in H^1(D_\eta/\Lambda),\, \bfF \times \bfe_3 |_{z=-h}= \mathbf{0},\, \bfF\cdot\bfn|_{z=\eta} = 0\}$$
coincides with $\{\bfF \in H^2(D_\eta/\Lambda)^3\colon \bfF \times \bfe_3 |_{z=-h}= \mathbf{0},\, \bfF\cdot\bfn|_{z=\eta} = 0\}$.
\end{itemize}
\end{proposition}

\begin{proposition} \label{prop:traces} $ $ 
\begin{itemize}
\item[(i)]
It follows from the formula
$$\int_\Omega \int_{-h}^\eta (\bfF\cdot\curl\bfG-\curl\bfF\cdot\bfG) = \int_\Omega \bfF_\parallel^\perp\cdot\bfG_\parallel,
\qquad \bfF \in {\mathscr D}(\overline{D_\eta/\Lambda})^3,\ \bfG \in H^1(D_\eta/\Lambda)^3, \bfG|_{z=-h}=\bfzero$$
that
the mapping $\bfF \mapsto \bfF_\parallel^\perp$ defined on ${\mathscr D}(\overline{D_\eta/\Lambda})^3$
extends to a continuous linear mapping\linebreak $\{\bfF \in L^2(D_\eta/\Lambda)^3\colon \curl \bfF \in L^2(D_\eta/\Lambda)^3\} \rightarrow H^{-\frac{1}{2}}({\mathbb R}^2/\Lambda)^2$,
where the former space is equipped with the norm\linebreak $\bfF \mapsto (\|\bfF\|_{L^2(D_\eta/\Lambda)^3}^2+\|\curl \bfF\|_{L^2(D_\eta/\Lambda)^3}^2)^\frac{1}{2}$.
\item[(ii)]
It follows from the formula
$$\int_\Omega \int_{-h}^\eta \curl\bfF \cdot \Grad \phi = \int_\Omega \curl \bfF\cdot\bfN \phi|_{z=\eta},
\qquad \bfF \in {\mathscr D}(\overline{D_\eta/\Lambda})^3,\ \phi \in H^1(D_\eta/\Lambda), \phi|_{z=-h}=0$$
that the mapping $\bfF \mapsto \curl \bfF\cdot\bfN|_{z=\eta}$ defined on ${\mathscr D}(\overline{D_\eta/\Lambda})^3$
extends to a continuous linear mapping\linebreak $\{\bfF \in L^2(D_\eta)^3\colon \Div \bfF \in L^2(D_\eta)\} \rightarrow H^{-\frac{1}{2}}({\mathbb R}^2)^2$,
where the former space is equipped with the norm\linebreak $\bfF \mapsto (\|\bfF\|_{L^2(D_\eta)^3}^2+\|\Div \bfF\|_{L^2(D_\eta)}^2)^\frac{1}{2}$.
\end{itemize}
\end{proposition}

\begin{proposition} \label{Solve scalar BVPs}
 The boundary-value problem
\begin{alignat*}{2}
\Delta \phi &=F & & \In D_\eta, \\
\partial_n \phi & = f & & \at z=\eta, \\
\phi &=0 & & \at z=-h
\end{alignat*}
has a unique solution $\phi \in H^2(D_\eta/\Lambda)$ for each $F \in L^2(D_\eta/\Lambda)$ and $f \in H^\frac{1}{2}(S_\eta/\Lambda)$.
\end{proposition}

\begin{lemma} \label{weak exist} $ $
\begin{itemize}
\item[(i)]
For all sufficiently small values of $|\alpha|$ the
boundary-value problem \eqref{A BVP 1}--\eqref{A BVP 5} admits a unique weak solution
for each $\bfgamma \in {\mathbb R}^2$ and $\Phi \in \mathring{H}^\frac{1}{2}({\mathbb R}^2/\Lambda)$. The weak solution is solenoidal and
satisfies \eqref{A BVP 1} in the sense of distributions and \eqref{A BVP 5} in $H^{-\frac{1}{2}}({\mathbb R}^2/\Lambda)^2$.
\item[(ii)]
For all sufficiently small values of $|\alpha|$ the
boundary-value problem \eqref{B BVP 1}--\eqref{B BVP 6} admits a unique weak solution
for each $\bfgamma \in {\mathbb R}^2$ and $\bfg \in H^\frac{1}{2}({\mathbb R}^2/\Lambda)^2$. The weak solution is solenoidal and
satisfies \eqref{B BVP 1} in the sense of distributions.
\end{itemize}
\end{lemma}
\begin{proof} $ $
\begin{itemize}
\item[(i)]
The estimates
\begin{align*}
\left| \int_\Omega \int_{-h}^\eta \curl \bfA \cdot\bfC \right| & \lesssim \|\bfA\|_{H^1(D_\eta/\Lambda)^3} \|\bfC\|_{H^1(D_\eta/\Lambda)^3}, \\
\left| \int_\Omega \nabla \Delta^{-1}( \nabla \cdot \bfA^{\!\perp}_\parallel) \cdot \bfC_\parallel \right|
& \lesssim \|\bfA_\parallel\|_0 \|\bfC_\parallel\|_0
\lesssim \|\bfA|_{z=\eta}\|_{\frac{1}{2}} \|\bfC|_{z=\eta}\|_{\frac{1}{2}}
\lesssim   \|\bfA\|_{H^1(D_\eta/\Lambda)^3} \|\bfC\|_{H^1(D_\eta/\Lambda)^3}
\end{align*}
and Proposition \ref{About spaces}(i) imply that for sufficiently small values of $|\alpha|$
the left-hand side of \eqref{Weak solution A}
is a continuous, coercive, bilinear form $\mathcal{X}_\eta \times \mathcal{X}_\eta \rightarrow {\mathbb R}$,
while the estimate
$$
\left| \int_\Omega (\bfgamma^\perp+\nablap \Phi)\cdot\bfC_\parallel \right| \lesssim
(|\bfgamma^\perp|+ \|\nablap \Phi\|_{-\frac{1}{2}})\|\bfC|_{z=\eta}\|_{\frac{1}{2}} \lesssim
(|\bfgamma|+ \|\Phi\|_{\frac{1}{2}}) \|\bfC\|_{H^1(D_\eta/\Lambda)^3}
$$
shows that its right-hand side is a continuous, bilinear form $({\mathbb R}^2 \times \mathring{H}^\frac{1}{2}({\mathbb R}^2/\Lambda)) \times {\mathcal X}_\eta
\rightarrow {\mathbb R}$. The existence of a unique solution $\bfA \in {\mathcal X}_\eta$
now follows from the Lax-Milgram lemma.

Let $\phi_\bfA \in H^2(D_\eta/\Lambda)$ be the unique function satisfying
$\Delta \phi_\bfA = \Div \bfA$ in $D_\eta$ with boundary conditions\linebreak
$\partial_n \phi_\bfA|_{y=\eta} =0$, $\phi_\bfA|_{z=-h}=0$ (see Proposition \ref{Solve scalar BVPs}).
It follows that $\bfC=\Grad \phi_\bfA \in {\mathcal X}_\eta$ and hence
$$
\int_\Omega \int_{-h}^\eta \big(-\alpha \curl \bfA \cdot \Grad \phi_\bfA + (\Div \bfA)^2\big)
-\alpha\int_\Omega \nabla \Delta^{-1}( \nabla \cdot \bfA^{\!\perp}_\parallel) \cdot \nabla (\phi_\bfA|_{z=\eta}) \\
=
0
$$
(because $\bfC_\parallel = \nabla ( \phi_\bfA|_{z=\eta})$, which is orthogonal
to $\bfgamma^\perp$ and $\nablap \Phi$). Since
$$
\int_\Omega \int_{-h}^\eta \!\!\curl \bfA \cdot \Grad \phi_\bfA
= \int_\Omega \curl \bfA \cdot \bfN\, \phi_\bfA|_{z=\eta}
=\int_\Omega \!\nabla \cdot \bfA^{\!\perp}_\parallel\, \phi_\bfA|_{z=\eta}
= - \int_\Omega \nabla \Delta^{-1}(\nabla \cdot \bfA^{\!\perp}_\parallel)\cdot \nabla (\phi_\bfA|_{z=\eta})
$$
(see Proposition \ref{prop:traces}(ii)), one concludes that $\Div \bfA=0$.

Choosing $\bfC \in {\mathscr D}(D_\eta/\Lambda)^3$, one finds that $\bfA$ solves \eqref{A BVP 1}
in the sense of distributions and hence that\linebreak
$\curl \curl \bfA \in L^2(D_\eta)^3$. It follows that
$(\curl \bfA)_\parallel^\perp \in H^{-\frac{1}{2}}({\mathbb R}^2)^2$ (Proposition \ref{prop:traces}(i)) and
$$
\int_\Omega \int_{-h}^\eta (\curl \curl \bfA -\alpha \curl \bfA) \cdot\bfC
+\int_\Omega \left((\curl \bfA)^\perp_\parallel -\bfgamma^\perp -\nablap \Phi
-\alpha\nabla \Delta^{-1}( \nabla \cdot \bfA^{\!\perp}_\parallel)\right) \cdot \bfC_\parallel
=0.
$$
One concludes that \eqref{A BVP 5} holds in $H^{-\frac{1}{2}}({\mathbb R}^2/\Lambda)^2$.
\item[(ii)]
Let $\bfF \in H^1(D_\eta/\Lambda)^3$ be a function such that $\bfF_\parallel = \bfg$ and $\bfF \times \bfe_3|_{z=-h}=\bfzero$, and let
$\phi_\bfF \in H^2(D_\eta/\Lambda)$ be the unique function satisfying
$\Delta \phi_\bfF = \Div \bfF$ in $D_\eta$ with boundary conditions
$\partial_n \phi_\bfF|_{z=\eta} =\bfF\cdot\bfn$, $\phi_\bfF|_{z=-h}=0$ (see Proposition \ref{Solve scalar BVPs}).
It follows that $\bfG\coloneqq \bfF-\Grad\phi_\bfF$ satisfies $\Div \bfG=0$,
$\bfG\cdot\bfn|_{z=\eta}=0$ and
$\nablac \bfG_\parallel^\perp = \nablac\bfg^\perp$ because $\nablac(\Grad \phi_\bfF)_\parallel^\perp=\nablac\nabla ( \phi_\bfF|_{z=\eta})^\perp=0$.
We accordingly seek $\bfC \in {\mathcal X}_\eta^0$ such that
\begin{align}
\int_\Omega \int_{-h}^\eta (\curl \bfC\cdot&\curl \bfD  -\alpha \curl \bfC \cdot\bfD + \Div \bfC\, \Div \bfD) \nonumber\\
&=-\int_\Omega \int_{-h}^\eta (\curl \bfG\cdot\curl \bfD -\alpha \curl \bfG \cdot\bfD)
+\int_\Omega (\bfgamma^\perp+\alpha\nabla \Delta^{-1}( \nablac \bfG^\perp) )\cdot \bfD_\parallel
\label{Weak solution mod}
\end{align}
 for all $\bfD \in {\mathcal X}_\eta^0$, so that $\bfB=\bf C+\bfG$ is a weak solution of \eqref{B BVP 1}--\eqref{B BVP 6}.

For sufficiently small values of $|\alpha|$
the left-hand side of \eqref{Weak solution mod}
is a continuous, coercive, bilinear form\linebreak
$\mathcal{X}_\eta^0 \times \mathcal{X}_\eta^0 \rightarrow {\mathbb R}$,
while the right-hand side is a continuous, bilinear form $({\mathbb R}^2 \times \mathcal{X}_\eta) \times \mathcal{X}_\eta^0 \rightarrow {\mathbb R}$.
The existence of a unique function $\bfC \in {\mathcal X}_\eta^0$ satisfying \eqref{Weak solution mod} for all $\bfD \in {\mathcal X}_\eta^0$
now follows from the Lax-Milgram lemma, and the corresponding weak solution $\bfB$ to \eqref{B BVP 1}--\eqref{B BVP 6}
is unique since the difference between two weak solutions satisfies \eqref{Weak solution mod} with $\bfgamma=\bfzero$ and
$\bfg=\bfzero$ for all
$\bfD \in {\mathcal X}_\eta^0$ and is therefore zero.

Let $\phi_\bfB \in H^2(D_\eta/\Lambda)$ be the unique function satisfying
$\Delta \phi_\bfB = \Div \bfB$ in $D_\eta$ with boundary conditions\linebreak
$\partial_n \phi_\bfB|_{z=\eta} =0$, $\phi_\bfB|_{z=-h}=0$
(see Proposition \ref{Solve scalar BVPs}). Substituting $\bfD=\Grad\phi_\bfB \in {\mathcal X}_\eta^0$ into
\eqref{Weak solution B}, we find that
$$
\int_\Omega \int_{-h}^\eta \big(-\alpha \curl \bfB \cdot \Grad \phi_\bfB + (\Div \bfB)^2\big)
=\alpha\int_\Omega \nabla \Delta^{-1}( \nabla \cdot \bfB^{\!\perp}_\parallel) \cdot \nabla (\phi_\bfB|_{z=\eta}),
$$
and since
$$
-\int_\Omega \int_{-h}^\eta \!\!\curl \bfB \cdot \Grad \phi_\bfB
=-\int_\Omega \!\nabla \cdot \bfB^{\!\perp}_\parallel\, \phi_\bfB|_{z=\eta}
=  \int_\Omega \nabla \Delta^{-1}(\nabla \cdot \bfB^{\!\perp}_\parallel)\cdot \nabla (\phi_\bfB|_{z=\eta}),
$$
one concludes that $\Div \bfB=0$.

Finally, taking $\bfD \in {\mathscr D}(D_\eta/\Lambda)^3$ in \eqref{Weak solution B}, we find that $\bfB$ satisfies \eqref{B BVP 1} in the sense of distributions.\qedhere
\end{itemize}
\end{proof}

\begin{lemma} $ $\label{strong exist}
\begin{itemize}
\item[(i)]
Suppose that $\bfgamma \in {\mathbb R}^2$ and $\Phi \in \mathring{H}^\frac{3}{2}({\mathbb R}^2/\Lambda)$. Any weak solution $\bfA$ of
\eqref{A BVP 1}--\eqref{A BVP 5} is in fact a strong solution.
\item[(ii)]
Suppose that $\bfgamma \in {\mathbb R}^2$ and $\bfg \in H^\frac{3}{2}({\mathbb R}^2/\Lambda)^2$. Any weak solution $\bfB$ of
\eqref{B BVP 1}--\eqref{B BVP 6} is in fact a strong solution.
\end{itemize}
\end{lemma}

\begin{proof}$ $
\begin{itemize}
\item[(i)]
Recall that $\curl \curl \bfA \in L^2(D_\eta/\Lambda)^3$ and
$$
(\curl \bfA)^\perp_\parallel =\bfgamma^\perp+\nablap \Phi
+\alpha\nabla \Delta^{-1}( \nabla \cdot \bfA^{\!\perp}_\parallel)
$$
holds in $H^{-\frac{1}{2}}({\mathbb R}^2/\Lambda)^2$; hence
$(\curl \bfA)_\parallel^\perp \in H^{\frac{1}{2}}({\mathbb R}^2/\Lambda)^2$
(because the right-hand side of this equation belongs to $H^{\frac{1}{2}}({\mathbb R}^2/\Lambda)^2$).
Since $0=\Div \curl \bfA \in L^2(D_\eta/\Lambda)$
and $\curl \bfA\cdot\bfe_3|_{z=-h} =0$ it follows that $\curl \bfA \in H^1(D_\eta/\Lambda)^3$ (Proposition
\ref{About spaces}(ii)), and furthermore
$\curl \bfA \in H^1(D_\eta/\Lambda)^3$, $0=\Div \bfA \in H^1(D_\eta/\Lambda)$ with
$\bfA \times \bfe_3 |_{z=-h}= \mathbf{0}$, $\bfA\cdot\bfn|_{z=\eta} = 0$ imply that $\bfA \in H^2(D_\eta/\Lambda)^3$
(Proposition \ref{About spaces}(iii)). Finally note that \eqref{A BVP 1} holds in $L^2(D_\eta/\Lambda)^3$ because
it holds in the sense of distributions and $\bfA \in H^2(D_\eta/\Lambda)^3$.
\item[(ii)]
Clearly $0=\Div \curl \bfB \in L^2(D_\eta/\Lambda)$
and
$\curl \curl \bfB \in L^2(D_\eta/\Lambda)^3$ because \eqref{B BVP 1} is satisfied in the sense of distributions
and $\curl \bfB \in L^2(D_\eta/\Lambda)^3$; furthermore
$$\curl \bfB \cdot \bfN|_{z=\eta} = \nablac \bfB_\parallel^\perp=\nablac \bfg^\perp \in \mathring{H}^\frac{1}{2}({\mathbb R}^2/\Lambda),
\qquad
\curl \bfB \cdot \bfe_3|_{z=-h}=0,$$
so that $\curl \bfB \in H^1(D_\eta / \Lambda)^3$ by Proposition \ref{About spaces}(ii).
Next note that
$\curl \bfB \in H^1(D_\eta/\Lambda)^3$, $0=\Div \bfB \in H^1(D_\eta/\Lambda)$ with
$\bfB \times \bfe_3 |_{z=-h}= \bfzero$, $\bfB\cdot\bfn|_{z=\eta} = 0$ implies that $\bfB \in H^2(D_\eta/\Lambda)^3$
by Proposition \ref{About spaces}(iii), and
\eqref{B BVP 1} holds in $L^2(D_\eta/\Lambda)^3$ because
it holds in the sense of distributions and $\bfB \in H^2(D_\eta/\Lambda)^3$. Finally
$$
\underbrace{\int_\Omega \int_{-h}^\eta (\curl \curl \bfB -\alpha \curl \bfB) \cdot\bfD}_{\displaystyle =0}
+\int_\Omega \left((\curl \bfB)^\perp_\parallel -\bfgamma^\perp
-\alpha\nabla \Delta^{-1}( \nabla \cdot \bfB^{\!\perp}_\parallel)\right) \cdot \bfD_\parallel
=0
$$
for all $\bfD \in {\mathcal X}_\eta^0$,
which implies that $(\curl \bfB)^\perp_\parallel =\bfgamma^\perp + \nablap \Phi
+\alpha\nabla \Delta^{-1}( \nabla \cdot \bfB^{\!\perp}_\parallel)$ for some
$\Phi \in \mathring{H}^\frac{3}{2}({\mathbb R}^2/\Lambda)$ and in particular that \eqref{B BVP 6} holds.
\qedhere
\end{itemize}
\end{proof}

\begin{remark} \label{rem:orthpart}
Suppose that $\bfB \in H^2(D_\eta/\Lambda)^3$ satisfies \eqref{B BVP 1}--\eqref{B BVP 6}.
The orthogonal gradient part of $(\curl \bf B)_\parallel$ is equal to
$-\alpha \, \nablap \Delta^{-1} (\nablac \bfB_\parallel^\perp)$.
\end{remark}

\begin{corollary}
The formulae \eqref{eq:alt defn of H and M} define linear operators
$H(\eta)\colon {\mathbb R}^2 \times \mathring{H}^\frac{3}{2}({\mathbb R}^2/\Lambda) \rightarrow \mathring{H}^\frac{1}{2}({\mathbb R}^2/\Lambda)$
and\linebreak
$\bfM(\eta)\colon {\mathbb R}^2 \times H^\frac{3}{2}({\mathbb R}^2/\Lambda)^2 \rightarrow H^\frac{1}{2}({\mathbb R}^2/\Lambda)^2$.
\end{corollary}

\subsection{Analyticity} \label{Analflat}
In this section we show that improved regularity of $\eta$, $\Phi$ and $\bfg$ leads to improved regularity of the solution to the boundary-value problems \eqref{A BVP 1}--\eqref{A BVP 5} and \eqref{B BVP 1}--\eqref{B BVP 6} and use this result to deduce that $H(\eta)$ and $\bfM(\eta)$ depend analytically upon $\eta$ (see Theorem \ref{thm:analinvGDNO}(i) below for a precise statement).
We proceed by transforming \eqref{A BVP 1}--\eqref{A BVP 5} and \eqref{B BVP 1}--\eqref{B BVP 6} into equivalent boundary-value problems in the fixed domain $D_0$ by means
of the following flattening transformation. Define $\Sigma\colon D_0 \to D_\eta$ by
$$
\Sigma\colon(\bfx,v) \mapsto (\bfx,v+\sigma(\bfx,v)), \qquad \sigma(\bfx,v)\coloneqq  \eta(\bfx)(1+v/h),
$$
and for $f\colon D_\eta \rightarrow {\mathbb R}$ and $\bfF\colon D_\eta \rightarrow {\mathbb R}^3$ write $\tilde{f}=f \circ \Sigma$, $\tilde{\bfF} = \bfF \circ \Sigma$
and use the notation
\begin{align*}
\mathrm{grad}^\sigma \tilde{f}(\bfx,v) &\coloneqq (\mathrm{grad} \, f)\circ\Sigma(\bfx,v), \\
\Div^\sigma \tilde{f}(\bfx,v) &\coloneqq  (\Div \, f) \circ \Sigma(\bfx,v),
\\
\curl^\sigma \tilde{\bfF}(\bfx,v) &\coloneqq (\curl \, \bfF)\circ\Sigma(\bfx,v),
\\
\Delta^\sigma \tilde{f}(\bfx,v) &\coloneqq  (\Delta \, f) \circ \Sigma(\bfx,v)
\end{align*}
and more generally
$$
\partial_x^\sigma \coloneqq  \partial_x - \frac{ \partial_x\sigma }{1+\partial_v\sigma} \partial_v, \quad
\partial_y^\sigma \coloneqq  \partial_y - \frac{ \partial_y\sigma }{1+\partial_v\sigma} \partial_v, \quad
\partial_v^\sigma \coloneqq  \frac{ \partial_v }{1+\partial_v\sigma}.
$$

\begin{remark}
The flattened versions of the operators $\curl$, $\Div$, $\mathrm{grad}$ and $\Delta$ applied to $\tilde\bfF(x,y,v)=\bfF(x,y,z)$ and to $\tilde{f}(x,y,v)=f(x,y,z)$ are given explicitly by
\begin{align*}
\curl^\sigma \tilde\bfF &= \curl \tilde\bfF - \frac{\eta}{\eta+h} (-\partial_v \tilde{F}_2,\partial_v \tilde{F}_1,0)^T - \frac{h+v}{\eta+h} (\eta_y \; \partial_v\tilde{F}_3, -\eta_x \; \partial_v\tilde{F}_3,\eta_x \; \partial_v \tilde{F}_2-\eta_y \; \partial_v \tilde{F}_1)^T, \\
\Div^\sigma \tilde\bfF &= \Div \tilde\bfF - \frac{h+v}{\eta+h} (\eta_x \; \partial_v\tilde{F}_1 + \eta_y \; \partial_v\tilde{F}_2) - \frac{\eta}{\eta+h} \partial_v\tilde{F}_3, \\
\mathrm{grad}^\sigma \tilde{f} &= \mathrm{grad} \; \tilde{f} - \frac{h+v}{\eta+h} (\eta_x \; \partial_v\tilde{f}, \eta_y \; \partial_v\tilde{f},0)^T - \frac{\eta}{\eta+h} (0,0,\partial_v\tilde{f})^T, \\
\Delta^\sigma \tilde{f} &= \Delta \tilde{f} - 2 \frac{h+v}{\eta+h}( \eta_x \; \partial_{vx}^2\tilde{f} + \eta_y \; \partial_{vy}^2\tilde{f} ) - \frac{h+v}{\eta+h} (\eta_{xx}+\eta_{yy}) \partial_v\tilde{f}  \\
&\qquad\mbox{}+ 2 \frac{h+v}{(\eta+h)^2} ( \, (\eta_x)^2+(\eta_y)^2 \, ) \partial_v\tilde{f} + \left( \frac{h+v}{\eta+h} \right)^2 \, ( \, (\eta_x)^2+(\eta_y)^2 \, ) \partial_v^2\tilde{f} - \frac{\eta^2+2h\eta}{(\eta+h)^2} \partial_v^2\tilde{f}.
\end{align*}
\end{remark}

Equations
\eqref{A BVP 1}--\eqref{A BVP 5} are equivalent to the flattened boundary-value problem
\begin{alignat}{2}
\curl^\sigma \curl^\sigma \tilde\bfA - \alpha  \curl^\sigma \tilde\bfA &= \bfzero & & \In D_0, \label{Flattened A BVP 1} \\
\Div^\sigma \bfA &= 0 & & \In D_0, \label{Flattened A BVP 2} \\
\tilde\bfA \times \bfe_3 &= \bfzero & & \at v=-h, \label{Flattened A BVP 3} \\
\tilde\bfA \cdot \bfN &= 0 & & \at v=0, \label{Flattened A BVP 4} \\
(\curl^\sigma \tilde\bfA)_\parallel &= \bfgamma+\nabla\Phi - \alpha \nablap\Delta^{-1}(\nablac \tilde{\bfA}_\parallel^\perp) & & \at v=0, \label{Flattened A BVP 5}
\end{alignat}
in terms of which
$$
H(\eta)(\bfgamma,\Phi)=\nablac \tilde\bfA_\parallel^\perp,
$$
while equations \eqref{B BVP 1}--\eqref{B BVP 6} are equivalent to the flattened boundary-value problem
\begin{alignat}{2}
\curl^\sigma \curl^\sigma \tilde\bfB - \alpha  \curl^\sigma \tilde\bfB &= \bfzero & & \In D_0, \label{Flattened B BVP 1} \\
\Div^\sigma \bfB &= 0 & & \In D_0, \label{Flattened B BVP 2} \\
\tilde\bfB \times \bfe_3 &= \bfzero & & \at v=-h, \label{Flattened B BVP 3} \\
\tilde\bfB \cdot \bfN &= 0 & & \at v=0, \label{Flattened B BVP 4} \\
\nablac \tilde{\bfB}_{\parallel}^{\perp} &= \nablac \bfg^\perp, \label{Flattened B BVP 5} \\
\langle (\curl^\sigma \tilde\bfB)_\parallel \rangle &= \bfgamma, \label{Flattened B BVP 6}
\end{alignat}
in terms of which
$$
\bfM(\eta)(\bfgamma,\bfg)=-(\curl^\sigma \tilde\bfB)_\parallel;
$$
note that the
orthogonal gradient part of $(\curl^\sigma \tilde{\bfB})_\parallel$ is equal to
$-\alpha \, \nablap \Delta^{-1} (\nablac \tilde{\bfB}_\parallel^\perp)$ for
any solution $\tilde{\bfB} \in H^2(D_0/\Lambda)^3$ of \eqref{Flattened B BVP 1}--\eqref{Flattened B BVP 5}. The spatially extended version of
the first of the above boundary-value problems
was studied by Groves and Horn \cite[\S4]{GrovesHorn20}, whose analysis in particular leads to the following result in the present context.

\begin{theorem} \label{tildeAisanal}
Suppose that $s \geq 2$, and assume that the non-resonance condition (NR) holds.
There exists an open neighbourhood $U$ of the origin in $H^{s+\frac{1}{2}}({\mathbb R}^2/\Lambda)$ such that
the boundary-value problem \eqref{Flattened A BVP 1}--\eqref{Flattened A BVP 5}
has a unique solution $\tilde{\bfA}=\tilde{\bfA}(\eta,\bfgamma,\Phi)$ in $H^s(D_0/\Lambda)^3$ which depends analytically upon
$\eta \in U$, $\bfgamma \in {\mathbb R}^2$ and\linebreak
$\Phi \in \mathring{H}^{s-\frac{1}{2}}({\mathbb R}^2/\Lambda)$ (and linearly upon $(\bfgamma,\Phi)$).
\end{theorem}

The corresponding result for the boundary-value problem \eqref{Flattened B BVP 1}--\eqref{Flattened B BVP 6},
together with the analyticity of the operators $H$ and $M$, is now readily deduced.

\begin{theorem} \label{thm:analinvGDNO}
Suppose that $s \geq 2$, and assume that the non-resonance condition (NR) holds.
for each $\bfk \in \Lambda^\prime$.
There exists an open neighbourhood $U$ of the origin in $H^{s+\frac{1}{2}}(\mathbb{R}^2/\Lambda)$ such that
\begin{itemize}
\item[(i)] $\eta \mapsto H(\eta)$ and $\eta \mapsto \bfM(\eta)$ are analytic mappings $U \to L({\mathbb R}^2 \times \mathring{H}^{s-\frac{1}{2}}({\mathbb R}^2/\Lambda),\mathring{H}^{s-\frac{3}{2}}({\mathbb R}^2/\Lambda))$ and\linebreak
$U \to L({\mathbb R}^2 \times H^{s-\frac{1}{2}}(\mathbb{R}^2/\Lambda)^2, H^{s-\frac{3}{2}}(\mathbb{R}^2/\Lambda)^2)$ respectively;
\item[(ii)] the boundary-value problem \eqref{Flattened B BVP 1}--\eqref{Flattened B BVP 6} has a unique solution
$\tilde{\bfB}=\tilde{\bfB}(\eta,\bfgamma,\bfg)$ in $H^s(D_0/\Lambda)^3$ which depends analytically upon
$\eta \in U$ and $\bfg \in H^{s-\frac{3}{2}}({\mathbb R}^2)^2$ (and linearly upon $(\bfgamma,\bfg)$).
\end{itemize}
\end{theorem}

\begin{proof}
$ $
The analyticity of $H(\cdot)\colon U \to L({\mathbb R}^2 \times \mathring{H}^{s-\frac{1}{2}}({\mathbb R}^2/\Lambda),\mathring{H}^{s-\frac{3}{2}}({\mathbb R}^2/\Lambda))$ follows from Theorem \ref{tildeAisanal} and equation \eqref{eq:alt defn of H and M}, and it follows that
the formula
$$V(\eta)\begin{pmatrix} \bfgamma \\ \Phi \end{pmatrix} = \begin{pmatrix} \bfgamma \\ H(\eta)(\bfgamma,\Phi) \end{pmatrix}$$
defines an analytic function $V\colon U \to L({\mathbb R}^2 \times \mathring{H}^{s-\frac{1}{2}}({\mathbb R}^2/\Lambda),{\mathbb R}^2 \times \mathring{H}^{s-\frac{3}{2}}({\mathbb R}^2/\Lambda))$. A straightforward calculation shows that
$$V(0)\begin{pmatrix} \bfgamma \\ \Phi \end{pmatrix}  = \begin{pmatrix} \bfgamma \\ D^2\mathtt{t}(D) \end{pmatrix},$$
and 
$V(0) \in L({\mathbb R}^2 \times \mathring{H}^{s-\frac{1}{2}}({\mathbb R}^2/\Lambda),{\mathbb R}^2 \times \mathring{H}^{s-\frac{3}{2}}({\mathbb R}^2/\Lambda))$ is an isomorphism because
$$\lim_{|\bfk| \rightarrow \infty} \frac{|\bfk|}{|\bfk|^2 \mathtt{t}(|\bfk|)} = 1.$$
One concludes that $V(\eta) \in L({\mathbb R}^2 \times \mathring{H}^{s-\frac{1}{2}}({\mathbb R}^2/\Lambda),{\mathbb R}^2 \times \mathring{H}^{s-\frac{3}{2}}({\mathbb R}^2/\Lambda))$ is an isomorphism for each $\eta \in U$ and that $V(\eta)^{-1} \in
L({\mathbb R}^2 \times \mathring{H}^{s-\frac{3}{2}}({\mathbb R}^2/\Lambda),{\mathbb R}^2 \times \mathring{H}^{s-\frac{1}{2}}({\mathbb R}^2/\Lambda))$
also depends analytically upon $\eta \in U$. Clearly
$$V(\eta)^{-1} = \begin{pmatrix} \mathbb{I}_2 \\ W_2(\eta) \end{pmatrix}$$
for some analytic function $W_2\colon U \rightarrow L({\mathbb R}^2 \times \mathring{H}^{s-\frac{3}{2}}({\mathbb R}^2/\Lambda),\mathring{H}^{s-\frac{1}{2}}({\mathbb R}^2/\Lambda))$.

Observe that $\tilde{\bfB}(\eta,\bfgamma,\bfg)\coloneqq \tilde\bfA(\eta,\bfgamma,\Phi)$ with $\Phi=W_2(\eta)(\bfgamma,\nablac \bfg^\perp)$ 
depends analytically upon $\eta$, $\bfgamma$ and $\bfg$, and solves \eqref{Flattened B BVP 1}--\eqref{Flattened B BVP 6} 
since by construction
$$\nablac \bfg^\perp = H(\eta)(\bfgamma,\Phi) = \nablac \tilde\bfA(\eta,\bfgamma,\Phi)_\parallel^\perp= \nablac \tilde\bfB(\eta,\bfgamma,\Phi)_\parallel^\perp.$$
The uniqueness of this solution follows by noting that any other solution
$\tilde{\bfB}(\eta,\bfgamma,\bfg)$ is equal to $\tilde\bfA(\eta,\bfgamma,\Phi)$ with $\Phi=\Delta^{-1}\nablac (\curl^\sigma \bfB)_\parallel$, so that 
$$H(\eta)(\bfgamma,\Phi)=\nablac \tilde{\bfA}(\eta,\bfgamma,\Phi)_\parallel^\perp=\nablac \tilde{\bfB}(\eta,\bfgamma,\bfg)_\parallel^\perp = \nablac \bfg^\perp,$$
 that is $\Phi=W_2(\eta)(\bfgamma,\nablac \bfg^\perp)$. Finally, the analyticity of $\bfM$ follows from the calculation
\begin{align*}
\bfM(\eta)(\bfgamma,\bfg) &=-(\curl^\sigma \tilde{\bfB}(\eta,\bfgamma,\bfg))_\parallel \\
& = -(\curl^\sigma \tilde{\bfA}(\eta,\bfgamma,\Phi))_\parallel  \\
& = -\bfgamma - \nabla \Phi + \alpha \, \nablap \Delta^{-1} (\nablac \bfg^\perp)
\end{align*}
with $\Phi=W_2(\eta)(\bfgamma,\nablac \bfg^\perp)$.\qedhere
\end{proof}

\begin{remark} \label{rem:explicit linearisations}
It follows from the proof of Theorem \ref{thm:analinvGDNO} that
\begin{align*}
H(0)(\bfgamma,\Phi) & = D^2 \, \mathtt{t}(D) \, \Phi, \\
\bfM(0)(\bfgamma,\bfg) & = -\bfgamma - \nabla \left( \frac{1}{ D^2 \mathtt{t}(D) } \, \nabla \cdot \bfg^{\perp} \right) + \alpha \, \nablap \Delta^{-1} \left( \nabla \cdot \bfg^{\perp} \right) \\
& = - \bfgamma + \frac{1}{D^2}\left( \alpha \, \mathbf{D}^{\perp} + \mathbf{D} \, \mathtt{c}(D) \right) \bfD \cdot \bfg^\perp.
\end{align*}
\end{remark}

We conclude this section by recording the following flattened version of Proposition \ref{Solve scalar BVPs}, which is established by the methods used by Groves and Horn \cite[\S4(c)]{GrovesHorn20}.

\begin{proposition} \label{Solve flattened scalar BVPs}
Suppose that $s \geq 2$. There exists an open neighbourhood $U$ of the origin in $H^{s+1/2}(\mathbb{R}^2/\Lambda)$ such that
the boundary-value problem
\begin{alignat*}{2}
\Delta^\sigma \phi &=F & & \In D_0, \\
\Grad^\sigma u\cdot\bfN & = f & & \at v=0, \\
\phi &=0 & & \at v=-h
\end{alignat*}
has a unique solution $\phi \in H^s(D_0/\Lambda)$ which depends analytically upon $\eta \in U$, 
$F \in H^{s-2}(D_\eta/\Lambda)$ and\linebreak $F \in H^{s-\frac{3}{2}}({\mathbb R}^2/\Lambda)$ (and linearly upon $F$ and $f$).
\end{proposition}

\subsection{Differentials} \label{sec:diffNHOp}

In this section we derive useful formulae for the differentials $\mathrm{d}H[\eta](\delta\eta)(\bfgamma,\Phi)$ and\linebreak
$\mathrm{d}\bfM[\eta](\delta\eta)(\bfgamma,\bfg)$, where  $\eta \in U$, $\bfgamma \in {\mathbb R}^2$,
$\Phi \in H^{s-3/2}(\mathbb{R}^2/\Lambda)$
and $\bfg \in H^{s-1/2}(\mathbb{R}^2/\Lambda)^2$, so that $\tilde{\bfA}$, $\tilde\bfB \in H^s(D_0/\Lambda)^3$ (in the notation of Section \ref{Analflat}),
working under the stronger condition $s \geq 3$ and again assuming the non-resonance condition (NR).
Recall the identity
\begin{equation}
\mathrm{d}(\partial_x^\sigma f) = \partial_x^\sigma(\mathrm{d}f - \mathrm{d}\sigma  \partial_v^\sigma f) + \mathrm{d}\sigma  \partial_v^\sigma\partial_x^\sigma f,
\label{eq:Alinhac}
\end{equation}
where $\partial_x$ can be replaced by $\partial_y$ or $\partial_v$ and
$\mathrm{d}$ can be any linearisation operator (see Castro and Lannes \cite[Eq.\ (3.41)]{CastroLannes15}); the quantity $\mathrm{d}f - \mathrm{d}\sigma  \partial_v^\sigma f$
is called \emph{Alinhac's good unknown}.

We proceed by finding a boundary-value problem for
$\tilde\bfC\coloneqq(\mathrm{d} \tilde\bfA -\mathrm{d}\sigma  \partial_v^\sigma\tilde\bfA) \in H^{s-1}(D_0/\Lambda)^3$, where $\mathrm{d} = \partial_\eta$, observing that
$H(\cdot)\colon U \to L({\mathbb R}^2 \times \mathring{H}^{s-\frac{1}{2}}({\mathbb R}^2/\Lambda),\mathring{H}^{s-\frac{3}{2}}({\mathbb R}^2/\Lambda))$
and $\tilde{\bfA}\colon U \rightarrow L({\mathbb R}^2 \times \mathring{H}^{s-\frac{1}{2}}({\mathbb R}^2/\Lambda,H^s(D_0/\Lambda)^3)$ are analytic.
Applying \eqref{eq:Alinhac} with $\mathrm{d} = \partial_\eta$ to
$$
H(\eta)(\bfgamma,\Phi)=\nablac \tilde\bfA_\parallel^\perp,
$$
and to equations \eqref{Flattened A BVP 1}--\eqref{Flattened A BVP 5}, we find that
$$
\mathrm{d}H[\eta](\delta\eta)(\bfgamma,\Phi) = \nablac \tilde{\bfC}_\parallel^\perp + \partial_v^\sigma \curl^\sigma \tilde{\bfA} \cdot \bfN|_{v=0}\delta\eta - (\curl^\sigma \tilde{\bfA})_\mathrm{h}\cdot\nabla \delta\eta
$$
where
\begin{alignat}{2}
 \curl^\sigma \curl^\sigma \tilde\bfC - \alpha  \curl^\sigma \tilde\bfC &= \bfzero & & \In D_0, \label{eq:differential A 1} \\
\Div^\sigma \tilde\bfC &= 0 & & \In D_0, \label{eq:differential A 2} \\
\tilde\bfC \times \bfe_3 &= \bfzero & & \at v=-h, \label{eq:differential A 3}\\
\tilde\bfC \cdot \bfN &= \nabla \delta \eta \cdot \tilde\bfA_\mathrm{h} - \delta\eta \; \partial_v^\sigma \tilde\bfA \cdot \bfN & & \at v=0, \label{eq:differential A 4}
\end{alignat}
and
\begin{align}
(\curl^\sigma \tilde\bfC)_\parallel^\perp & = -\delta\eta \partial_v^\sigma (\curl^\sigma \tilde{\bfA})_\mathrm{h}-\delta\eta \partial_v^\sigma (\curl^\sigma \tilde{\bfA})_3|_{v=0}\nabla\eta \nonumber \\
& \quad\qquad\mbox{}- (\curl^\sigma \tilde{\bfA})_3|_{v=0}\nabla \delta\eta
-\alpha\nablap\Delta^{-1} (\nablac \tilde{\bfC}_\parallel^\perp - \nablac((\curl^\sigma \tilde{\bfA})_\mathrm{h})\delta\eta).
 \label{eq:differential A 5}
\end{align}
(Equation \eqref{eq:differential A 4} can be rewritten as
$$\tilde\bfC \cdot  \bfN = \nablac ( \tilde\bfA_\mathrm{h}\delta \eta)$$
because $\Div^\sigma \tilde\bfA|_{v=0}=0$ implies that 
$$\partial_v^\sigma \tilde\bfA \cdot \bfN|_{v=0} = - \nablac \tilde\bfA_\mathrm{h}.\ )$$

Using the relation
$$-(\partial_v^\sigma \curl \tilde{\bfA})_\parallel = -\nabla(\curl \tilde{\bfA})_3|_{v=0} - \alpha (\curl^\sigma \tilde{\bfA})_\mathrm{h}^\perp$$
we can rewrite equation \eqref{eq:differential A 5} as
\begin{align*}
(\curl^\sigma \tilde\bfC)_\parallel^\perp & = -\nabla((\curl^\sigma \tilde{\bfA})_3|_{v=0}\delta\eta)-\alpha(\curl^\sigma \tilde{\bfA})_\mathrm{h}^\perp \delta \eta \\
& \quad\qquad -\alpha\nablap\Delta^{-1} \nablac \tilde{\bfC}_\parallel^\perp + \alpha \nablap \Delta^{-1} \nablapc ((\curl^\sigma \tilde{\bfA})_\mathrm{h}^\perp\delta\eta)\\
& = -\alpha \langle (\curl^\sigma \tilde{\bfA})_\mathrm{h}^\perp \delta\eta \rangle
- \alpha \nabla \Delta^{-1} \nabla\cdot ((\curl^\sigma \tilde{\bfA})_\mathrm{h}^\perp\delta\eta)
-\nabla((\curl^\sigma \tilde{\bfA})_3|_{v=0}\delta\eta)-\alpha\nablap\Delta^{-1} \nablac \tilde{\bfC}_\parallel^\perp,
\end{align*}
and writing $\tilde\bfC=\tilde\bfC^\prime +\Grad^\sigma \varphi$, where $\varphi \in H^s(D_0)$ is the unique solution of the boundary-value problem
\begin{alignat*}{2}
\Delta^\sigma \varphi&=0 & & \In D_0,\\
\Grad^\sigma\varphi\cdot\bfN &=\nablac ( \tilde\bfA_\mathrm{h}\delta \eta) & & \at v=0,\\
\varphi &=0 & & \at v=-h
\end{alignat*}
(see Proposition \ref{Solve flattened scalar BVPs}), one finds that
$$
\mathrm{d}H[\eta](\delta\eta)(\bfgamma,\Phi) = \nablac \tilde{\bfC}_\parallel^{\prime\perp} + \partial_v^\sigma \curl^\sigma \tilde{\bfA} \cdot \bfN|_{v=0}\delta\eta - (\curl^\sigma \tilde{\bfA})_\mathrm{h}\cdot\nabla \delta\eta,
$$
where
\begin{alignat*}{2}
\curl^\sigma \curl^\sigma \tilde\bfC^\prime - \alpha  \curl^\sigma \tilde\bfC^\prime &= \bfzero & & \In D_0, \\
\Div^\sigma \tilde\bfC^\prime &= 0 & & \In D_0, \\
\tilde\bfC^\prime \times \bfe_3 &= \bfzero & & \at v=-h, \\
\tilde\bfC^\prime \cdot \bfN &= 0 & & \at v=0,
\end{alignat*}
and 
\begin{align*}
(\curl^\sigma \tilde\bfC^\prime)_\parallel^\perp & =
-\alpha \langle (\curl^\sigma \tilde{\bfA})_\mathrm{h}^\perp \delta\eta \rangle - \alpha \nabla \Delta^{-1} \nabla\cdot ((\curl^\sigma \tilde{\bfA})_\mathrm{h}^\perp\delta\eta)\\
& \quad\qquad\mbox{}
-\nabla((\curl^\sigma \tilde{\bfA})_3|_{v=0}\delta\eta)-\alpha\nablap\Delta^{-1} \nablac \tilde\bfC_\parallel^{\prime\perp}.
\end{align*}

It follows that
\begin{align*}
\mathrm{d}H&[\eta](\delta\eta)(\bfgamma,\Phi)\\
& = H(\eta)\!\left(\!-\alpha \langle(\curl^\sigma \tilde{\bfA})_\mathrm{h}^\perp\delta\eta\rangle,
- \alpha\Delta^{-1} \nablac((\curl^\sigma \tilde{\bfA})_\mathrm{h}^\perp\delta\eta)-(\curl^\sigma \tilde{\bfA})_3|_{v=0}\delta\eta
+\langle (\curl^\sigma \tilde{\bfA})_3|_{v=0}\delta\eta \rangle\!\right) \\
& \qquad\quad\mbox{}-\nablac((\curl^\sigma \tilde{\bfA})_\mathrm{h} \delta\eta),
\end{align*}
and we obtain our final theorem by setting
$u\coloneqq (\curl^\sigma \tilde{\bfA})_{3}|_{v=0}$.

\begin{theorem} \label{diff of H}
Suppose that $s\geq3$ and that the non-resonance condition (NR) holds. The differential of the operator
$H(\cdot)\colon U \to L({\mathbb R}^2 \times \mathring{H}^{s-\frac{1}{2}}({\mathbb R}^2/\Lambda),\mathring{H}^{s-\frac{3}{2}}({\mathbb R}^2/\Lambda))$ is given by
\begin{align*}
\mathrm{d}H&[\eta](\delta\eta)(\bfgamma,\Phi)\\
&=
H(\eta)\left(-\alpha\langle(\bfK(\eta)(\bfgamma,\Phi)-u\nabla\eta)^\perp\delta\eta\rangle,-\alpha\Delta^{-1}\nablac((\bfK(\eta)(\bfgamma,\Phi)-u\nabla\eta)^\perp\delta\eta)-u\delta\eta + \langle u \delta\eta\rangle\right)\\
& \qquad\quad\mbox{}
-\nablac((\bfK(\eta)(\bfgamma,\Phi)-u\nabla\eta)\delta\eta),
\end{align*}
where
$$
\bfK(\eta)(\bfgamma,\Phi) = \bfgamma + \nabla\Phi - \alpha \nablap \Delta^{-1} H(\eta)(\bfgamma,\Phi), \qquad
u=\frac{\bfK(\eta)(\bfgamma,\Phi)\cdot \nabla \eta+H(\eta)(\gamma,\Phi)}{1+|\nabla \eta|^2}.$$
\end{theorem}

The corresponding result for $\bfM(\eta)$ is obtained in a similar fashion.

\begin{theorem} \label{diff of M}
Suppose that $s\geq3$ and that the non-resonance condition (NR) holds. The differential of the operator
$\bfM(\cdot)\colon U \to L({\mathbb R}^2 \times H^{s-\frac{1}{2}}(\mathbb{R}^2/\Lambda)^2, H^{s-\frac{3}{2}}(\mathbb{R}^2/\Lambda)^2)$ is given by
\begin{align*}
\mathrm{d}\bfM&[\eta](\delta\eta)(\bfgamma,\bfg) \\
&=
\bfM(\eta)\left(\alpha\langle(\bfM(\eta)(\bfgamma,\bfg)+u\nabla\eta)^\perp\delta\eta\rangle, (\bfM(\eta)(\bfgamma,\bfg)+u\nabla\eta)^\perp\delta\eta\right)-\nabla(u\delta\eta)
+\alpha(\bfM(\eta)(\bfgamma,\bfg)+u\delta\eta)^\perp\delta\eta,
\end{align*}
where
\begin{align*}
u &= \frac{ \nablac \bfg^\perp - \bfM(\eta)(\bfgamma,\bfg) \cdot \nabla\eta }{1+|\nabla\eta|^2} .
\end{align*}
\end{theorem}

\subsection{Taylor expansions}\label{sec:Taylor}
The terms in the expansion
\begin{equation}
H(\eta)=\sum_{k=0}^\infty H_k(\eta), \label{eq:Hexpansion}
\end{equation}
where
$H_k(\eta)$
is homogeneous of degree $k$ in $\eta$, can be calculated recursively from the equation
\begin{align*}
\mathrm{d}H[\eta](\eta)(\bfgamma,\Phi) 
&= H(\eta) \bigg( - \alpha \, \langle \left( \bfK(\eta)(\bfgamma,\Phi) - u \, \nabla\eta \right)^{\perp}\eta \rangle , -\alpha \, \Delta^{-1} \nabla \cdot \left( \bfK(\eta)(\bfgamma,\Phi) - u \, \nabla\eta \right)^{\perp}\eta - u \, \eta + \langle u \, \eta \rangle \bigg)\\
& \qquad\quad\mbox{}- \nabla \cdot \left( ( \bfK(\eta)(\bfgamma,\Phi) - u \, \nabla\eta) \, \eta \right),
\end{align*}
(see Theorem \ref{diff of H}), and the explicit formula
$$
H_0(\bfgamma,\Phi) = D^2 \, \mathtt{t}(D) \, \Phi$$
(see Remark \ref{rem:explicit linearisations}); these results hold under the non-resonance condition (NR). (Note that we
suppress the argument in the $\eta$-independent terms in Taylor series of this kind).

Expanding
\begin{equation}
\bfK(\eta)(\bfgamma,\Phi) = \sum_{k =0}^\infty \bfK_k(\eta)(\bfgamma,\Phi) , \qquad u(\eta)(\bfgamma,\Phi) = \sum_{k=0}^\infty u_k(\eta)(\bfgamma,\Phi), \label{eq:Kuexpansions}
\end{equation}
we find that
\begin{align*}
\bfK_0(\bfgamma,\Phi) &= \bfgamma + \nabla\Phi - \alpha \, \nablap \, \Delta^{-1} H_0(\bfgamma,\Phi) , \\
u_0(\bfgamma,\Phi) &=  H_0(\bfgamma,\Phi) , \\
\intertext{and}
\;\; \bfK_k(\eta)(\bfgamma,\Phi) &= - \alpha \, \nablap \, \Delta^{-1} H_k(\eta)(\bfgamma,\Phi), \\
u_k(\eta)(\bfgamma,\Phi) &\!=\!
\begin{cases}
\displaystyle(-1)^{k/2}|\nabla\eta|^{k} \, H_0(\bfgamma,\Phi) +\!\!\! \sum_{i+2j=k} \!\!\left(  \bfK_{i-1}(\eta)(\bfgamma,\Phi) \cdot \nabla\eta + H_i(\eta)(\bfgamma,\Phi) \right) \, (-1)^j |\nabla\eta|^{2j},
& \mbox{if $k \in 2 \mathbb{N}$,} \\
\displaystyle \sum_{i+2j=k} \left(  \bfK_{i-1}(\eta)(\bfgamma,\Phi) \cdot \nabla\eta + H_i(\eta)(\bfgamma,\Phi) \right) \, (-1)^j |\nabla\eta|^{2j},
& \mbox{if $k \not\in 2 \mathbb{N}$,} \\
\end{cases}
\end{align*}
for $k \geq 1$, and inserting the expansions \eqref{eq:Hexpansion} and \eqref{eq:Kuexpansions} into the formula
for $\mathrm{d}H[\eta](\eta)$ yields
\begin{align*}
\sum_{k \geq 0} & k H_k(\eta)(\bfgamma,\Phi) \\
&= \sum_{k \geq 0} H_k(\eta) \left( - \alpha \, \langle \bfK_0(\bfgamma,\Phi)^{\perp}\eta \rangle , -\alpha \, \Delta^{-1} \nabla \cdot \big( \bfK_0(\bfgamma,\Phi)^{\perp}\eta \big) - H_0(\bfgamma,\Phi)\eta + \langle H_0(\bfgamma,\Phi)\eta \rangle \right) \\
&\qquad\mbox{} + \sum_{k \geq 1} \sum_{j=1}^k H_{k-j}(\eta)
\bigg(\!\! -\alpha \, \langle \left( \bfK_j(\eta)(\bfgamma,\Phi) - u_{j-1}(\eta)(\bfgamma,\Phi) \, \nabla\eta \right)^{\perp} \, \eta \rangle , \\
&\hspace{3.6cm}\mbox{} -\alpha \, \Delta^{-1} \nabla \cdot \big( \left( \bfK_j(\eta)(\bfgamma,\Phi) - u_{j-1}(\eta)(\bfgamma,\Phi) \nabla\eta \right)^{\perp}\eta \big) - u_j(\eta)(\bfgamma,\Phi)\eta + \langle u_j(\eta)(\bfgamma,\Phi)\eta \rangle   \bigg) \\
&\qquad\mbox{}  - \nabla \cdot \big( \bfK_0(\bfgamma,\Phi)\eta \big) - \sum_{k \geq 1} \nabla \cdot \big( \left( \bfK_k(\eta)(\bfgamma,\Phi) - u_{k-1}(\eta)(\bfgamma,\Phi) \, \nabla\eta \right) \, \eta \big) ,
\end{align*}
so that
\begin{align*}
H_1(&\eta)(\bfgamma,\Phi) \\
&= H_0 \bigg( - \alpha \, \langle \bfK_0(\bfgamma,\Phi)^{\perp}\eta \rangle , -\alpha \, \Delta^{-1} \nabla \cdot \bfK_0(\bfgamma,\Phi)^{\perp}\eta - H_0(\bfgamma,\Phi) \, \eta + \langle H_0(\bfgamma,\Phi) \, \eta \rangle \bigg)  - \nabla \cdot \left( \bfK_0(\bfgamma,\Phi)\eta \right) , 
\end{align*}\pagebreak
and
\begin{align*}
H_k&(\eta)(\bfgamma,\Phi) \\
&= \frac{1}{k} \, \bigg\{ H_{k-1}(\eta) \left( - \alpha \, \langle \bfK_0(\bfgamma,\Phi)^{\perp}\eta \rangle , -\alpha \, \Delta^{-1} \nabla \cdot \big( \bfK_0(\bfgamma,\Phi)^{\perp}\eta \big) - H_0(\bfgamma,\Phi)\eta + \langle H_0(\bfgamma,\Phi)\eta \rangle \right) \\
&\quad\qquad\mbox{}+ \sum_{j=0}^{k-2} H_j(\eta) \bigg( -\alpha \, \langle \left( \bfK_{k-1-j}(\eta)(\bfgamma,\Phi) - u_{k-2-j}(\eta)(\bfgamma,\Phi) \, \nabla\eta \right)^{\perp} \, \eta \rangle , \\
&\hspace{3.2cm}\mbox{} -\alpha \, \Delta^{-1} \nabla \cdot \big( \left( \bfK_{k-1-j}(\eta)(\bfgamma,\Phi) - u_{k-2-j}(\eta)(\bfgamma,\Phi) \nabla\eta \right)^{\perp}\eta \big) \\
&\hspace{3.75cm}\mbox{} - u_{k-1-j}(\eta)(\bfgamma,\Phi)\eta + \langle u_{k-1-j}(\eta)(\bfgamma,\Phi)\eta \rangle   \bigg) \\
&\quad\qquad\mbox{}- \nabla \cdot \big( \left( \bfK_{k-1}(\eta)(\bfgamma,\Phi) - u_{k-2}(\eta)(\bfgamma,\Phi) \, \nabla\eta \right) \, \eta \big) \; \bigg\} .
\end{align*}
for $k \geq 2$.

In particular, we find that
\begin{align*}
H_0(\bfgamma,\Phi) &= H_0 \Phi , \\
H_1(\eta)(\bfgamma,\Phi) &= -\alpha \, H_0 \Delta^{-1} \nabla \cdot \bfK_0(\bfgamma,\Phi)^{\perp}\eta - H_0 ( \eta \, H_0\Phi )  - \nabla \cdot \left( \bfK_0(\bfgamma,\Phi)\eta \right) , 
\end{align*}
where $H_0 = D^2 \mathtt{t}(D)$ and $\bfK_0(\bfgamma,\Phi) = \bfgamma + \nabla\Phi - \alpha \, \nablap \Delta^{-1} \, H_0\Phi$, and that
\begin{align*}
H_2&(\eta)(\bfgamma,\Phi) \\
&= \frac{1}{2} \, \bigg\{ H_{1}(\eta) \left( - \alpha \, \langle \bfK_0(\bfgamma,\Phi)^{\perp}\eta \rangle , -\alpha \, \Delta^{-1} \nabla \cdot \left[ \bfK_0(\bfgamma,\Phi)^{\perp}\eta \right] - \eta H_0\Phi + \langle \eta H_0\Phi\rangle \right) \\
&\hspace{11mm}\mbox{}+ H_0 \bigg(   -\alpha \, \Delta^{-1} \nabla \cdot \big( \left( \bfK_{1}(\eta)(\bfgamma,\Phi) - H_0\Phi \nabla\eta \right)^{\perp}\eta \big) - u_{1}(\eta)(\bfgamma,\Phi)\eta  \bigg) \\
&\hspace{11mm}\mbox{}- \nabla \cdot \big( \left( \bfK_{1}(\eta)(\bfgamma,\Phi) - H_0\Phi \, \nabla\eta \right) \, \eta \big)\bigg\} \\
&= \!\frac{1}{2}\bigg\{\!\! -\alpha \, H_0 \Delta^{-1} \nabla \cdot \bigg( \bfK_0 \bigg( - \alpha \, \langle \bfK_0(\bfgamma,\Phi)^{\perp}\eta \rangle , -\alpha \, \Delta^{-1} \nabla \cdot \left( \bfK_0(\bfgamma,\Phi)^{\perp}\eta \right)  - \eta \, H_0\Phi  + \langle \eta \, H_0\Phi  \rangle \bigg)^{\perp} \eta \bigg) \\
&\hspace{11mm}\mbox{} + \alpha \, H_0 \left( \eta \, H_0 \Delta^{-1} \nabla \cdot ( \bfK_0(\bfgamma,\Phi)^{\perp}\eta ) \right) + H_0 ( \eta \, H_0 (\eta \, H_0\Phi) ) \\
&\hspace{11mm}\mbox{} - \nabla \cdot \bigg(  \bfK_0 \left( - \alpha \, \langle \bfK_0(\bfgamma,\Phi)^{\perp}\eta \rangle ,   -\alpha \, \Delta^{-1} \nabla \cdot \left( \bfK_0(\bfgamma,\Phi)^{\perp}\eta \right) - H_0(\bfgamma,\Phi)\eta  \right) \, \eta \bigg) \\
&\hspace{11mm}\mbox{} - \alpha  H_0 \Delta^{-1} \nabla\! \cdot\! \left(\! \eta\!  \left(\! - \alpha  \nablap\!  \Delta^{-1}\! \left( \!-\alpha  H_0 \Delta^{-1} \nabla \cdot \bfK_0(\bfgamma,\Phi)^{\perp}\eta \!-\! H_0 ( \eta  H_0\Phi )  \!-\! \nabla\! \cdot\! \left( \bfK_0(\bfgamma,\Phi)\eta \right) \right) \!-\! H_0\Phi  \nabla\eta  \right)^{\!\perp} \! \right) \\
&\hspace{11mm}\mbox{} - H_0 \left( \eta \, \bfK_0(\bfgamma,\Phi) \cdot \nabla\eta -\alpha\eta \, H_0\Delta^{-1} \nabla \cdot \left( \bfK_0(\bfgamma,\Phi)^{\perp} \eta \right) - \eta \, H_0( \eta \, H_0\Phi) - \eta \, \nabla \cdot \left( \bfK_0(\bfgamma,\Phi)\eta  \right) \right) \\
&\hspace{11mm}\mbox{} +\nabla \cdot \left( \left( \alpha \nablap \Delta^{-1} \left( -\alpha \, H_0 \Delta^{-1} \nabla \cdot \bfK_0(\bfgamma,\Phi)^{\perp}\eta - H_0 ( \eta \, H_0\Phi )  - \nabla \cdot \left( \bfK_0(\bfgamma,\Phi)\eta \right) \right) + H_0\Phi \, \nabla\eta  \right) \, \eta \right) \bigg\} .
\end{align*}

\begin{remark} \label{rem:IrrotCase}

For $\alpha = 0$ we recover the formulae for the classical Dirichlet--Neumann operator, in particular
\begin{align*}
H_0(\bfgamma,\Phi) & = H_0\Phi, \\
H_1(\eta)(\bfgamma,\Phi) &= - H_0( \eta \, H_0\Phi ) - \nabla \cdot (\eta \, \nabla\Phi ), \\
H_2(\eta)(\bfgamma,\Phi) &= H_0( \eta \, H_0( \eta \, H_0\Phi ) ) + \frac{1}{2} \,  H_0(\eta^2 \, \Delta\Phi) + \frac{1}{2} \, \Delta (\eta^2 \, H_0\Phi),
\end{align*}
where $H_0=D \tanh(h D)$.
\end{remark}

Similarly, the terms in the expansion
\begin{equation}
\bfM(\eta)=\sum_{k=0}^\infty \bfM_k(\eta), \label{eq:Mexpansion}
\end{equation}
where $\bfM_k(\eta)$
is homogeneous of degree $k$ in $\eta$, can be calculated recursively from the equation
\begin{align*}
\mathrm{d}\bfM[\eta](\eta)(\bfgamma,\bfg) &= \bfM(\eta) \bigg( \alpha \, \langle  \left( \bfM(\eta)(\bfgamma,\bfg)+ u \, \nabla\eta \right)^{\perp} \, \eta \rangle ,  \left( \bfM(\eta)(\bfgamma,\bfg)+ u \, \nabla\eta \right)^{\perp} \, \eta  \bigg) \\
&\qquad\quad\mbox{}- \nabla  ( u \, \eta) + \alpha \, \left( \bfM(\eta)(\bfgamma,\bfg)+ u \, \nabla\eta \right)^{\perp} \, \eta
\end{align*}
(see Theorem \ref{diff of M}) and the explicit formula
$$
\bfM_0(\bfgamma,\bfg) = -\bfgamma + \frac{1}{D^2} \, \left( \alpha \, \bfD^{\perp} + \bfD \, \mathtt{c}(D) \right) \, \bfD \cdot \bfg^{\perp}
$$
(see Remark \ref{rem:explicit linearisations}).

Expanding
\begin{equation}
u(\bfgamma,\bfg) = \sum_{k=0}^\infty u_k(\eta)(\bfgamma,\bfg), \label{eq:uexpansion}
\end{equation}
we find that
\begin{align*}
u_0(\bfgamma,\bfg) &= \nabla \cdot \bfg^{\perp} \\
\intertext{and}
u_k(\eta)(\bfgamma,\bfg) &=
\begin{cases}
\displaystyle (-1)^{k/2}|\nabla\eta|^{k} \, (\nabla \cdot \bfg^{\perp}) -\!\!\! \sum_{i+2j=k-1} \left(  M_{i}(\eta)(\bfgamma,\bfg) \cdot \nabla\eta \right) \, (-1)^j |\nabla\eta|^{2j},  & \mbox{if $k \in 2 \mathbb{N}$,} \\
\displaystyle- \!\!\!\sum_{i+2j=k-1} \left(  M_{i}(\eta)(\bfgamma,\bfg) \cdot \nabla\eta \right) \, (-1)^j |\nabla\eta|^{2j} , & \mbox{if $k \not\in 2 \mathbb{N}$,} \\\end{cases}
\end{align*}
for $k \geq 1$, and inserting the expansions \eqref{eq:Mexpansion} and \eqref{eq:uexpansion} into the formula
for $\mathrm{d}\bfM[\eta](\eta)$ yields
\begin{align*}
\sum_{k \geq 0} & k \, \bfM_k(\eta)(\bfgamma,\bfg) \\
&= \sum_{k \geq 0} \bfM_k(\eta) \left(  \alpha \, \langle \bfM_0(\bfgamma,\bfg)^{\perp}\eta \rangle ,  \bfM_0(\bfgamma,\bfg)^{\perp}\eta \right) \\
&\quad\qquad\mbox{}+ \sum_{k \geq 1} \, \sum_{j=1}^k \bfM_{k-j}(\eta) \bigg( \alpha \, \langle \left( \bfM_j(\eta)(\bfgamma,\bfg) + u_{j-1}(\eta)(\bfgamma,\bfg) \, \nabla\eta \right)^{\perp} \, \eta \rangle ,    \left(  \bfM_j(\eta)(\bfgamma,\bfg) + u_{j-1}(\eta)(\bfgamma,\bfg) \, \nabla\eta \right)^{\perp} \, \eta   \bigg) \\
&\quad\qquad\mbox{} - \sum_{k \geq 0} \nabla \left( u_k(\eta)(\bfgamma,\bfg)\eta \right) + \alpha \, \bfM_0(\bfgamma,\bfg)^{\perp}\eta + \alpha \, \sum_{k \geq 1}  \left(  \bfM_k(\eta)(\bfgamma,\bfg) + u_{k-1}(\eta)(\bfgamma,\bfg) \, \nabla\eta \right)^{\perp} \, \eta,
\end{align*}
so that
$$
\bfM_1(\eta)(\bfgamma,\bfg) = \bfM_0 \bigg( \alpha \, \langle  \bfM_0(\bfgamma,\bfg)^{\perp} \, \eta \rangle , \bfM_0(\bfgamma,\bfg)^{\perp} \, \eta  \bigg) - \nabla  ( (\nabla \cdot \bfg^{\perp} ) \, \eta) + \alpha \, \bfM_0(\bfgamma,\bfg)^{\perp} \, \eta ,
$$
and
\begin{align*}
& \bfM_k(\eta)(\bfgamma,\bfg) \\
&= \frac{1}{k} \, \bigg\{
\bfM_{k-1}(\eta) \left(  \alpha \, \langle \bfM_0(\bfgamma,\bfg)^{\perp}\eta \rangle ,  \bfM_0(\bfgamma,\bfg)^{\perp}\eta \right) \\
&\quad\qquad\mbox{}+ \sum_{j=0}^{k-2} \bfM_j(\eta) \bigg( \alpha \, \langle \left( \bfM_{k-1-j}(\eta)(\bfgamma,\bfg) + u_{k-2-j}(\eta)(\bfgamma,\bfg) \, \nabla\eta \right)^{\perp} \, \eta \rangle ,    \left(  \bfM_{k-1-j}(\eta)(\bfgamma,\bfg) + u_{k-2-j}(\eta)(\bfgamma,\bfg) \, \nabla\eta \right)^{\perp} \, \eta   \bigg) \\
&\quad\qquad\mbox{}- \nabla \left( u_{k-1}(\eta)(\bfgamma,\bfg)\eta \right) + \alpha \, \left(  \bfM_{k-1}(\eta)(\bfgamma,\bfg) + u_{k-2}(\eta)(\bfgamma,\bfg) \, \nabla\eta \right)^{\perp} \, \eta \bigg\} .
\end{align*}
In particular, we find that
\begin{align*}
&\bfM_2(\eta)(\bfgamma,\bfg) \\
&= \frac{1}{2}  \bigg\{ \bfM_0 \left( \alpha  \langle  \bfM_0( \alpha  \langle \bfM_0(\bfgamma,\bfg)^{\perp}\eta \rangle , \bfM_0(\bfgamma,\bfg)^{\perp}\eta )^{\perp}  \eta \rangle ,  \bfM_0( \alpha  \langle \bfM_0(\bfgamma,\bfg)^{\perp}\eta \rangle , \bfM_0(\bfgamma,\bfg)^{\perp}\eta )^{\perp}  \eta  \right)\\
&\hspace{11mm}+  \nabla  ( \eta  \nabla \cdot ( \bfM_0(\bfgamma,\bfg)\eta ) ) + \alpha  \bfM_0 \left(  \alpha  \langle \bfM_0(\bfgamma,\bfg)^{\perp}\eta \rangle , \bfM_0(\bfgamma,\bfg)^{\perp}\eta \right)^{\!\perp} \! \eta \\
&\hspace{11mm}+ \bfM_0\bigg( \alpha  \left\langle  \eta  \bfM_0 \left( \alpha  \langle  \bfM_0(\bfgamma,\bfg)^{\perp}  \eta \rangle , \bfM_0(\bfgamma,\bfg)^{\perp}  \eta  \right)^{\!\perp} \!\!\! - \eta  \nablap  ( (\nabla \cdot \bfg^{\perp} )  \eta) - \alpha  \bfM_0(\bfgamma,\bfg)  \eta^2 + \eta  ( \nabla \cdot \bfg^{\perp} )  \nablap\eta \right\rangle\! , \\
&\hspace{25mm}\eta  \bfM_0 \left( \alpha  \langle  \bfM_0(\bfgamma,\bfg)^{\perp}  \eta \rangle , \bfM_0(\bfgamma,\bfg)^{\perp}  \eta  \right)^{\!\perp} \!\!\!- \eta  \nablap  ( (\nabla \cdot \bfg^{\perp} )  \eta) - \alpha  \bfM_0(\bfgamma,\bfg)  \eta^2 + \eta  ( \nabla \cdot \bfg^{\perp} )  \nablap\eta   \bigg) \\
&\hspace{11mm}+ \nabla ( \eta  \bfM_0(\bfgamma,\bfg) \cdot \nabla\eta ) \\
&\hspace{11mm}+ \alpha \eta  \bfM_0 \left( \alpha  \langle  \bfM_0(\bfgamma,\bfg)^{\perp}  \eta \rangle , \bfM_0(\bfgamma,\bfg)^{\perp}  \eta  \right)^{\!\perp} \!\!\!- \alpha \eta  \nablap  ( (\nabla \cdot \bfg^{\perp} )  \eta) - \alpha^2  \bfM_0(\bfgamma,\bfg)  \eta^2 + \alpha \eta  ( \nabla \cdot \bfg^{\perp} )  \nablap\eta \bigg\} .
\end{align*}

Finally, the terms in the Taylor expansion
\begin{align*}
\bfT(\eta) &= \sum_{k=0}^\infty \bfT_k(\eta),
\end{align*}
where
$\bfT_k(\eta)$
is homogeneous of degree $k$ in $\eta$, can be computed from the formula
$\bfT(\eta)=\bfM(\eta)(\bfzero,\bfS(\eta))$ using the expansion of $\bfM(\eta)$ derived above and
the corresponding expansion
$$
\bfS(\eta) = \sum_{k=1}^\infty \bfS_k(\eta)
$$
of $\bfS(\eta)$, where
$$
\bfS_k(\eta)\coloneqq \begin{cases}
\displaystyle(-1)^{\frac{k}{2}} \, \frac{\alpha^{k-1} \, \eta^{k}}{k!} \bfc,
& \mbox{if $k\in 2\mathbb{N}$,} \\[3mm]
\displaystyle(-1)^{\frac{k-1}{2}} \, \frac{\alpha^{k-1} \, \eta^{k}}{k!} \bfc^\perp,
& \mbox{if $k \not\in 2\mathbb{N}$.}
\end{cases}
$$

Clearly $\bfT_0=\bfzero$, $\bfT_k(\eta) = \sum\limits_{j=1}^{k} \bfM_{k-j}(\eta) (\bfzero,\bfS_{j}(\eta))$ for $k \geq 1$, and because
$$\bfS_1(\eta)=\eta \bfc^\perp, \quad
\bfS_2(\eta) =  -\dfrac{\alpha}{2}\eta^2  \bfc, \quad \bfS_3(\eta) =  - \dfrac{\alpha^2}{6}\eta^3 \bfc^\perp,$$
and
\begin{align*}
\bfM_0(\bfzero,\bfg) &= \bfL_1\bfg,\\
\bfM_1(\eta)(\bfzero,\bfg) &= - \alpha \, \langle  \eta \,  \bfL_2  \rangle + \bfL_1 \big( \eta \, (\bfL_2\bfg) \big) -\eta \, \bfL \bfg - \nabla\eta \,  \nabla \cdot \bfg^{\perp}, \\
\bfM_2(\eta)(\bfzero,\bfg)
&= \frac{1}{2}  \bigg\{ 2 \alpha  \bfL_1 \big( \langle \eta  \bfL_1\bfg \rangle  \eta  \big) + 2 \alpha^2  \langle \eta  \bfL_1\bfg \rangle  \eta  + 2 \bfL_1 \big( \eta  \bfL_2 (\eta   \bfL_2\bfg )  \big) \\
&\hspace{11mm}-\eta  \bfL ( \eta \bfL_2\bfg) + \nabla\eta  \nablac (  \eta \bfL_1\bfg)   - \bfL_1 \big(  \eta^2  (\bfL \bfg)^\perp   \big) \\
&\hspace{11mm}+ \nabla \big( \eta  \bfL_1\bfg \cdot \nabla \eta \big)  + \alpha  \eta\Big(   \bfL_2 ( \eta  \bfL_2\bfg )  -\eta  (\bfL \bfg)^{\perp}  \Big)  \\
&\hspace{11mm}- \alpha  \Big\langle \Big(   2 \alpha  \langle  \eta  \bfL_1\, \bfg \rangle + 2  \bfL_2 ( \eta  \bfL_2\bfg)  -\eta  (\bfL \bfg)^{\perp}  \Big)  \eta \Big\rangle \bigg\},
\end{align*}
where
\begin{align*}
\bfL \bfg & \coloneqq  \frac{1}{D^2} \, \left( (\alpha^2-D^2) \bfD - \alpha \, \mathtt{c}(D) \, \bfD^{\perp} \right) \, \bfD \cdot \bfg^{\perp}, \\
\bfL_1\bfg &\coloneqq  \frac{1}{D^2} \, \left( \alpha \, \bfD^{\perp} + \bfD \, \mathtt{c}(D) \right) \, \bfD \cdot \bfg^{\perp},\\
\bfL_2\bfg &\coloneqq  \frac{1}{D^2} \, \left( -\alpha \, \bfD +\bfD^\perp \mathtt{c}(D) \right) \, \bfD \cdot \bfg^{\perp},
\end{align*}
we find in particular that
\begin{align*}
\bfT_1(\eta) &= \bfM_0( \bfzero, \bfS_1(\eta) )\\
&=  \bfL_1 ( \eta \, \bfc^{\perp}), \\[4mm]
\bfT_2(\eta) &= \bfM_1(\eta)(\bfzero, \bfS_0(\eta)) + \bfM_0(\eta)(\bfzero, \bfS_1(\eta)) \\
&= - \tfrac{1}{2} \alpha \bfL_1 ( \eta^2 \, \bfc) + \bfL_1 \big( \eta \, \bfL_2 (\eta \bfc^{\perp} )\big) -\eta \, \bfL (\eta \bfc^{\perp}) + \nabla\eta \, \nablac(\eta \bfc) - \alpha \big\langle \eta \, \bfL_2 (\eta \bfc^{\perp})  \big\rangle, \\[4mm]
\bfT_3(\eta) &= \bfM_2(\eta)(\bfzero, \bfS_1(\eta)) + \bfM_1(\eta)(\bfzero, \bfS_2(\eta)) + \bfM_0(\bfzero, \bfS_3(\eta)) \\
&= \frac{1}{2}  \bigg\{ 2 \alpha  \bfL_1 \big( \langle \eta  \bfL_1(\eta\,\bfc^\perp) \rangle  \eta  \big) + 2 \alpha^2  \langle \eta  \bfL_1(\eta\,\bfc^\perp) \rangle  \eta  + 2 \bfL_1 \big( \eta  \bfL_2 (\eta   \bfL_2(\eta\,\bfc^\perp) )  \big) \\
&\hspace{11mm}-\eta  \bfL ( \eta \bfL_2(\eta\,\bfc^\perp)) + \nabla\eta  \nablac (  \eta \bfL_1(\eta\,\bfc^\perp))   - \bfL_1 \big(  \eta^2  (\bfL (\eta\,\bfc^\perp))^\perp   \big) \\
&\hspace{11mm}+ \nabla \big( \eta  \bfL_1(\eta\,\bfc^\perp) \cdot \nabla \eta \big)  + \alpha  \eta\Big(   \bfL_2 ( \eta  \bfL_2(\eta\,\bfc^\perp) )  -\eta  (\bfL (\eta\,\bfc^\perp))^{\perp}  \Big)  \\
&\hspace{11mm}- \alpha  \Big\langle \Big(   2 \alpha  \langle  \eta  \bfL_1\, (\eta\,\bfc^\perp) \rangle + 2  \bfL_2 ( \eta  \bfL_2(\eta\,\bfc^\perp))  -\eta  (\bfL (\eta\,\bfc^\perp))^{\perp}  \Big)  \eta \Big\rangle \bigg\}\\
&\qquad\mbox{}  + \frac{\alpha^2}{2} \, \langle  \eta \bfL_2( \eta^2 \, \bfc )\rangle - \frac{\alpha}{2} \,  \bfL_1( \eta  \bfL_2 (\eta^2 \, \bfc )) +\frac{\alpha}{2} \eta \bfL (\eta^2 \bfc) + \frac{\alpha}{2} \, \nabla\eta \nablac (\eta^2 \bfc)^\perp - \frac{\alpha^2}{6} \,  \bfL_1 (\eta^3 \bfc^{\perp}).
\end{align*}

\section{Description of $H(\eta)$ and $\bfM(\eta)$ as pseudodifferential operators} \label{sec:asympexpNHOp}

\subsection{Flattening and factorisation} \label{subsec:Flatfac}
Choose $\eta \in C^\infty({\mathbb R}^2 / \Lambda)$.
In this section we prove that $H(\eta)$ and $\bfM(\eta)$ are smooth perturbations of properly supported pseudodifferential operators and compute their asymptotic expansions, working under the non-resonance condition (NR).
We begin by introducing a flattening transform (which differs from that used in Section \ref{sec:operators}). Choose $\delta>0$ so that the fluid domain $D_\eta$ contains the strip
\begin{align*}
	\Omega_\delta&\coloneqq  \{ (\bfx,z) \in \mathbb{R}^2 \times \mathbb{R}\colon \eta(\bfx)-\delta  h \leq z < \eta(\bfx) \}
\end{align*}
for $\eta \in U$ and define $\hat\Sigma\colon D_0 \to \Omega_\delta$ by
\begin{align*} 
	\hat\Sigma\colon (\bfx,w) \mapsto (\bfx,\varrho(\bfx,w)), &\; \; \varrho(\bfx,w)\coloneqq \delta  w + \eta(\bfx).
\end{align*}
For $f\colon D_\eta \rightarrow {\mathbb R}$ and $\bfF\colon D_\eta \rightarrow {\mathbb R}^3$ we write $\hat{f}=f \circ \hat\Sigma$, $\hat{\bfF} = \bfF \circ \hat\Sigma$
and use the notation
\begin{align*}
\mathrm{grad}^\varrho \hat{f}(\bfx,w) &\coloneqq (\mathrm{grad} \, f)\circ\hat\Sigma(\bfx,w), \\
\Div^\varrho  \hat{f}(\bfx,w) &\coloneqq  (\Div \, f) \circ \hat\Sigma(\bfx,w),
\\
\curl^\varrho \hat{\bfF}(\bfx,w) &\coloneqq (\curl \, \bfF)\circ\hat\Sigma(\bfx,w),
\\
\Delta^\varrho \hat{f} &\coloneqq  (\Delta \, f) \circ \hat\Sigma(\bfx,w)
\end{align*}
and more generally
$$\partial_x^\varrho \coloneqq  \partial_x - \frac{ \partial_x\eta }{\delta } \partial_w, \quad
\partial_y^\varrho \coloneqq  \partial_y - \frac{ \partial_y\eta }{\delta } \partial_w, \quad
\partial_w^\varrho \coloneqq  \frac{ \partial_w }{\delta }. $$

\begin{remark} \label{rem:flatOprho}

The flattened versions of the operators $\curl$, $\Div$ and $\Delta$ applied to $\hat\bfF(x,y,w)=\bfF(x,y,z)$ and to $\hat{f}(x,y,w)=f(x,y,z)$ are given explicitly by
\begin{align*}
\curl^\varrho \hat\bfF &= (\partial_y^\varrho \hat{F}_3 - \partial_w^\varrho \hat{F}_2, -\partial_x^\varrho \hat{F}_3+\partial_w^\varrho \hat{F}_1,\partial_x^\varrho \hat{F}_2-\partial_y^\varrho \hat{F}_1)^T  \\
&=\curl \hat\bfF - \frac{1}{\delta } (\eta_y \; \partial_w\hat{F}_3,-\eta_x \; \partial_w\hat{F}_3, \eta_x \; \partial_w\hat{F}_2 - \eta_y \; \partial_w\hat{F}_1)^T - \left( \frac{1}{\delta } -1 \right) (\partial_w\hat{F}_2, - \partial_w\hat{F}_1,0)^T,\\ 
\Div^\varrho \hat\bfF &= \partial_x^\varrho \hat{F}_1 + \partial_y^\varrho \hat{F}_2 + \partial_w^\varrho \hat{F}_3 \\
& = \, \partial_x \hat{F}_1 + \partial_y\hat{F}_2 + \frac{1}{\delta} \, \left( -\eta_x \, \partial_w\hat{F}_1  -\eta_y \, \partial_w\hat{F}_2 + \partial_w\hat{F}_3 \right) ,\\ 
-\Delta^\varrho \hat{f} &= -(\partial_x^\varrho)^2 \hat{f} - (\partial_y^\varrho)^2 \hat{f} - (\partial_w^\varrho)^2 \hat{f}\\
&= -\Delta \hat{f} + \frac{2}{\delta } \left( \eta_x \; \partial_{xw}^2\hat{f} + \eta_y \; \partial_{yw}^2\hat{f} \right) - \left( \frac{1}{\delta^2} - 1\right) \partial_w^2\hat{f} + \frac{1}{\delta } ( \eta_{xx}  + \eta_{yy} ) \, \partial_w\hat{f}
- \frac{1}{\delta^2} \; ( \eta_x^2  + \eta_y^2 ) \, \partial_w^2\hat{f} .\\
\end{align*}
\end{remark}

The flattening transform converts the equation
$$-\Delta \bfF = \alpha  \curl \bfF \In \Omega_\delta$$
into
$$
-\Delta^\varrho \hat{\bfF} - \alpha  \curl^\varrho \hat{\bfF} = \bfzero \In D_0,
$$
which is equivalent to the system
\begin{align}
L\hat{\bfF} &=\bfzero,\label{eq:matrix_eq}
\end{align}
where
$L\coloneqq a I \partial_w^2+L_1 \partial_w + L_0$ with
$$L_1\coloneqq \begin{pmatrix} \bfb\cdot\nabla -c & -\frac{\alpha}{\delta} & -\frac{\alpha}{\delta}\eta_y \\
\frac{\alpha}{\delta} & \bfb\cdot\nabla -c & \frac{\alpha}{\delta}\eta_x \\
\frac{\alpha}{\delta}\eta_y & -\frac{\alpha}{\delta}\eta_x &  \bfb\cdot\nabla -c \end{pmatrix}, \qquad
L_0\coloneqq \begin{pmatrix} \Delta & 0 & \alpha \partial_y \\
0 & \Delta & - \alpha \partial_x \\
-\alpha \partial_y & \alpha \partial_x & \Delta \end{pmatrix}
$$
and
\begin{equation*}
a\coloneqq \frac{1+\vert\nabla\eta\vert^2}{\delta^2},\quad \bfb\coloneqq -\frac{2\nabla\eta}{\delta},\quad c\coloneqq \frac{\Delta\eta}{\delta}.
\end{equation*}

\begin{lemma} \label{lem:factorisation}
There are properly supported operators $M$, $N \in \Psi^1({\mathbb R}^2/\Lambda)$ such that
\begin{itemize}
\item[(i)]
$L-a(\partial_w I - N)(\partial_w I -M ) \in \Psi^{-\infty}({\mathbb R}^2/\Lambda)$,
\item[(ii)]
the principal symbols $\mathtt{M}^{(1)}$, $\mathtt{N}^{(1)}$ of $M$, $N$ take the form $\mathtt{M}^{(1)} = \mathtt{m}^{(1)}{\mathbb I}_3$,
$\mathtt{N}^{(1)} = \mathtt{n}^{(1)}{\mathbb I}_3$, where the scalar-valued symbols $\mathtt{m}^{(1)}$, $-\mathtt{n}^{(1)} \in S^1({\mathbb R}^2/\Lambda)$
are strongly elliptic.
\end{itemize}
\end{lemma}
\begin{proof}
Because
$$L-a(\partial_w I - N)(\partial_w I -M )=(L_1+a(M+N))\partial_w + (L_0-aNM)$$
we set
\begin{equation}
N=-a^{-1}L_1-M \label{eq:Defn of N}
\end{equation}
and seek $M$ with
$$L_0+L_1M+aM^2=0$$
by constructing a symbol $\mathtt{M} \in S^1({\mathbb R}^2/\Lambda)$ such that
$$\begin{pmatrix}
-|\bfk|^2 & 0 & \alpha \ii  k_2 \\ 0 & -|\bfk|^2 & -\alpha \ii  k_1 \\ -\alpha \ii  k_2 & \alpha \ii  k_1 & -|\bfk|^2
\end{pmatrix}
+\begin{pmatrix}
\ii \bfb\cdot\bfk - c & - \frac{\alpha}{\delta} & - \frac{\alpha}{\delta}\eta_y \\
\frac{\alpha}{\delta} & \ii \bfb\cdot\bfk - c & \frac{\alpha}{\delta}\eta_x \\
\frac{\alpha}{\delta}\eta_y & - \frac{\alpha}{\delta}\eta_x & \ii \bfb \cdot \bfk - c
\end{pmatrix}\mathtt{M}+(\bfb\cdot\nabla)\mathtt{M}+a \sum_{\bfalpha \in {\mathbb N}_0^2}\partial_{ k_1}^{\alpha_1}\partial_{ k_2}^{\alpha_2}\mathtt{M}\,D_1^{\alpha_1}D_2^{\alpha_2}\mathtt{M} \sim 0$$
and
$$\mathtt{M} \sim \sum_{j \leq 1} \mathtt{M}^{(j)},$$
where $\mathtt{M}^{(j)} \in S^j({\mathbb R}^2/\Lambda)$.

We proceed by computing the terms in the aysmptotic expansion of $\mathtt{M}$ inductively.
\begin{itemize}
\item
Obviously
$$-|\bfk|^2 {\mathbb I}_3 + \ii \bfb \cdot \bfk \mathtt{M}^{(1)} + a (\mathtt{M}^{(1)})^2 = 0,$$
so that
\begin{equation*}
\mathtt{M}^{(1)}=\delta \mathtt{m}^{(1)} \mathbb{I}_3,
\end{equation*}
where 
\begin{align*}
\mathtt{m}^{(1)}(\bfx,\bfk)\coloneqq \frac{ \ii \bfk \cdot \nabla\eta +\lambda^{(1)} }{1+|\nabla\eta|^2},\qquad \lambda^{(1)}(\bfx,\bfk)\coloneqq  \sqrt{(1+|\nabla\eta|^2)|\bfk|^2-(\bfk \cdot \nabla\eta)^2}.
\end{align*}
Note that $\lambda^{(1)}$ is the leading order symbol of the classical Dirichlet--Neumann operator.
\item
The subprincipal symbol of $\mathtt{M}$ is found from the equation
\begin{align*}
\alpha \begin{pmatrix}
0 & 0& \ii k_2\\
0&0&-\ii k_1\\
-\ii k_2& \ii k_1& 0
\end{pmatrix}
&-
\left(c \mathbb{I}_3 + \frac{\alpha}{\delta}\left(\begin{array}{ccc}
0 & 1 &\eta_y\\
-1 & 0& -\eta_x\\
-\eta_y& \eta_x& 0
\end{array}\right)\right)\mathtt{M}^{(1)} \\
&\hspace{-1cm}\mbox{}+\ii \bfb\cdot\bfk \mathtt{M}^{(0)} + a \mathtt{M}^{(0)}\mathtt{M}^{(1)} + a \mathtt{M}^{(1)}\mathtt{M}^{(0)} \\
&\hspace{-1cm}\mbox{}+(\bfb\cdot\nabla)\mathtt{M}^{(1)}-\ii a \partial_{ k_1} \mathtt{M}^{(1)} \partial_x \mathtt{M}^{(1)}-\ii a \partial_{ k_2} \mathtt{M}^{(1)} \partial_y \mathtt{M}^{(1)} = 0,
\end{align*}
which yields
$$
\mathtt{M}^{(0)}=\delta \mathtt{m}^{(0)}\mathbb{I}_3+\delta\mathtt{M}_1^{(0)},
$$
where
\begin{align*}
\mathtt{m}^{(0)}(\bfx,\bfk) &\coloneqq  \frac{1}{2\lambda^{(1)}}\left(\nablac(\mathtt{m}^{(1)}\nabla\eta)+\ii\nabla_{\bfk}\lambda^{(1)}\cdot \nabla \mathtt{m}^{(1)}\right), \\
\mathtt{M}_1^{(0)}(\bfx,\bfk) &\coloneqq  \frac{\alpha}{2\lambda^{(1)}}\left(\begin{array}{ccc}
0 & \mathtt{m}^{(1)} & -\ii k_2+\mathtt{m}^{(1)}\eta_y\\
- \mathtt{m}^{(1)} & 0 & \ii k_1-\mathtt{m}^{(1)}\eta_x \\
\ii k_2 - \mathtt{m}^{(1)}\eta_y & -\ii k_1 + \mathtt{m}^{(1)}\eta_x & 0
\end{array}\right).
\end{align*}
\item
Suppose that $\mathtt{M}^{(j)}$ has been calculated for $j=1,0,\ldots -j_0$ for some $j_0\geq 0$. The term $\mathtt{M}^{(-j_0-1)}$ can be found from the equation
$$(2a\delta \mathtt{m}^{(1)} + \ii \bfb \cdot \bfk)\mathtt{M}^{(-j_0-1)}=\tilde{\mathtt{M}}^{(-j_0)},$$
where $\tilde{\mathtt{M}}^{(-j_0)} \in S^{(-j_0)}({\mathbb R}^2 / \Lambda)$ is given by
\begin{align*}
\tilde{\mathtt{M}}^{(-j_0)}(\bfx,\bfk) & = \left(c \mathbb{I}_3 + \frac{\alpha}{\delta}\left(\begin{array}{ccc}
0 & 1 &\eta_y\\
-1 & 0& -\eta_x\\
-\eta_y& \eta_x& 0
\end{array}\right)\right)\mathtt{M}^{(-j_0)} - \bfb\cdot\nabla \mathtt{M}^{(-j_0)} \\
& \qquad \mbox{}-a \sum_{j_1,j_2 \leq 0 \atop |\bfalpha|+j_1+j_2=j_0} \hspace{-3mm} \partial_{ k_1}^{\alpha_1}\partial_{ k_2}^{\alpha_2}\mathtt{M}^{(-j_1)}D_1^{\alpha_1}D_2^{\alpha_2}\mathtt{M}^{(-j_2)}
-a \sum_{|\bfalpha|=j_0+2} \partial_{ k_1}^{\alpha_1}\partial_{ k_2}^{\alpha_2}\mathtt{M}^{(1)}D_1^{\alpha_1}D_2^{\alpha_2}\mathtt{M}^{(1)}\\
& \qquad \mbox{}-a \sum_{|\bfalpha|=j_0+1} \partial_{ k_1}^{\alpha_1}\partial_{ k_2}^{\alpha_2}\mathtt{M}^{(1)}D_1^{\alpha_1}D_2^{\alpha_2}\mathtt{M}^{(0)}
-a \sum_{|\bfalpha|=j_0+1} \partial_{ k_1}^{\alpha_1}\partial_{ k_2}^{\alpha_2}\mathtt{M}^{(0)}D_1^{\alpha_1}D_2^{\alpha_2}\mathtt{M}^{(1)},
\end{align*}
so that
$$\mathtt{M}^{(-j_0-1)} = \frac{\delta}{2\lambda^{(1)}}\tilde{\mathtt{M}}^{(-j_0)}.$$
\end{itemize}
The construction is completed by noting that there exists a symbol $\mathtt{M} \in S^1({\mathbb R}^2/\Lambda)$ such that
$$\mathtt{M} \sim \sum_{j \leq 1} \mathtt{M}^{(j)}$$
(see Shubin \cite[\S3.3]{Shubin}).

Defining $M=\Op\mathtt{M}$ and $N$ by equation \eqref{eq:Defn of N}, we find that $M$, $N \in \Psi^1({\mathbb R}^2/\Lambda)$.
The terms in the asymptotic expansion
$$\mathtt{N} \sim \sum_{j \leq 1} \mathtt{N}^{(j)}$$
of $N$ are readily computed using \eqref{eq:Defn of N}; in particular we find that
$$\mathtt{N}^{(1)} = \delta \mathtt{n}^{(1)} \mathbb{I}_3, \qquad 
\mathtt{n}^{(1)}(\bfx,\bfk)\coloneqq \frac{ \ii \bfk \cdot \nabla\eta - \lambda^{(1)} }{1+|\nabla\eta|^2}.
$$
Finally, note that
$$\mathrm{Re}\, \mathtt{m}^{(1)}(\bfx,\bfk) = - \mathrm{Re}\, \mathtt{n}^{(1)}(\bfx,\bfk) = \frac{\delta \lambda^{(1)}}{1+|\nabla\eta|^2}
\geq \frac{\delta |\bfk|}{1+\max |\nabla\eta|^2} \gtrsim \langle \bfk \rangle$$
for sufficiently large $|\bfk|$, so that $\mathtt{m}^{(1)}$, $-\mathtt{n}^{(1)}$
are strongly elliptic.
\end{proof}

Theorem \ref{lem:regtobd} below gives information on the Neumann boundary data of a solution to \eqref{eq:matrix_eq}. It is proved using
Lemmata \ref{lem:heat} and \ref{lem:Peetre} below, the former of which is an existence result for an abstract heat equation (see Treves \cite[Ch.\ III \S1]{Treves} for a more general theory).

\begin{lemma} \label{lem:heat}
Suppose that $T>0$, $\Gamma$ is a full rank lattice in ${\mathbb R}^{n-1}$ and $A \in \Psi^m({\mathbb R}^{n-1}/\Gamma)$ for some $m \in {\mathbb N}$
is a properly supported pseudodifferential operator
whose principal symbol $\mathtt{A}^{(m)}$ takes the form $\mathtt{A}^{(m)} = \mathtt{a}^{(m)}{\mathbb I}_n$,
where the scalar-valued symbol $\mathtt{a}^{(m)} \in S^m({\mathbb R}^{n-1}/\Gamma)$ is strongly elliptic.

There is a properly supported pseudodifferential operator $P \in \Psi^{0,m}([T_0,T_0+T];{\mathbb R}^{n-1}/\Gamma)$ which satisfies
\begin{align*}
& \partial_tP+ AP \in \Psi^{-\infty}([T_0,T_0+T];{\mathbb R}^{n-1}/\Gamma), \\
& P|_{t=T_0} = I.
\end{align*}
In particular, any solution of the initial-value problem
\begin{align*}
& \partial_t\hat{\bfU} + A\hat{\bfU} = \hat{\bfF}, \qquad t \in [T_0,T_0+T], \\
& \hat{\bfU}|_{t=T_0} = \hat{\bfU}_0,
\end{align*}
where $\hat{\bfF} \in C^\infty([T_0,T_0+T];C^\infty({\mathbb R}^{n-1}/\Gamma)^n)$ and $\hat{\bfU}_0 \in C^\infty({\mathbb R}^{n-1}/\Gamma)^n$,
belongs to $C^\infty([T_0,T_0+T];C^\infty({\mathbb R}^{n-1}/\Gamma)^n)$.
\end{lemma}

\begin{lemma} \label{lem:Peetre}
Suppose that $T>0$, $\Gamma$ is a full rank lattice in ${\mathbb R}^{n-1}$ and $\mathcal{P}$ is a linear differential operator of order $m$ in the variables
$(\bfz,t) \in {\mathbb R}^{n}$ of the form
$$\mathcal{P}= {\mathbb I}_n \partial_t^m + \sum_{|\bfalpha| \leq m \atop \alpha_n \leq m-1} A_{\bfalpha}(\bfz) \partial^{\bfalpha},$$
where $\bfalpha \in {\mathbb N}_0^n$, $\partial^{\bfalpha}=\partial_{z_1}^{\alpha_1},\ldots,\partial_{z_{n-1}}^{\alpha_{n-1}}\partial_t^{\alpha_n}$ and
the coefficients of the matrix $A_{\bfalpha}(\bfz)$ are functions of $\bfz$ of class
$C^\infty({\mathbb R}^{n-1}/\Gamma)$.
Any solution $\hat{\bfU} \in H^{m-1}({\mathbb R}^{n-1}/\Gamma \times (T_0,T_0+T))^n$ of 
$\mathcal{P}\hat{\bfU}=\bfzero$ lies in $C^\infty([T_0,T_0+T]; {\mathscr D}^\prime({\mathbb R}^{n-1}/\Gamma)^n)$.
\end{lemma}
\begin{proof}
Suppose that $\mathcal{P}\hat{\bfU}=\bfzero$. A straightforward argument using Fubini's theorem shows that $\hat{\bfU}$ lies in\linebreak
$H^{m-1}((T_0,T_0+T); L^2({\mathbb R}^{n-1}/\Gamma)^n)$, and the next step
is to show inductively that $\hat{\bfU}$ in fact lies in\linebreak
$H^{m-1+k}((T_0,T_0+T); H^{-km}({\mathbb R}^{n-1}/\Gamma)^n)$ for every $k \in {\mathbb N}_0$.

To this end let $\ell_1 \in {\mathbb N}_0$, ${\ell_2} \in {\mathbb Z}$ and $a \in C^\infty({\mathbb R}^{n-1}/\Gamma)$, and observe that 
the mappings $w \mapsto \partial_t w$ and  $w \mapsto \partial_{z_j} w$\linebreak induce continuous linear operators $H^{\ell_1}((T_0,T_0+T);H^{\ell_2}({\mathbb R}^{n-1}/\Gamma)) \rightarrow H^{{\ell_1}-1}((T_0,T_0+T);H^{\ell_2}({\mathbb R}^{n-1}/\Gamma))$ 
and $H^{\ell_1}((T_0,T_0+T); H^{{\ell_2}}({\mathbb R}^{n-1}/\Gamma)) \rightarrow H^{\ell_1}((T_0,T_0+T); H^{{\ell_2}-1}({\mathbb R}^{n-1}/\Gamma))$ respectively, while the mapping
$w \mapsto a w$ induces a continuous linear operator $H^{\ell_1}((T_0,T_0+T); H^{{\ell_2}}({\mathbb R}^{n-1}/\Gamma)) \rightarrow H^{\ell_1}((T_0,T_0+T); H^{{\ell_2}}({\mathbb R}^{n-1}/\Gamma))$. It follows that the formula
$$\mathcal{Q}\hat{\bfU} = -\sum_{|\bfalpha| \leq m \atop \alpha_n \leq m-1} A_{\bfalpha}(\bfz) \partial^{\bfalpha}\hat{\bfU}$$
defines a continuous linear operator
$H^{m-1+k}((T_0,T_0+T);H^{-km}({\mathbb R}^{n-1}\!/\Gamma)^n)\! \rightarrow\! H^{k}((T_0,T_0+T);H^{-(k+1)m}({\mathbb R}^{n-1}\!/\Gamma)^n)$
for each $k \in {\mathbb N}_0$.

Returning to the induction, let $k \in {\mathbb N}_0$ and suppose that $\hat{\bfU} \in H^{m-1+k}((T_0,T_0+T); H^{-km}({\mathbb R}^{n-1}/\Gamma)^n)$ satisfies $P\hat{\bfU}=0$,
so that $\partial_t^m \hat{\bfU} = \mathcal{Q}\hat{\bfU}$ in ${\mathscr D}^\prime((T_0,T_0+T);H^{-(k+1)m}({\mathbb R}^{n-1}/\Gamma)^n)$. 
The above argument implies that\linebreak
$\partial_t^m \hat{\bfU} \in H^k((T_0,T_0+T);H^{-(k-1)m}({\mathbb R}^{n-1}/\Gamma)^n)$,
and since
$$\hat{\bfU} \in H^{m-1+k}((T_0,T_0+T);H^{-km}({\mathbb R}^{n-1}/\Gamma)^n) \subseteq H^{m-1+k}((T_0,T_0+T);H^{-(k+1)m}({\mathbb R}^{n-1}/\Gamma)^n),$$
we conclude that $\hat{\bfU} \in H^{m+k}((T_0,T_0+T);H^{-(k+1)m}({\mathbb R}^{n-1}/\Gamma)^n)$.

Finally, choose $\ell \in {\mathbb N}$ with $\ell \geq m-2$ and set $k=\ell-m+2$, so that
\begin{align*}
\hat{\bfU} & \in H^{\ell+1}((T_0,T_0+T); H^{-(\ell-m+2)m}({\mathbb R}^{n-1}/\Gamma)^n)\\
&\subseteq C^\ell([T_0,T_0+T]; H^{-(\ell-m+2)m}({\mathbb R}^{n-1}/\Gamma)^n) \\
& \subseteq C^\ell([T_0,T_0+T]; {\mathscr D}^\prime({\mathbb R}^{n-1}/\Gamma)^n).
\end{align*}
However this result holds for arbitrarily large $\ell \in {\mathbb N}$, so that $\hat{\bfU} \in C^\infty([T_0,T_0+T]; {\mathscr D}^\prime({\mathbb R}^{n-1}/\Gamma)^n)$.
\end{proof}

\begin{theorem} \label{lem:regtobd}
Any function $\hat{\bfU} \in H^2(D_0 / \Lambda)^3$ with $L\hat{\bfU}=0$ in $D_0$ satisfies
$$\partial_w \hat{\bfU} = M\hat{\bfU} + R_\infty \hat{\bfU}\at w=0,$$
where the symbol $R_\infty$ denotes a linear function of its argument whose range lies in $C^\infty({\mathbb R^2}/\Lambda)^3$.
\end{theorem}
\begin{proof}
The equation
\begin{equation}
L \hat{\bfU} = \bfzero \label{eq:Full eqn}
\end{equation}
is equivalent to the coupled equations
\begin{align}
(\partial_w I - M) \hat{\bfU} & = \hat{\bfU}_1, \label{eq:fac1} \\
(\partial_w I - N) \hat{\bfU}_1 & = -R_\infty \hat{\bf U} \label{eq:fac2}
\end{align}
(the smoothing operator in equation \eqref{eq:fac2} in fact lies in $\Psi^{-\infty}({\mathbb R}^2/\Lambda)$).

By elliptic regularity theory $\hat{\bfU} \in C^\infty(D_0/\Lambda)^3 \cong C^\infty((-h,0);C^\infty({\mathbb R}^2/\Lambda)^3)$, and it follows from equation \eqref{eq:fac1} that
$\hat{\bfU}_1 \in C^\infty((-h,0);C^\infty({\mathbb R}^2/\Lambda)^3)$; in particular $\hat{\bfU}_1|_{w=-\frac{1}{2}h} \in C^\infty({\mathbb R}^2/\Lambda)^3$. Furthermore, applying Lemma \ref{lem:Peetre} to \eqref{eq:Full eqn} shows that
$\hat{\bfU} \in C^\infty([-h,0]; {\mathscr D}^\prime({\mathbb R}^2/\Lambda)^3)$, so that $R_\infty\hat{\bfU} \in C^\infty([-h,0];C^\infty({\mathbb R}^2/\Lambda)^3)$.
Applying Lemma \ref{lem:heat} to equation \eqref{eq:fac2} for $w\in[-\frac{1}{2}h,0]$, we thus find that
$\hat{\bfU}_1 \in C^\infty([-\frac{1}{2}h,0]; C^\infty({\mathbb R}^2/\Lambda)^3)$. Finally, equation \eqref{eq:fac1} shows that
$$\partial_w \hat{\bfU} = M\hat{\bfU} + \hat{\bfU}_1\at w=0$$
because $\hat{\bfU}_1$ is a linear function of $\hat{\bfU}$.
\end{proof}

\subsection{The operator $H(\eta)$} \label{subsec:asympH}
Let $s \geq 2$, $\Phi \in \mathring{H}^{s-\frac{1}{2}}({\mathbb R}^2 / \Lambda)$ and $\tilde{\bfA} \in H^s(D_0/\Lambda)^3$ be
the unique solution of the boundary-value problem \eqref{Flattened A BVP 1}--\eqref{Flattened A BVP 5} with $\bfgamma=\bfzero$. The variable
$$\hat{\bfA}(\bfx, w) \coloneqq \tilde{\bfA}(\bfx,v),\qquad w \coloneqq \frac{1}{\delta h}(h+\eta)v$$
satisfies 
$$
L\hat{\bfA}= \bfzero \In D_0,
$$
and
\begin{alignat*}{2}
\hat{\bfA}\cdot \bfN&=0 & & \at  w=0, \\
(\text{curl}^\varrho\hat{\bfA})_\parallel&=\nabla\Phi-\alpha\nablap\Delta^{-1}(\nablac\hat{\bfA}_\parallel^\perp) & & \at  w=0,
\end{alignat*}
which can be written explicitly as
\begin{alignat}{2}
&\eta_x\hat{A}_1+\eta_y\hat{A}_2-\hat{A}_3=0 & & \at  w=0, \label{eq:bdry1}\\
&\hat{A}_{3y}-\frac{\eta_y}{\delta}\partial_w\hat{A}_3-\frac{1}{\delta}\partial_w\hat{A}_2+\eta_x\left(\hat{A}_{2x}-\hat{A}_{1y}-\frac{1}{\delta}\eta_x\partial_w\hat{A}_2+\frac{1}{\delta}\eta_y\partial_w\hat{A}_1\right) \label{eq:bdry2}\\
&\qquad\mbox{}+\alpha\Delta^{-1}(\hat{A}_{2xy}+\eta_{yy}\hat{A}_{3x}+\eta_y\hat{A}_{3xy}-\hat{A}_{1yy}-\eta_{xy}\hat{A}_{3y}-\eta_x\hat{A}_{3yy})=\Phi_x & & \at  w=0, \nonumber\\
&-\hat{A}_{3x}+\frac{\eta_x}{\delta}\partial_w\hat{A}_3+\frac{1}{\delta}\partial_w\hat{A}_1+\eta_y\left(\hat{A}_{2x}-\hat{A}_{1y}-\frac{1}{\delta}\eta_x\partial_w\hat{A}_2+\frac{1}{\delta}\eta_y\partial_w\hat{A}_1\right)  \label{eq:bdry3} \\
&\qquad\mbox{}+\alpha\Delta^{-1}(-\hat{A}_{2xx}-\eta_{xy}\hat{A}_{3x}-\eta_y\hat{A}_{3xx}+\hat{A}_{1xy}+\eta_{xx}\hat{A}_{3y}+\eta_x\hat{A}_{3xy})=\Phi_y & & \at  w=0.\nonumber
\end{alignat}

Substituting
\begin{equation}
\hat{A}_3=\eta_x\hat{A}_1+\eta_y\hat{A}_2 \label{eq:eqn for A3}
\end{equation}
(see equation \eqref{eq:bdry1}) and
\begin{equation}
\partial_w \hat{\bfA}|_{w=0} = M\hat{\bfA}|_{w=0} + R_\infty\Phi \label{eq:Hpseudo 1}
\end{equation}
(see Lemma \ref{lem:regtobd}, noting that $\hat{\bfA}$ is a linear function of $\Phi$) into equations \eqref{eq:bdry2}, \eqref{eq:bdry3}, we find that
\begin{equation}
P\begin{pmatrix} \hat{A}_1|_{w=0} \\ \hat{A}_2|_{w=0} \end{pmatrix} = \begin{pmatrix}  \Phi_x \\ \Phi_y \end{pmatrix} + R_\infty\Phi,
\label{eq:eqn for P (A)}
\end{equation}
where $P \in \Psi^1({\mathbb R}^2 / \Lambda)$ is a properly supported pseudodifferential operator with principal symbol
$$\mathtt{P}^{(1)}(\bfx,\bfk) = \begin{pmatrix}
0 & -(1+|\nabla\eta|^2)\mathtt{m}^{(1)} +\ii \bfk \cdot \nabla\eta  \\
(1+|\nabla\eta|^2)\mathtt{m}^{(1)} - \ii \bfk \cdot \nabla\eta & 0
\end{pmatrix}.$$
Observe that $\mathtt{P}^{(1)}$ is invertible for $|\bfk| \neq 0$,
so that $P$ is elliptic and hence admits a parametrix $Q \in \Psi^{-1}({\mathbb R}^2/\Lambda)$ such that $PQ-I \in \Psi^{-\infty}({\mathbb R}^2/\Lambda)$
(see Grubb \cite[Theorem 7.18]{Grubb}).
We thus find from equation \eqref{eq:eqn for P (A)} that
$$
\begin{pmatrix} \hat{A}_1|_{w=0} \\ \hat{A}_2|_{w=0} \end{pmatrix} = Q\begin{pmatrix} \Phi_x \\ \Phi_y \end{pmatrix} + R_\infty\Phi,
$$
and appending \eqref{eq:eqn for A3} to this equation yields
\begin{equation}
\hat{\bfA}|_{w=0} = Z \Phi + R_\infty\Phi, \label{eq:Hpseudo 2}
\end{equation}
where $\mathtt{Z} \in S^0({\mathbb R}^2 / \Lambda)$ and $Z=\Op \mathtt{Z}$.

We have that 
$$
H(\eta)(\bfgamma,\Phi)=\underbrace{H(\eta)(\bfgamma,0)}_{\displaystyle \in C^\infty({\mathbb R}^2/\Lambda)}+H(\eta)(\bfzero,\Phi),
$$
and in the new coordinates 
\begin{equation}
H(\eta)(\bfzero,\Phi) = \hat{A}_{2x}+\eta_y\hat{A}_{3x}-\hat{A}_{1y}-\eta_x\hat{A}_{3y}\big\vert_{w=0}. \label{eq:def_GDNO}
\end{equation}
Inserting $\hat{\bfA}|_{w=0}$ and $\partial_w\hat{\bfA}|_{w=0}$ from \eqref{eq:Hpseudo 1}, \eqref{eq:Hpseudo 2}
into this formula shows that
$$H(\eta)(\bfzero,\Phi)=\Op\lambda_{\alpha}\Phi+R_\infty\Phi,$$
where $\lambda_{\alpha} \in S^1({\mathbb R}^2/\Lambda)$. The asymptotic expansions
$$\mathtt{Z} \sim \sum_{j \leq 0} \mathtt{Z}^{(j)}, \qquad \lambda_\alpha \sim\sum_{j \leq 1} \lambda_\alpha^{(j)}$$
can be determined recursively by substituting
$$\hat{\bfA}|_{w=0} = Z \Phi + R_\infty\Phi, \qquad \partial_w\hat{\bfA}|_{w=0} = M Z \Phi + R_\infty\Phi$$
into
\eqref{eq:bdry1}--\eqref{eq:bdry3}.

\subsubsection{Principal symbol} \label{subsubsec:leadingGDNOp}
Equating the order $0$ terms in \eqref{eq:bdry1} and order $1$ terms in \eqref{eq:bdry2}, \eqref{eq:bdry3} yields the equations
\begin{align}
\eta_x\mathtt{Z}^{(0)}_1+\eta_y\mathtt{Z}^{(0)}_2-\mathtt{Z}^{(0)}_3&=0, \label{eq:bdry11}\\
(\ii k_2-\eta_y\mathtt{m}^{(1)})\mathtt{Z}^{(0)}_3+(-\mathtt{m}^{(1)}+\ii k_1\eta_x-\eta_x^2\mathtt{m}^{(1)})\mathtt{Z}^{(0)}_2+(-\ii k_2\eta_x+\eta_x\eta_y\mathtt{m}^{(1)})\mathtt{Z}^{(0)}_1&=\ii k_1, \label{eq:bdry12}\\
(-\ii k_1+\eta_x\mathtt{m}^{(1)})\mathtt{Z}^{(0)}_3+(\mathtt{m}^{(1)}-\ii k_2\eta_y+\eta_y^2\mathtt{m}^{(1)})\mathtt{Z}^{(0)}_1+(\ii k_1\eta_y-\eta_x\eta_y\mathtt{m}^{(1)})\mathtt{Z}^{(0)}_2&=\ii k_2. \label{eq:bdry13}
\end{align}
Substituting for $\mathtt{Z}^{(0)}_3$ from \eqref{eq:bdry11} into \eqref{eq:bdry12}, one finds that
$$
(\ii k_2-\eta_y\mathtt{m}^{(1)})(\eta_x\mathtt{Z}^{(0)}_1+\eta_y\mathtt{Z}^{(0)}_2)+(-\mathtt{m}^{(1)}+\ii k_1\eta_x-\eta_x^2\mathtt{m}^{(1)})\mathtt{Z}^{(0)}_2+(-\ii k_2\eta_x+\eta_x\eta_y\mathtt{m}^{(1)})\mathtt{Z}^{(0)}_1=\ii k_1,
$$
so that
$$
\underbrace{(-\mathtt{m}^{(1)}(1+|\nabla\eta|^2)+\ii\bfk \cdot \nabla\eta)}_{\displaystyle=-\lambda^{(1)}}\mathtt{Z}^{(0)}_2=\ii k_1
$$
and hence
\begin{align*}
\mathtt{Z}^{(0)}_2(\bfx,\bfk) &= -\frac{\ii k_1}{\lambda^{(1)}}.
\end{align*}
Similarly, substituting for $\mathtt{Z}^{(0)}_3$ from \eqref{eq:bdry11} into \eqref{eq:bdry13} yields
\begin{align*}
\mathtt{Z}^{(0)}_1(\bfx,\bfk) &= \frac{\ii k_2}{\lambda^{(1)}},
\end{align*}
and it follows from \eqref{eq:bdry11} that
\begin{align*}
\mathtt{Z}^{(0)}_3(\bfx,\bfk) &= -\frac{\ii(\bfk \cdot \nabla^{\perp}\eta)}{\lambda^{(1)}}.
\end{align*}

Equating terms of order $1$ in equation \eqref{eq:def_GDNO}, we find that
\begin{align*}
\lambda_\alpha^{(1)}(\bfx,\bfk) &=
\ii k_1\mathtt{Z}^{(0)}_2+\eta_y\ii k_1\mathtt{Z}^{(0)}_3-\ii k_2\mathtt{Z}^{(0)}_1-\eta_x\ii k_2\mathtt{Z}^{(0)}_3 \\
&= \ii(\bfk \cdot \nabla^{\perp}\eta)\mathtt{Z}^{(0)}_3+\ii k_1\mathtt{Z}^{(0)}_2-\ii k_2\mathtt{Z}^{(0)}_1 \\
&= \frac{1}{\lambda^{(1)}}\underbrace{((\bfk \cdot \nabla^{\perp}\eta)^2+|\bfk|^2)}_{\displaystyle=(\lambda^{(1)})^2} \\
&= \lambda^{(1)};
\end{align*}
the principal symbol of the generalised Dirichlet--Neumann operator is thus the same as the principal symbol of the classical Dirichlet--Neumann operator.

\subsubsection{Sub-principal symbol} \label{subsubsec:subleadingGDNOp}
Equating the order $-1$ terms in \eqref{eq:bdry1} and the order $0$ terms in \eqref{eq:bdry2}, \eqref{eq:bdry3} yields the equations
\begin{align}
\eta_x\mathtt{Z}^{(-1)}_1+\eta_y\mathtt{Z}^{(-1)}_2-\mathtt{Z}^{(-1)}_3&=0, \label{eq:bdry10}\\
(\ii k_2-\eta_y\mathtt{m}^{(1)})\mathtt{Z}^{(-1)}_3+(-\mathtt{m}^{(1)}+\ii k_1\eta_x-\eta_x^2\mathtt{m}^{(1)})\mathtt{Z}^{(-1)}_2+(-\ii k_2\eta_x+\eta_x\eta_y\mathtt{m}^{(1)})\mathtt{Z}^{(-1)}_1&=-\mathtt{F}_1-\alpha \mathtt{F}_3, \label{eq:bdry20} \\
(-\ii k_1+\eta_x\mathtt{m}^{(1)})\mathtt{Z}^{(-1)}_3+(\mathtt{m}^{(1)}-\ii k_2\eta_y+\eta_y^2\mathtt{m}^{(1)})\mathtt{Z}^{(-1)}_1+(\ii k_1\eta_y-\eta_x\eta_y\mathtt{m}^{(1)})\mathtt{Z}^{(-1)}_2&=-\mathtt{F}_2-\alpha \mathtt{F}_4, \label{eq:bdry30}
\end{align}
where
\begin{align*}
\mathtt{F}_1(\bfx,\bfk)&=\eta_{xy}\mathtt{Z}^{(0)}_1+\eta_{yy}\mathtt{Z}^{(0)}_2+\nabla \mathtt{Z}^{(0)}_2\cdot\nabla\eta+[\eta_y(\mathtt{Z}^{(0)}_1\nabla\eta_x+\mathtt{Z}^{(0)}_2\nabla\eta_y)+(1+|\nabla\eta|^2)\nabla \mathtt{Z}^{(0)}_2]\cdot\ii\nabla_{\bfk}\mathtt{m}^{(1)}\\
&\quad -(1+|\nabla\eta|^2)\mathtt{Z}^{(0)}_2m^{(0)},\\
\mathtt{F}_2(\bfx,\bfk)&=-\eta_{xx}\mathtt{Z}^{(0)}_1-\eta_{xy}\mathtt{Z}^{(0)}_2-\nabla\eta\cdot\nabla \mathtt{Z}^{(0)}_1-[\eta_x(\mathtt{Z}^{(0)}_1\nabla\eta_x+\mathtt{Z}^{(0)}_2\nabla\eta_y)+(1+|\nabla\eta|^2)\nabla \mathtt{Z}^{(0)}_1]\cdot\ii\nabla_{\bfk}\mathtt{m}^{(1)}\\
&\quad+(1+|\nabla\eta|^2)\mathtt{Z}^{(0)}_1m^{(0)},\\
\mathtt{F}_3(\bfx,\bfk)&=\frac{1}{2}\left((1+\eta_x^2)\mathtt{Z}_1^{(0)}+\eta_x\eta_y\mathtt{Z}_2^{(0)}\right)-\frac{ k_2}{|\bfk|^2}\left(\left( k_2-\eta_x(\bfk\cdot\nabla\eta^\perp)\right)\mathtt{Z}_1^{(0)}
-\left( k_1+\eta_y(\bfk\cdot\nabla\eta^\perp)\right)\mathtt{Z}_2^{(0)}\right),\\
\mathtt{F}_4(\bfx,\bfk)&=\frac{1}{2}\left(\eta_x\eta_y\mathtt{Z}_1^{(0)}+(1+\eta_y^2)\mathtt{Z}_2^{(0)}\right)-\frac{ k_1}{|\bfk|^2}\left(\left(- k_2+\eta_x(\bfk\cdot\nabla\eta^\perp)\right)\mathtt{Z}_1^{(0)} +\left( k_1+\eta_y(\bfk\cdot\nabla\eta^\perp)\right)\mathtt{Z}_2^{(0)}\right).
\end{align*}
Substituting for $\mathtt{Z}^{(-1)}_3$ from \eqref{eq:bdry10} into \eqref{eq:bdry20}--\eqref{eq:bdry30}, we obtain
\begin{equation*}
\mathtt{Z}^{(-1)}_1(\bfx,\bfk)=-\frac{\mathtt{F}_2+\alpha \mathtt{F}_4}{\lambda^{(1)}},\qquad \mathtt{Z}^{(-1)}_2(\bfx,\bfk)=\frac{\mathtt{F}_1+\alpha \mathtt{F}_3}{\lambda^{(1)}}
\end{equation*}
and hence
$$\mathtt{Z}^{(-1)}_3(\bfx,\bfk)=\frac{1}{\lambda^{(1)}}\nabla\eta^\perp\cdot(\mathtt{F}_1+\alpha \mathtt{F}_3,\mathtt{F}_2+\alpha \mathtt{F}_4)^T.$$

Equating terms of order $1$ in equation \eqref{eq:def_GDNO}, we find that
\begin{align*}
\lambda_\alpha^{(0)}(\bfx,\bfk) & =
\partial_x\mathtt{Z}^{(0)}_2+\ii k_1\mathtt{Z}^{(-1)}_2+\eta_y\partial_x\mathtt{Z}^{(0)}_3+\ii k_1\eta_y\mathtt{Z}^{(-1)}_3-\partial_y\mathtt{Z}^{(0)}_1-\ii k_2\mathtt{Z}^{(-1)}_1-\eta_x\partial_y\mathtt{Z}^{(0)}_3-\ii k_2\eta_x\mathtt{Z}^{(-1)}_3 \nonumber \\
&=\partial_x\mathtt{Z}^{(0)}_2-\partial_y\mathtt{Z}^{(0)}_1+\eta_y\partial_x\mathtt{Z}^{(0)}_3-\eta_x \partial_y\mathtt{Z}^{(0)}_3+\frac{\ii k_1+\ii\eta_y( k\cdot\nablap\eta)}{\lambda^{(1)}}\mathtt{F}_1-\frac{-\ii k_2+\ii\eta_x( k\cdot\nablap\eta)}{\lambda^{(1)}}\mathtt{F}_2\nonumber \\
&\qquad\mbox{} +\alpha\left[\frac{\ii k_1+\ii\eta_y( \bfk\cdot\nablap\eta)}{\lambda^{(1)}}\mathtt{F}_3-\frac{-\ii k_2+\ii\eta_x( \bfk\cdot\nablap\eta)}{\lambda^{(1)}}\mathtt{F}_4\right] \nonumber \\
&=\lambda^{(0)}+\alpha\frac{(\bfk\cdot\nabla\eta)(\bfk\cdot\nablap\eta)}{|\bfk|^2},
\end{align*}
where
\begin{align*}
\lambda^{(0)}(\bfx,\bfk)&\coloneqq  \frac{1+|\nabla\eta|^2}{2\lambda^{(1)}} \left(\nablac (\mathtt{m}^{(1)} \, \nabla\eta )+ \ii \nabla_{\bfk} \lambda^{(1)}\cdot \nabla \mathtt{m}^{(1)}  \right) 
\end{align*}
is the sub-principal symbol of the classical Dirichlet--Neumann operator (see Alazard, Burq and Zuily
\cite[Eq.\ (3.11)]{AlazardBurqZuily11}).

\subsection{The operator $\bfM(\eta)$}  \label{subsec:asympM}
Let $s \geq 2$, $\bfg \in H^{s-\frac{1}{2}}({\mathbb R}^2 / \Lambda)^2$ and $\hat{\bfB} \in H^s(D_0/\Lambda)^3$ be
the unique solution of the boundary-value problem \eqref{Flattened B BVP 1}--\eqref{Flattened B BVP 6} with $\bfgamma=\bfzero$. The variable
$$\hat{\bfB}(\bfx, w) \coloneqq \tilde{\bfB}(\bfx,v),\qquad w \coloneqq \frac{1}{\delta h}(h+\eta)v$$
satisfies 
$$
L\hat{\bfB}= \bfzero \In D_0,
$$
and
\begin{alignat}{2}
\Div^{\varrho} \hat{\bfB} &=0 & & \at w=0, \label{bddryB1}\\
\hat{\bfB}\cdot\bfN&=0 & & \at w=0, \label{bddryB2}\\
\nablac \hat{\bfB}_{||}^\perp &=\nablac \bfg^\perp & & \at w=0,\label{bddryB3}\\
\langle (\curl^{\varrho} \hat{\bfB})_\parallel \rangle &=\bfzero.\label{bddryB4}
\end{alignat}
(equation \eqref{bddryB1} actually holds in $\bar{D}_0$). The boundary conditions \eqref{bddryB1}--\eqref{bddryB3} can be written more explicitly as
\begin{alignat}{2}
& \hat{B}_{1x}+\hat{B}_{2y}+\frac{1}{\delta}(-\eta_x\hat{B}_{1w}-\eta_y\hat{B}_{2w}+\hat{B}_{3w})=0 & & \at w=0, \label{eq:bdry1_B} \\
& \eta_x\hat{B}_1+\eta_y\hat{B}_2-\hat{B}_3=0 & & \at w=0, \label{eq:bdry2_B} \\
& (\hat{B}_2+\hat{B}_3\eta_y)_x-(\hat{B}_1+\hat{B}_3\eta_x)_y=g_{2x}-g_{1y} & & \at w=0, \label{eq:bdry3_B}
\end{alignat}

Substituting
\begin{equation}
\hat{B}_3=\eta_x\hat{B}_1+\eta_y\hat{B}_2 \label{eq:eqn for B3}
\end{equation}
(see equation \eqref{eq:bdry2_B}) and
\begin{equation}
\partial_w \hat{\bfB}|_{w=0} = M\hat{\bfB}|_{w=0} + R_\infty\bfg\label{eq:Mpseudo 1}
\end{equation}
(see Lemma \ref{lem:regtobd}, noting that $\bfB$ is a linear function of $\bfg$) into equations \eqref{eq:bdry1_B}, \eqref{eq:bdry3_B}, we find that
\begin{equation}
P\begin{pmatrix} \hat{B}_1|_{w=0} \\ \hat{B}_2|_{w=0} \end{pmatrix} = \begin{pmatrix} 0 \\ g_{2x}-g_{1y} \end{pmatrix} + R_\infty\bfg,
\label{eq:eqn for P (B)}
\end{equation}
where $P \in \Psi^1({\mathbb R}^2 / \Lambda)$ is a properly supported pseudodifferential operator with principal symbol
$$\mathtt{P}^{(1)}(\bfx,\bfk) = \begin{pmatrix}
\ii  k_1 & \ii  k_2 \\
\ii\eta_x\eta_y k_1-\ii(1+\eta_x^2) k_2 & \ii(1+\eta_y^2) k_1 -\ii\eta_x\eta_y k_2
\end{pmatrix}.$$
Observe that $\mathtt{P}^{(1)}(\bfx,\bfk)$ is invertible for $|\bfk| \neq 0$,
so that $P$ is elliptic and hence admits a parametrix $Q \in \Psi^{-1}({\mathbb R}^2/\Lambda)$ such that $PQ-I \in \Psi^{-\infty}({\mathbb R}^2/\Lambda)$.
We thus find from equation \eqref{eq:eqn for P (B)} that
$$
\begin{pmatrix} \hat{B}_1|_{w=0} \\ \hat{B}_2|_{w=0} \end{pmatrix} = Q\begin{pmatrix} 0 \\ g_{2x}-g_{1y}  \end{pmatrix} + R_\infty\bfg,
$$
and appending \eqref{eq:eqn for B3} to this equation yields
\begin{equation}
\hat{\bfB}|_{w=0} = Z \bfg + R_\infty\bfg, \label{eq:Mpseudo 2}
\end{equation}
where $\mathtt{Z} \in S^0({\mathbb R}^2 / \Lambda)$ and $Z=\Op \mathtt{Z}$.

We have that 
$$
\bfM(\eta)(\bfgamma,\bfg)=\underbrace{\bfM(\eta)(\bfgamma,\bfzero)}_{\displaystyle \in C^\infty({\mathbb R}^2/\Lambda)^2}+\bfM(\eta)(\bfzero,\bfg),
$$
and in the new coordinates
\begin{align}
\bfM(\eta)(\bfzero,\bfg) &= - 
\begin{pmatrix}
\tilde{B}_{3y} \\ - \tilde{B}_{3x}
\end{pmatrix}
 + \frac{1}{\delta } \, \tilde{B}_{3w} \, \nablap\eta + \frac{1}{\delta } \, 
\begin{pmatrix}
\tilde{B}_{2w} \\ -  \tilde{B}_{1w} 
\end{pmatrix}
 - (\tilde{B}_{2x}-\tilde{B}_{1y}) \, \nabla\eta + \frac{1}{\delta} \, (\eta_x \, \tilde{B}_{2w} - \eta_y \, \tilde{B}_{1w} ) \, \nabla\eta \bigg|_{w=0} . \label{eq:def_MOp}
\end{align}
Inserting $\hat{\bfB}|_{w=0}$ and $\partial_w\hat{\bfB}|_{w=0}$ from \eqref{eq:Mpseudo 1}, \eqref{eq:Mpseudo 2}
into this formula shows that
$$\bfM(\eta)(\bfzero,\bfg)=\Op\nu_{\alpha}\bfg+R_\infty\bfg,$$
where $\nu_{\alpha} \in S^1({\mathbb R}^2/\Lambda)$. The asymptotic expansions
$$\mathtt{Z} \sim \sum_{j \leq 0} \mathtt{Z}^{(j)}, \qquad \nu_\alpha \sim\sum_{j \leq 1} \nu_\alpha^{(j)}$$
can be determined recursively by substituting\enlargethispage{5mm}
$$\hat{\bfB}|_{w=0} = Z \bfg + R_\infty\bfg, \qquad \partial_w\hat{\bfB}|_{w=0} = MZ \bfg + R_\infty\bfg$$
into
\eqref{eq:bdry1_B}--\eqref{eq:bdry3_B}.

\begin{remark}
The asymptotic expansion of $\nu_\alpha$ can also be determined from the formula
$$
\bfM(\eta)(\bfzero,\bfg) \coloneqq  - \nabla \big( H(\eta)(\bfzero,\cdot)^{-1} (\nablac \bfg^\perp) \big) + \alpha \, \nablap \Delta^{-1} (\nablac \bfg^\perp)
$$
and the asymptotic expansion of the symbol $\lambda_\alpha$ of $H(\eta)(\bfzero,\cdot)$.
\end{remark}

\subsubsection{Principal symbol} \label{subsubsec:leadingMOp}

Equating terms of order $1$ in equations \eqref{eq:bdry1_B}, \eqref{eq:bdry3_B} and order $0$ in equation \eqref{eq:bdry2_B} yields
\begin{align*}
\ii k_1 \, \mathtt{Z}^{(0)}_{11}  + \ii k_2 \mathtt{Z}^{(0)}_{21}  -\eta_x \,   \mathtt{m}^{(1)} \mathtt{Z}^{(0)}_{11}  -\eta_y \,  \mathtt{m}^{(1)}  \mathtt{Z}^{(0)}_{21}  +  \mathtt{m}^{(1)}  \mathtt{Z}^{(0)}_{31}  &=0, \\
\ii k_1 \, \mathtt{Z}^{(0)}_{12}  + \ii k_2 \mathtt{Z}^{(0)}_{22}  -\eta_x \,  \mathtt{m}^{(1)} \mathtt{Z}^{(0)}_{12}  -\eta_y \,  \mathtt{m}^{(1)}  \mathtt{Z}^{(0)}_{22}  +  \mathtt{m}^{(1)} \mathtt{Z}^{(0)}_{32}    &=0, \\
\eta_x\mathtt{Z}_{11}^{(0)}+\eta_y\mathtt{Z}_{21}^{(0)}-\mathtt{Z}_{31}^{(0)}&=0,\\
\eta_x\mathtt{Z}_{12}^{(0)}+\eta_y\mathtt{Z}_{22}^{(0)}-\mathtt{Z}_{32}^{(0)}&=0,\\
\ii k_1\mathtt{Z}_{21}^{(0)}-\ii k_2\mathtt{Z}_{11}^{(0)}+\ii(\bfk\cdot\nablap\eta)\mathtt{Z}_{31}^{(0)}&=-\ii k_2,\\
\ii k_1\mathtt{Z}_{22}^{(0)}-\ii k_2\mathtt{Z}_{12}^{(0)}+\ii(\bfk\cdot\nablap\eta)\mathtt{Z}_{32}^{(0)}&=\ii k_1,
\end{align*}
whose unique solution is
\begin{align*}
& \mathtt{Z}_{11}^{(0)}(\bfx,\bfk)=\frac{ k_2^2}{(\lambda^{(1)})^2}, &\mathtt{Z}_{12}^{(0)}(\bfx,\bfk)&=-\frac{ k_1 k_2}{(\lambda^{(1)})^2},\\
&\mathtt{Z}_{21}^{(0)}(\bfx,\bfk)=-\frac{ k_1 k_2}{(\lambda^{(1)})^2}, &\mathtt{Z}_{22}^{(0)}(\bfx,\bfk)&=\frac{ k_1^2}{(\lambda^{(1)})^2},\\
&\mathtt{Z}_{31}^{(0)}(\bfx,\bfk)=-\frac{ k_2(\bfk\cdot\nablap\eta)}{(\lambda^{(1)})^2}, &\mathtt{Z}_{32}^{(0)}(\bfx,\bfk)&=\frac{ k_1(\bfk\cdot\nablap\eta)}{(\lambda^{(1)})^2}.
\end{align*}

Equating terms of order $1$ in equation \eqref{eq:def_MOp}, we find that
\begin{align}
\nu_{\alpha}^{(1)}(\bfx,\bfk) \bfg
&= 
\begin{pmatrix}
- \ii  k_2 (\mathtt{Z}^{(0)}_{31}g_1 + \mathtt{Z}^{(0)}_{32}g_2) \\
\ii  k_1 (\mathtt{Z}^{(0)}_{31}g_1 + \mathtt{Z}^{(0)}_{32}g_2)
\end{pmatrix}
 + \mathtt{m}^{(1)} \; (\mathtt{Z}^{(0)}_{31} g_1+  \mathtt{Z}^{(0)}_{32} g_2) \nablap\eta  
  + \begin{pmatrix}
\mathtt{m}^{(1)} \, (\mathtt{Z}^{(0)}_{21} g_1 +  \mathtt{Z}^{(0)}_{22} g_2)  \\
-\mathtt{m}^{(1)} \, (\mathtt{Z}^{(0)}_{11} g_1 +  \mathtt{Z}^{(0)}_{12} g_2) 
\end{pmatrix} \nonumber \\
&\qquad\mbox{} - \left[ \ii k_1 \, (\mathtt{Z}^{(0)}_{21}g_1 + \mathtt{Z}^{(0)}_{22} g_{2}) - \ii k_2 \, ( \mathtt{Z}^{(0)}_{11} g_1 + \mathtt{Z}^{(0)}_{12} g_2)  \right] \nabla\eta \nonumber \\
&\qquad\mbox{} + \bigg[ \eta_x \mathtt{m}^{(1)} \, (\mathtt{Z}^{(0)}_{21} g_1 +  \mathtt{Z}^{(0)}_{22} g_2)   -\eta_y \, \mathtt{m}^{(1)} \, (\mathtt{Z}^{(0)}_{11} g_1 +  \mathtt{Z}^{(0)}_{12} g_2) \bigg] \nabla\eta . \nonumber 
\end{align}
The first component of $\nu_{\alpha}^{(1)} \bfg$ can be rewritten as
\begin{align*}
&\eta_x ( \ii  k_2 - \eta_y \, \mathtt{m}^{(1)} ) \, ( \mathtt{Z}^{(0)}_{11} g_1 + \mathtt{Z}^{(0)}_{12} g_2) + \left( -\ii  k_1 \, \eta_x +(1+\eta_x^2) \mathtt{m}^{(1)} \right) \, ( \mathtt{Z}^{(0)}_{21} g_1 + \mathtt{Z}^{(0)}_{22} g_2) \nonumber \\
&\quad+ (-\ii  k_2 + \eta_y \mathtt{m}^{(1)}) \, ( \mathtt{Z}^{(0)}_{31} g_1 + \mathtt{Z}^{(0)}_{32} g_2) \nonumber \\
=&\left[\eta_x(\ii k_2-\eta_y\mathtt{m}^{(1)} )\mathtt{Z}_{11}^{(0)}+(-\ii k_1\eta_x+(1+\eta_x^2)\mathtt{m}^{(1)} )\mathtt{Z}_{21}^{(0)}+(-\ii k_2+\eta_y\mathtt{m}^{(1)} )\mathtt{Z}_{31}^{(0)}\right]g_1\\
&\quad\mbox{} +\left[\eta_x(\ii k_2-\eta_y\mathtt{m}^{(1)} )\mathtt{Z}_{12}^{(0)}+(-\ii k_1\eta_x+(1+\eta_x^2)\mathtt{m}^{(1)} )\mathtt{Z}_{22}^{(0)}+(-\ii k_2+\eta_y \mathtt{m}^{(1)} )\mathtt{Z}_{32}^{(0)}\right]g_2\\
=&-\frac{ k_1 k_2}{\lambda^{(1)}}g_1+\frac{ k_1^2}{\lambda^{(1)}}g_2\\
=&\frac{ k_1}{\lambda^{(1)}}(\bfk\cdot\bfg^\perp),
\end{align*}
and in the same way we find that the second component of $\nu_{\alpha}^{(1)}(\bfx,\bfk) \bfg$ is
$\dfrac{ k_2}{\lambda^{(1)}}(\bfk\cdot\bfg^\perp)$;
altogether we obtain
\begin{align*}
\nu_\alpha^{(1)}(\bfx,\bfk)\bfg=\bfk\frac{(\bfk\cdot\bfg^\perp)}{\lambda^{(1)}}. 
\end{align*}

\subsubsection{Sub-principal symbol} \label{subsubsec:subleadMOp}

Equating terms of order $0$ in equations \eqref{eq:bdry1_B}, \eqref{eq:bdry3_B} and order $-1$ in equation \eqref{eq:bdry2_B} yields\enlargethispage{5mm}
\begin{align*}
\ii k_1\mathtt{Z}_{11}^{(-1)}+\ii k_2\mathtt{Z}_{21}^{(-1)}+ k_2\mathtt{G}_1&=0,\\
\ii k_1\mathtt{Z}_{12}^{(-1)}+\ii k_2\mathtt{Z}_{22}^{(-1)}+ k_1\mathtt{G}_2&=0,\\
\mathtt{Z}_{31}^{(-1)}&=\eta_x\mathtt{Z}_{11}^{(-1)}+\eta_y\mathtt{Z}_{21}^{(-1)},\\
\mathtt{Z}_{32}^{(-1)}&=\eta_x\mathtt{Z}_{12}^{(-1)}+\eta_y\mathtt{Z}_{22}^{(-1)},\\
(-\ii k_2+\ii(\bfk\cdot\nablap\eta)\eta_x)\mathtt{Z}_{11}^{(-1)}+(\ii k_1+\ii(\bfk\cdot\nablap\eta)\eta_y)\mathtt{Z}_{21}^{(-1)}+ k_2\mathtt{G}_3&=0,\\
(-\ii k_2+\ii(\bfk\cdot\nablap\eta)\eta_x)\mathtt{Z}_{12}^{(-1)}+(\ii k_1+\ii(\bfk\cdot\nablap\eta)\eta_y)\mathtt{Z}_{22}^{(-1)}+ k_1\mathtt{G}_4&=0,
\end{align*}
where 
\begin{align*}
\mathtt{G}_1(\bfx,\bfk)&=\frac{2(\bfk\cdot\nablap\lambda^{(1)})}{(\lambda^{(1)})^3}-\frac{\ii}{(\lambda^{(1)})^2}\left[ k_2(\nabla_{\bfk}\cdot\nabla\eta_x)- k_1(\nabla_{\bfk}\mathtt{m}^{(1)} \cdot\nabla\eta_y)\right]+\frac{\ii\alpha}{2\lambda^{(1)}},\\
\mathtt{G}_2(\bfx,\bfk)&=-\frac{2(\bfk\cdot\nablap\lambda^{(1)})}{(\lambda^{(1)})^3}+\frac{\ii}{(\lambda^{(1)})^2}\left[ k_2(\nabla_{\bfk}\mathtt{m}^{(1)} \cdot\nabla\eta_x)- k_1(\nabla_{\bfk}\mathtt{m}^{(1)} \cdot\nabla\eta_y)\right]-\frac{\ii\alpha}{2\lambda^{(1)}},\\
\mathtt{G}_3(\bfx,\bfk)&=\frac{1}{(\lambda^{(1)})^3}\bigg[2(1+\eta_y^2)\partial_x\lambda^{(1)} k_1+2(1+\eta_x^2)\partial_y\lambda^{(1)} k_2-2\eta_x\eta_y(\partial_x\lambda^{(1)} k_2+\partial_y\lambda^{(1)} k_1)\\
&\hspace{18mm}\mbox{} +(\eta_y\eta_{xx}-\eta_x\eta_{xy})\lambda^{(1)} k_2-(\eta_y\eta_{xy}-\eta_x\eta_{yy})\lambda^{(1)} k_1\bigg],\\
\mathtt{G}_4(\bfx,\bfk)&=\frac{1}{(\lambda^{(1)})^3}\bigg[-2(1+\eta-y^2)\partial_x\lambda^{(1)} k_1-2(1+\eta_x^2)\partial_y\lambda^{(1)} k_2+2\eta_x\eta_y(\partial_x\lambda^{(1)} k_2+\partial_y\lambda^{(1)} k_1)\\
&\hspace{18mm}\mbox{}-(\eta_y\eta_{xx}-\eta_x\eta_{xy})\lambda^{(1)} k_2+(\eta_y\eta_{xy}-\eta_x\eta_{yy})\lambda^{(1)} k_1\bigg],
\end{align*}
whose unique solution is
\begin{align*}
\mathtt{Z}_{11}^{(-1)}(\bfx,\bfk)&=\frac{\ii k_2}{(\lambda^{(1)})^2}\left(( k_1+(\bfk\cdot\nablap\eta)\eta_y)- k_2\mathtt{G}_3\right),\\
\mathtt{Z}_{12}^{(-1)}(\bfx,\bfk)&=\frac{\ii \mathtt{G}_2}{(\lambda^{(1)})^2}\left((\lambda^{(1)})^2- k_2( k_2-(\bfk\cdot\nablap\eta)\eta_x)\right)-\frac{\ii k_1 k_2\mathtt{G}_4}{(\lambda^{(1)})^2},\\
\mathtt{Z}_{21}^{(-1)}(\bfx,\bfk)&=\frac{\ii \mathtt{G}_1}{(\lambda^{(1)})^2}\left((\lambda^{(1)})^2- k_1( k_1+(\bfk\cdot\nablap\eta)\eta_y)\right)+\frac{\ii k_1 k_2\mathtt{G}_3}{(\lambda^{(1)})^2},\\
\mathtt{Z}_{22}^{(-1)}(\bfx,\bfk)&=\frac{\ii k_1}{(\lambda^{(1)})^2}\left(( k_2-(\bfk\cdot\nablap\eta)\eta_x)\mathtt{G}_2+ k_1\mathtt{G}_4\right),\\
\mathtt{Z}_{31}^{(-1)}(\bfx,\bfk)&=\eta_x\mathtt{Z}_{11}^{(-1)}+\eta_y\mathtt{Z}_{21}^{(-1)},\\
\mathtt{Z}_{32}^{(-1)}(\bfx,\bfk)&=\eta_x\mathtt{Z}_{12}^{(-1)}+\eta_y\mathtt{Z}_{22}^{(-1)}.
\end{align*}

Inserting these formulae into the zeroth order part of \eqref{eq:def_MOp}, we find after a lengthy but straighforward computation that
$$
\nu_{\alpha}^{(0)}(\bfx,\bfk)\bfg=\left(\begin{array}{c}
\zeta_1(\bfx,\bfk)\\
\zeta_2(\bfx,\bfk)
\end{array}\right)(\bfk\cdot\bfg^\perp),
$$
where 
\begin{align*}
\zeta_1(\bfx,\bfk)&=\frac{\ii}{2(\lambda^{(1)})^5}\bigg( k_1^2(-1+2\eta_y^2)\eta_x- k_1 k_2\eta_y(3+4\eta_x^2)+2 k_2^2\eta_x(1+\eta_x^2)+\ii k_1\lambda^{(1)}\bigg)\\
&\qquad \times \bigg( k_1^2\eta_{yy}-2 k_1 k_2\eta_{xy}+ k_2^2\eta_{xx}\bigg)+\frac{\alpha}{(\lambda^{(1)})^2}\left( k_2(1+\eta_x^2)- k_1\eta_x\eta_y\right),\\
\zeta_2(\bfx,\bfk)&=\frac{\ii}{2(\lambda^{(1)})^5}\bigg(2 k_1^2\eta_y(1+\eta_y^2)- k_1 k_2\eta_x(3+4\eta_y^2)+ k_2^2\eta_y(-1+2\eta_x^2)+\ii k_2\lambda^{(1)}\bigg)\\
&\qquad \times \bigg( k_1^2\eta_{yy}-2 k_1 k_2\eta_{xy}+ k_2^2\eta_{xx}\bigg)+\frac{\alpha}{(\lambda^{(1)})^2}\left(- k_1(1+\eta_y^2)+ k_2\eta_x\eta_y\right).
\end{align*}

\section{Approximate solutions} \label{sec:approx}

In this section we construct approximate solutions of 
\begin{equation} J(\eta,\bfmu)=0 \label{eq:hydroeq with mu}
\end{equation}
 for $\beta\geq 0$ in the form of power series and moreover prove their convergence for
$\beta>0$; the solutions have wave velocity $\bfc$ close to a reference value $\bfc_0$ chosen such that the transversality condition (T) holds. Assuming that the
non-resonance condition (NR) also holds, we consider $J$ as a locally analytic mapping $X_s^\beta \times {\mathbb R}^2 \rightarrow H^s({\mathbb R}^2/ \Lambda)$ for a sufficiently large value of $s$, where
\begin{align*}
 X_s^\beta&\coloneqq\begin{cases}
H^{s+2}(\mathbb{R}^2/\Lambda), &  \mbox{if $\beta>0$,}\\
H^{s+1}(\mathbb{R}^2/\Lambda),  &  \mbox{if $\beta=0$.}
\end{cases}
\end{align*}
Our strategy is to perform a Lyapunov--Schmidt reduction, and we therefore proceed to investigate the kernel and range of 
$$J_{10}(\eta)\coloneqq \mathrm{d}_1J[0,\bfzero](\eta)=\bfT_1(\eta)\cdot\bfc_0+g\eta-\beta\Delta\eta.$$

Write
\begin{equation*}
\eta(\bfx)=\sum_{\bfk\in\Lambda^\prime}\hat{\eta}_{\bfk}\ee^{\ii\bfk\cdot\bfx},
\end{equation*}
so that
$$(J_{10}\eta)(\bfx) =g\hat\eta_{\bfzero}+\sum_{\bfk\in\Lambda^\prime\setminus\{\bfzero\}}\frac{\mathtt{c}(|\bfk|)}{|\bfk|^2}\rho(\bfk,\bfc_0,\beta)\hat{\eta}_{\bfk}\ee^{\ii\bfk\cdot\bfx}.$$
The equation $J_{10} \eta=0$ is equivalent to
\begin{equation*}
\rho(\bfk,\bfc_0,\beta)\hat{\eta}_{\bfk}=0
\end{equation*}
for $\bfk \in \Lambda^\prime\setminus\{\bfzero\}$,
which by assumption has non-trivial solutions if and only if $\bfk=\pm\bfk_1,\pm\bfk_2$; it follows that
\begin{equation*}
\ker(J_{10})=\{A\ee^{\ii\bfk_1\cdot\bfx}+B\ee^{\ii\bfk_2\cdot\bfx}+\bar{A}\ee^{-\ii\bfk_1\cdot\bfx}+\bar{B}\ee^{-\ii\bfk_2\cdot\bfx}\colon A, B \in {\mathbb C}\}.
\end{equation*}
We next consider the range of $J_{10}$. Let 
\begin{align*}
f(\bfx)&=\sum_{\bfk\in\Lambda^\prime}\hat{f}_{\bfk}\ee^{\ii\bfk\cdot\bfx} \in H^s(\mathbb{R}^2/\Lambda).
\end{align*}
The equation $J_{10} \eta=f$ is equivalent to
$$
g\hat\eta_{\bfzero} = \hat{f}_{\bfzero}
$$
and
\begin{equation}\label{range-eq}
\frac{\mathtt{c}(|\bfk|)}{|\bfk|^2}\rho(\bfk,\bfc_0,\beta)\hat{\eta}_{\bfk}=\hat{f}_{\bfk}
\end{equation}
for $\bfk \in \Lambda^\prime\setminus\{\bfzero\}$. Obviously
\begin{equation}
\hat\eta_{\bfzero} = \frac{1}{g}\hat{f}_{\bfzero}, \label{eq:zeromode}
\end{equation}
while for $\bfk\neq \pm\bfk_1,\pm \bfk_2$ equation \eqref{range-eq} has the unique solution 
\begin{equation}\label{eq:solution-linear-eq}
\hat{\eta}_{\bfk}=\frac{|\bfk|^2}{\mathtt{c}(|\bfk|)\rho(\bfk,\bfc_0,\beta)}\hat{f}_{\bfk},
\end{equation}
and for $\bfk=\pm\bfk_1,\pm\bfk_2$ it is solvable if and only if $\hat{f}_{\pm\bfk_1}=\hat{f}_{\pm\bfk_2}=0$.
For $\beta>0$ we find that $\rho(\bfk,\bfc_0,\beta)\gtrsim |\bfk|^3$ for sufficiently large $|\bfk|$, so that the series
\begin{equation*}
\sum_{\bfk\in\Lambda^\prime \atop \bfk\neq \pm\bfk_1,\pm\bfk_2}\hat{\eta}_{\bfk}\ee^{\ii\bfk\cdot\bfx},
\end{equation*}
where $\hat{\eta}_{\bfk}$ is given by \eqref{eq:zeromode}, \eqref{eq:solution-linear-eq}, converges in $H^{s+2}(\mathbb{R}^2/\Lambda)$. It follows that
$J_{10}\colon H^{s+2}(\mathbb{R}^2/\Lambda) \rightarrow H^s(\mathbb{R}^2/\Lambda) $ is Fredholm with index $0$, where
\begin{equation*}
\ran(J_{10})=\{f\in H^s(\mathbb{R}^2/\Lambda)\colon \hat{f}_{\pm\bfk_1}=\hat{f}_{\pm\bfk_2}=0\}
\end{equation*}
and $J_{10}^{-1}\colon \ran(J_{10}) \rightarrow H^{s+2}(\mathbb{R}^2/\Lambda)$ is given by
\eqref{eq:zeromode}, \eqref{eq:solution-linear-eq}.
In contrast $\rho(\bfk,\bfc_0,0)$ is not bounded from below as $|\bfk|\rightarrow \infty$, so that \eqref{eq:zeromode},
\eqref{eq:solution-linear-eq} does not define a bounded operator from  $H^s(\mathbb{R}^2/\Lambda)$ to $H^{s+1}(\mathbb{R}^2/\Lambda)$ for any $s$.
We therefore proceed formally, noting that the procedure is rigorously valid for $\beta>0$.

To apply the Lyapunov--Schmidt reduction let $\Pi$ be the orthogonal projection of $H^s(\mathbb{R}^2/\Lambda)$ onto $\text{ker}(J_{10})$ with respect to the $L^2(\mathbb{R}^2/\Lambda)$ inner product $\langle \cdot\,,\cdot\rangle$. Write $\eta=\eta_1+\eta_2$, where
\begin{equation*}
\eta_1=A\ee^{\ii\bfk_1\cdot\bfx}+B\ee^{\ii\bfk_2\cdot\bfx}+\bar{A}\ee^{-\ii\bfk_1\cdot\bfx}+\bar{B}\ee^{-\ii\bfk_2\cdot\bfx},
\end{equation*}
and $\eta_2\in\ker(J_{10})^\perp= (I-\Pi)X_s^\beta$, and decompose \eqref{eq:hydroeq with mu} as
\begin{align}
\Pi J(\eta_1+\eta_2,\bfmu)&=0,\label{kereq}\\
(I-\Pi) J(\eta_1+\eta_2,\bfmu)&=0.\label{raneq}
\end{align}
The linearisation of $(I-\Pi) J$ at $0$ is 
\begin{equation*}
(I-\Pi)J_{10}\colon (I-\Pi)X_s^\beta\rightarrow (I-\Pi)H^s(\mathbb{R}^2/\Lambda).
\end{equation*} 
For $\beta>0$ this operator is an isomorphism (see above) and we can solve \eqref{raneq} to determine $\eta_2$ as a locally analytic function of $\eta_1$ and $\bfmu$; substituting $\eta_2=\eta_2(\eta_1,\mu)$ into \eqref{kereq} yields the reduced equation
\begin{equation}
\Pi J(\eta_1+\eta_2(\eta_1,\bfmu),\bfmu)=0. \label{eq:redeq}
\end{equation}\enlargethispage{5mm}
Note that $\eta_2={\mathcal O}(|(\eta_1,\bfmu)||\eta_1|)$ and the left-hand side of equation \eqref{eq:redeq} is
also ${\mathcal O}(|(\eta_1,\bfmu)||\eta_1|)$ because\linebreak $\Pi J_{10}(\eta_1 +\eta_2(\eta_1,\bfmu))=0$.
For $\beta=0$ we can only formally solve \eqref{raneq} for $\eta_2$ as a function of $\eta_1$ and $\bfmu$.

We proceed to solve equation \eqref{eq:redeq}, which can be written as
$$
\langle  J(\eta_1+\eta_2(\eta_1,\bfmu),\bfmu),\ee^{\ii\bfk_i\cdot\bfx}\rangle=0,\qquad i=1, 2,
$$
because $ J$ is real-valued. We write these equations as
\begin{align*}
f_1(A,B,\bar{A},\bar{B},\bfmu)&=0,\\
f_2(A,B,\bar{A},\bar{B},\bfmu)&=0,
\end{align*}
and note that
$$f_j(A,B,\bar{A},\bar{B},\bfmu) = O(|(A,B,\bfmu)||(A,B)|).$$
Recall that $ J$ is equivariant with respect to the symmetries $S_0$ and $T_{\bfv^\prime}$ (see Remark \ref{rem:symmetries}), which act on the coordinates $(A,B,\bar{A},\bar{B})$ as 
$$S_0(A,B,\bar{A},\bar{B})=(\bar{A},\bar{B},A,B), \qquad T_{\bfv^\prime}(A,B,\bar{A},\bar{B})=(A\ee^{\ii\bfk_1\cdot\bfv^\prime},B\ee^{\ii\bfk_2\cdot\bfv^\prime},\bar{A}\ee^{-\ii\bfk_1\cdot\bfv^\prime},\bar{B}\ee^{-\ii\bfk_2\cdot\bfv^\prime}),$$
so that the reduced equation remains equivariant under these symmetries, that is
\begin{align*}
f_1(A\ee^{\ii\bfk_1\cdot\bfv^\prime},B\ee^{\ii\bfk_2\cdot\bfv^\prime},\bar{A}\ee^{-\ii\bfk_1\cdot\bfv^\prime},\bar{B}\ee^{-\ii\bfk_2\cdot\bfv^\prime},\bfmu)&=\ee^{\ii\bfk_1\cdot\bfv^\prime}f_1(A,B,\bar{A},\bar{B},\bfmu),\\
f_2(A\ee^{\ii\bfk_1\cdot\bfv^\prime},B\ee^{\ii\bfk_2\cdot\bfv^\prime},\bar{A}\ee^{-\ii\bfk_1\cdot\bfv^\prime},\bar{B}\ee^{-\ii\bfk_2\cdot\bfv^\prime},\bfmu)&=\ee^{\ii\bfk_2\cdot\bfv^\prime}f_2(A,B,\bar{A},\bar{B},\bfmu),\\
f_1(\bar{A},\bar{B},A,B,\bfmu)&=\bar{f}_1(A,B,\bar{A},\bar{B},\bfmu),\\
f_2(\bar{A},\bar{B},A,B,\bfmu)&=\bar{f}_2(A,B,\bar{A},\bar{B},\bfmu).
\end{align*}
It follows that
\begin{align}
f_1(A,B,\bar{A},\bar{B},\bfmu)&=Ag_1(|A|^2,|B|^2,\bfmu), \label{f1 eqn}\\
f_2(A,B,\bar{A},\bar{B},\bfmu)&=Bg_2(|A|^2,|B|^2,\bfmu), \label{f2 eqn}
\end{align}
where $g_1,g_2$ are real-valued locally analytic functions which vanish at the origin.

Solutions to equations \eqref{f1 eqn}, \eqref{f2 eqn} 
with $A \neq 0$, $B=0$, such that
\begin{equation*}
g_1(|A|^2,0,\bfmu)=0,
\end{equation*}
lead to solutions of \eqref{eq:hydroeq with mu}
of the form $\eta=\eta_1+\eta_2(\eta_1,\bfmu)$ with $\eta_1=A\ee^{\ii\bfk_1\cdot\bfx}+\overline{A}\ee^{-\ii\bfk_1\cdot\bfx}$,
so that $\eta$ depends on the single variable $\tilde{x}\coloneqq \bfk_1\cdot\bfx$. Such waves are often called \emph{$2\tfrac{1}{2}$-dimensional waves} since they only depend upon one horizontal variable $\tilde{x}$. Similarly, solutions to \eqref{f1 eqn}, \eqref{f2 eqn} 
with $A = 0$, $B \neq 0$ give rise to $2\tfrac{1}{2}$-dimensional waves depending on the single horizontal variable $\bfk_2\cdot\bfx$. We refer to Lokharu, Seth and Wahl\'{e}n \cite[Section 1.2.2]{LokharuWahlen19} for a more detailed discussion on $2\tfrac{1}{2}$-dimensional waves.  Fully three-dimensional waves are found by assuming that $A\neq 0$ and $B \neq 0$, in which case \eqref{f1 eqn}, \eqref{f2 eqn}  are equivalent to 
\begin{align}
g_1(|A|^2,|B|^2,\bfmu)&=0, \label{reduced_eq1}\\
g_2(|A|^2,|B|^2,\bfmu)&=0. \label{reduced_eq2}
\end{align}

\begin{proposition}\label{implicit_prop}
There exist $\varepsilon>0$ and analytic functions $\mu_i\colon B_\varepsilon ({\bf 0},\mathbb{R}^2)\rightarrow {\mathbb R}$, $i=1,2$ such that $\mu_i(0,0)=0$ and $(|A|^2,|B|^2,\mu_1(|A|^2,|B|^2),\mu_2(|A|^2,|B|^2))$ is the unique local solution of \eqref{reduced_eq1}, \eqref{reduced_eq2}.
\end{proposition}
\begin{proof}
Write equations \eqref{reduced_eq1}, \eqref{reduced_eq2} as
\begin{align}
a_1\mu_1+a_2\mu_2+\mathcal{O}(|(|A|^2,|B|^2)|+|(|A|^2,|B|^2,\bfmu)|^2)&=0, \label{eq:trans1} \\
b_1\mu_1+b_2\mu_2+\mathcal{O}(|(|A|^2,|B|^2)|+|(|A|^2,|B|^2,\bfmu)|^2)&=0, \label{eq:trans2}
\end{align}
where
\begin{align*}
a_1&=\langle J_{11}\ee^{\ii\bfk_1\cdot\bfx},\ee^{\ii\bfk_1\cdot\bfx}\rangle, & b_1&=\langle J_{11}\ee^{\ii\bfk_2\cdot\bfx},\ee^{\ii\bfk_2\cdot\bfx}\rangle,\\
a_2&=\langle J_{12}\ee^{\ii\bfk_1\cdot\bfx},\ee^{\ii\bfk_1\cdot\bfx}\rangle, &b_2&=\langle J_{12}\ee^{\ii\bfk_2\cdot\bfx},\ee^{\ii\bfk_2\cdot\bfx}\rangle,
\end{align*}
and $J_{11}=\partial_{\mu_1}\mathrm{d}_1 J[0,\bfmu]|_{\bfmu=0}$, $J_{12}=\partial_{\mu_2}\mathrm{d}_1 J[0,\bfmu]|_{\bfmu=0}$.
A short calculation shows that
$$
\frac{\partial}{\partial c_{1}}\rho(\bfk,\bfc_0,\beta)=\frac{|\bfk|^2}{\mathtt{c}(|\bfk|)}\langle J_{11}\ee^{\ii\bfk\cdot\bfx},\ee^{\ii\bfk\cdot\bfx}\rangle,\qquad
\frac{\partial}{\partial c_{2}}\rho(\bfk,\bfc_0,\beta)=\frac{|\bfk|^2}{\mathtt{c}(|\bfk|)}\langle J_{12}\ee^{\ii\bfk\cdot\bfx},\ee^{\ii\bfk\cdot\bfx}\rangle,
$$
and hence
\begin{align*}
a_1&=\frac{\mathtt{c}(|\bfk_1|)}{|\bfk_1|^2}\frac{\partial}{\partial{c_{1}}}\rho(\bfk_1,\bfc_0,\beta),& b_1&=\frac{\mathtt{c}(\bfk_2)}{|\bfk_2|^2}\frac{\partial}{\partial{c_{1}}}\rho(\bfk_2,\bfc_0,\beta),\\
a_2&=\frac{\mathtt{c}(\bfk_1)}{|\bfk_1|^2}\frac{\partial}{\partial{c_{2}}}\rho(\bfk_1,\bfc_0,\beta), & b_2&=\frac{\mathtt{c}(\bfk_2)}{|\bfk_2|^2}\frac{\partial}{\partial{c_{2}}}\rho(\bfk_2,\bfc_0,\beta).
\end{align*}
Equations \eqref{eq:trans1}, \eqref{eq:trans2} can be locally solved for $\mu_1$, $\mu_2$ as functions of $|A|^2$, $|B|^2$
by the implicit function theorem provided that
$$
\det\left(\begin{array}{cc}
a_1 & a_2\\
b_1 & b_2
\end{array}\right)\neq 0.
$$
The above formulae show that this condition holds if and only if
$\nabla_{\!\bfc}\,\rho(\bfk_1,\bfc_0,\beta)$ and $\nabla_{\!\bfc}\,\rho(\bfk_2,\bfc_0,\beta)$ are linearly independent.
\end{proof}

Our main result now follows by substituting $\bfmu = \bfmu(|A|^2,|B|^2)$ into $\eta=\eta_1+\eta_2(\eta_1,\bfmu)$.
\begin{theorem}
Suppose that $\beta>0$. There exist $\varepsilon>0$, a neighbourhood $V$ of the origin in $X_s^\beta \times {\mathbb R}^2$ and analytic functions
$\mu_1,\mu_2 \colon B_\varepsilon ({\bf 0},{\mathbb R}^2)\rightarrow {\mathbb R}$ and $\eta \colon B_\varepsilon ({\bf 0},{\mathbb C}^4) \rightarrow X_s^\beta$ such that
$$\{(\eta,\bfmu) \in X_s^\beta \times {\mathbb R}^2 \colon  J(\eta,\bfmu)=0,\ \eta \neq 0\} \cap V = \{ (\eta(A,B,\bar{A},\bar{B}),\bfmu(|A|^2,|B|^2)) \colon (A,B,\bar{A},\bar{B}) \in B_\varepsilon^\prime ({\bf 0},{\mathbb C}^4)\};$$
furthermore $\bfmu(0,0)=\bfzero$ and
$$\eta(\bfx)=A\ee^{\ii\bfk_1\cdot\bfx}+B\ee^{\ii\bfk_2\cdot\bfx}+\bar{A}\ee^{-\ii\bfk_1\cdot\bfx}+\bar{B}\ee^{-\ii\bfk_2\cdot\bfx}
+O(|(A,B,\bar{A},\bar{B})|^2).$$
\end{theorem}

\begin{remarks} $ $
\begin{itemize}
\item[(i)] Elements of the solution set $\{ (\eta(A,B,\bar{A},\bar{B}),\bfmu(|A|^2,|B|^2)) \colon (A,B,\bar{A},\bar{B}) \in B_\varepsilon^\prime ({\bf 0},{\mathbb C}^4)\}$
with $A=0$ or $B=0$ are $2\tfrac{1}{2}$-dimensional waves (see above).
\item[(ii)] Elements of the solution set $\{ (\eta(A,B,\bar{A},\bar{B}),\bfmu(|A|^2,|B|^2)) \colon (A,B,\bar{A},\bar{B}) \in B_\varepsilon^\prime ({\bf 0},{\mathbb C}^4)\}$
with $A$, $B \in {\mathbb R}$ are waves which are invariant under the reflection $S_0$. Note that it is possible to restrict to such solutions before performing the
Lyapunov--Schmidt reduction; this approach was taken by Craig and Nicholls \cite{CraigNicholls02} in a similar study of irrotational travelling waves.
\end{itemize}
\end{remarks}

The terms in the series
$$\eta=A\ee^{\ii\bfk_1\cdot\bfx}+B\ee^{\ii\bfk_2\cdot\bfx}+\bar{A}\ee^{-\ii\bfk_1\cdot\bfx}+\bar{B}\ee^{-\ii\bfk_2\cdot\bfx}
+\hspace{-4mm}\sum_{i+j+k+l \geq 2}\hspace{-2mm}\eta_{ijkl}A^iB^j\bar{A}^k\bar{B}^l
$$
and
$$\mu_i = \sum_{j+k \geq 1} \mu_{i,jk} |A|^{2j} |B|^{2k}, \qquad i=1,2,$$
can be determined recursively by substituting these expressions into \eqref{eq:hydroeq with mu}
and equating monomials in $(A,B,\bar{A},\bar{B})$. Note that the series can be computed to any order for $\beta \geq 0$
but their convergence has been established only for $\beta>0$.

\begin{itemize}
\item
We find that
\begin{equation*}
\eta_{2,2}(\eta_1)=\sum_{i+j+k+l=2}\eta_{2,ijkl}A^iB^j\bar{A}^k\bar{B}^l
\end{equation*}
satisfies the equation
$$
J_{10} \eta_{2,2}=-J_{20}(\eta_1,\eta_1), \label{eq:matching}
$$
where $J_{20}\coloneqq \frac{1}{2}\mathrm{d}_1^2 J[0,\bfzero]$, so that
\begin{align*}
J_{20}(\eta_1,\eta_1)&=A^2J_{20}(\ee^{\ii\bfk_1\cdot\bfx},\ee^{\ii\bfk_1\cdot\bfx})+2ABJ_{20}(\ee^{\ii\bfk_1\cdot\bfx},\ee^{\ii\bfk_2\cdot\bfx})+2|A|^2J_{20}(\ee^{\ii\bfk_1\cdot\bfx},\ee^{-\ii\bfk_1\cdot\bfx})\\
&\qquad\mbox{} +2A\bar{B}J_{20}(\ee^{\ii\bfk_1\cdot\bfx},\ee^{-\ii\bfk_2\cdot\bfx})+B^2J_{20}(\ee^{\ii\bfk_2\cdot\bfx},\ee^{\ii\bfk_2\cdot\bfx})+2\bar{A}BJ_{20}(\ee^{-\ii\bfk_1\cdot\bfx},\ee^{\ii\bfk_2\cdot\bfx})\\
&\qquad\mbox{} +2|B|^2J_{20}(\ee^{\ii\bfk_2\cdot\bfx},\ee^{-\ii\bfk_2\cdot\bfx})+\bar{A}^2J_{20}(\ee^{-\ii\bfk_1\cdot\bfx},\ee^{-\ii\bfk_1\cdot\bfx})+2\bar{A}\bar{B}J_{20}(\ee^{-\ii\bfk_1\cdot\bfx},\ee^{-\ii\bfk_2\cdot\bfx})\\
&\qquad\mbox{} +\bar{B}^2J_{20}(\ee^{-\ii\bfk_2\cdot\bfx},\ee^{-\ii\bfk_2\cdot\bfx}).
\end{align*}

For $\bfl,\bfk$ with $\bfk \neq -\bfl$ we find that
\begin{align*}
J_{20}(\ee^{\ii\bfk\cdot\bfx}, \ee^{\ii\bfl\cdot\bfx})&=\underbrace{\bigg[\frac{1}{2}\mathtt{T}_{10}(\bfk)\cdot\mathtt{T}_{10}(\bfl)+\frac{1}{2}(\bfk\cdot \bfc_0)(\bfl \cdot\bfc_0)+\mathtt{T}_{20,2}(\bfk,\bfl)\cdot \bfc_0+\frac{\alpha}{2}(\mathtt{T}_{10}(\bfk)+\mathtt{T}_{10}(\bfl))\cdot\bfc_0^\perp\bigg]}_{\displaystyle \eqqcolon \mathtt{p}_{20,2}(\bfk,\bfl)}\ee^{\ii(\bfk+\bfl)\cdot\bfx},
\end{align*}
while
\begin{align*}
J_{20}(\ee^{\ii\bfk\cdot\bfx},\ee^{-\ii\bfk\cdot\bfx})&=\underbrace{\frac{1}{2}|\mathtt{T}_{10}(\bfk)|^2-\frac{1}{2}(\bfk\cdot\bfc_0)^2+\mathtt{T}_{20,1}(\bfk)\cdot\bfc_0+\alpha\mathtt{T}_{10}(\bfk)\cdot\bfc_0^\perp}_{\displaystyle \eqqcolon\mathtt{p}_{20,1}(\bfk)},
\end{align*}
where
\begin{align*}
\mathtt{T}_{10}(\bfk) & = - \left( \alpha \, \bfk^{\perp} + \bfk \, \mathtt{c}(|\bfk|) \right) \frac{ \bfc_0 \cdot \bfk}{|\bfk|^2}, \\
\mathtt{T}_{20,2}(\bfk,\bfl)&=\frac{1}{2|\bfk+\bfl|^2}\left(\alpha(\bfk+\bfl)^\perp+(\bfk+\bfl)\mathtt{c}(|\bfk+\bfl|)\right)\bigg[\alpha(\bfk+\bfl)\cdot\bfc_0+\alpha(\bfk\cdot\bfl^\perp)\left(\frac{\bfc_0\cdot\bfl}{|\bfl|^2}-\frac{\bfc_0\bfk}{|\bfk|^2}\right)\\
&\hspace{6.75cm}\mbox{} +(\bfk+\bfl)\cdot\left(\frac{\bfc_0\cdot\bfl}{|\bfl|^2}\mathtt{c}(|\bfl|)\bfl+\frac{\bfc_0\cdot\bfk}{|\bfk|^2}\mathtt{c}(|\bfk|)\bfk\right)\bigg]\\
& \qquad\mbox{}+\left((\alpha^2-|\bfl|^2)\bfl-\alpha\mathtt{c}(|\bfl|)\bfl^\perp\right)\frac{\bfc_0\cdot\bfl}{2|\bfl|^2}
+\left((\alpha^2-|\bfk|^2)\bfk-\alpha\mathtt{c}(|\bfk|)\bfk^\perp\right)\frac{\bfc_0\cdot\bfk}{2|\bfk|^2}\\
&\qquad -\frac{1}{2}\bfk(\bfc_0\cdot\bfl)-\frac{1}{2}\bfl(\bfc_0\cdot\bfk),\\
\mathtt{T}_{20,1}(\bfk)&=\alpha(\alpha\bfk-\mathtt{c}(|\bfk|)\bfk^\perp).
\end{align*}
The solution of the equation
\begin{equation*}
J_{10} f \ee^{\ii(\bfl+\bfk)\cdot\bfx}=\mathtt{p}_{20,2}(\bfk,\bfl)\ee^{\ii(\bfl+\bfk)\cdot\bfx}, \qquad \bfk \neq -\bfl,
\end{equation*}
is
\begin{align*}
f=\underbrace{\frac{|\bfk+\bfl|^2}{\mathtt{c}(|\bfk+\bfl|)\rho(\bfk+\bfl,\bfc_0,\beta)}\mathtt{p}_{20,2}(\bfk,\bfl)}_{\displaystyle \eqqcolon\mathtt{q}_{20,2}(\bfl,\bfk)},
\end{align*}
while the solution of
\begin{equation*}
J_{10} f=\mathtt{p}_{20,1}(\bfk)
\end{equation*}
is simply
\begin{equation*}
f=\frac{1}{g}\mathtt{p}_{20,1}(\bfk).
\end{equation*}
Altogether we find that
\begin{align*}
\eta_{2,2000}&=-\mathtt{q}_{20,2}(\bfk_1,\bfk_1)\ee^{2\ii\bfk_1\cdot\bfx}, & \eta_{2,0020}=\overline{\eta_{2,2000}},\\
\eta_{2,1100}&=-2\mathtt{q}_{20,2}(\bfk_1,\bfk_2)\ee^{\ii(\bfk_1+\bfk_2)\cdot\bfx},& \eta_{2,0011}=\overline{\eta_{2,1100}},\\
\eta_{2,1010}&=-\frac{2}{g}\mathtt{p}_{20,1}(\bfk_1),\\
\eta_{2,1001}&=-2\mathtt{q}_{20,2}(\bfk_1,-\bfk_2)\ee^{\ii(\bfk_1-\bfk_2)\cdot\bfx}, & \eta_{2,0110}=\overline{\eta_{2,1001}},\\
\eta_{2,0200}&=-\mathtt{q}_{20,2}(\bfk_2,\bfk_2)\ee^{2\ii\bfk_2\cdot\bfx}, & \eta_{2,0002}=\overline{\eta_{2,0200}},\\
\eta_{2,0101}&=-\frac{2}{g}\mathtt{p}_{20,1}(\bfk_2).
\end{align*}

\item
Expanding \eqref{reduced_eq1}, \eqref{reduced_eq2} further as
\begin{align*}
a_1\mu_1+a_2\mu_2+a_3|A|^2+a_4|B|^2+\mathcal{O}(|(|A|^2,|B|^2,\bfmu)|^2)&=0, \\
b_1\mu_1+b_2\mu_2+b_3|A|^2+b_4|B|^2+\mathcal{O}(|(|A|^2,|B|^2,\bfmu)|^2)&=0,
\end{align*}
we find that
\begin{align*}
\mu_1(|A|^2,|B|^2)&=-\frac{a_3b_2-a_2b_3}{a_1b_2-b_1a_2}|A|^2-\frac{a_4b_2-a_2b_4}{a_1b_2-b_1a_2}|B|^2+\mathcal{O}(|(|A|^2,|B|^2)|^2),\\
\mu_2(|A|^2,|B|^2)&=-\frac{a_1b_3-a_3b_1}{a_1b_2-b_1a_2}|A|^2-\frac{a_1b_4-a_4b_1}{a_1b_2-b_1a_2}|B|^2+\mathcal{O}(|(|A|^2,|B|^2)|^2).
\end{align*}
The coefficients $a_3,a_4,b_3,b_4$ are given by 
\begin{align*}
a_3&=\left\langle 2J_{20}(\ee^{\ii\bfk_1\cdot\bfx},\eta_{2,1010})+2J_{20}(\ee^{-\ii\bfk_1\bfx},\eta_{2,2000})+3J_{30}(\ee^{\ii\bfk_1\cdot\bfx},\ee^{\ii\bfk_1\cdot\bfx},\ee^{-\ii\bfk_1\cdot\bfx}),\ee^{\ii\bfk_1\cdot\bfx}\right\rangle\\
&=-\frac{4}{g}\mathtt{p}_{20,1}(\bfk_1)\mathtt{p}_{20,2}(\bfk_1,\bfzero)-2\mathtt{q}_{20,2}(\bfk_1,\bfk_1)\mathtt{p}_{20,2}(-\bfk_1,2\bfk_1)+3\mathtt{p}_{30,1}(\bfk_1),\\
a_4&=\left\langle 2J_{20}(\ee^{\ii\bfk_1\cdot\bfx},\eta_{2,0101})+2J_{20}(\ee^{\ii\bfk_2\cdot\bfx},\eta_{2,1001})+2J_{20}(\ee^{-\ii\bfk_2\cdot\bfx},\eta_{2,1100})
+6J_{30}(\ee^{\ii\bfk_1\cdot\bfx},\ee^{\ii\bfk_2\cdot\bfx},\ee^{-\ii\bfk_2\cdot\bfx}),\ee^{\ii\bfk_1\cdot\bfx}\right\rangle\\
&=-\frac{4}{g}\mathtt{p}_{20,1}(\bfk_2)\mathtt{p}_{20,2}(\bfk_1,\bfzero)-4\mathtt{q}_{20,2}(\bfk_1,-\bfk_2)\mathtt{p}_{20,2}(\bfk_2,\bfk_1-\bfk_2)-4\mathtt{q}_{20,2}(\bfk_1,\bfk_2)\mathtt{p}_{20,2}(-\bfk_2,\bfk_1+\bfk_2)+6\mathtt{p}_{30,2}(\bfk_1,\bfk_2),\\
b_3&=\left\langle 2J_{20}(\ee^{\ii\bfk_2\cdot\bfx},\eta_{2,1010})+2J_{20}(\ee^{\ii\bfk_1\cdot\bfx},\eta_{2,0110})+2J_{20}(\ee^{-\ii\bfk_1\cdot\bfx},\eta_{2,1100})+6J_{30}(\ee^{\ii\bfk_2\cdot\bfx},\ee^{\ii\bfk_1\cdot\bfx},\ee^{-\ii\bfk_1\cdot\bfx}),\ee^{\ii\bfk_2\cdot\bfx}\right\rangle\\
&=-\frac{4}{g}\mathtt{p}_{20,1}(\bfk_1)\mathtt{p}_{20,2}(\bfk_2,\bfzero)-4\mathtt{q}_{20,2}(-\bfk_1,\bfk_2)\mathtt{p}_{20,2}(\bfk_1,\bfk_2-\bfk_1)-4\mathtt{q}_{20,2}(\bfk_1,\bfk_2)\mathtt{p}_{20,2}(-\bfk_1,\bfk_1+\bfk_2)+6\mathtt{p}_{30,2}(\bfk_2,\bfk_1),\\
b_4&=\left\langle 2J_{20}(\ee^{\ii\bfk_2\cdot\bfx},\eta_{2,0101})+2J_{20}(\ee^{-\ii\bfk_2\cdot\bfx},\eta_{2,0200})+3J_{30}(\ee^{\ii\bfk_2\cdot\bfx},\ee^{\ii\bfk_2\cdot\bfx},\ee^{-\ii\bfk_2\cdot\bfx}),\ee^{\ii\bfk_2\cdot\bfx}\right\rangle\\
&=-\frac{4}{g}\mathtt{p}_{20,1}(\bfk_2)\mathtt{p}_{20,2}(\bfk_2,\bfzero)-2\mathtt{q}_{20,2}(\bfk_2,\bfk_2)\mathtt{p}_{20,2}(-\bfk_2,2\bfk_2)+3\mathtt{p}_{30,1}(\bfk_2),
\end{align*}
where $J_{30}=\frac{1}{3!}\mathrm{d}_1^3 J[0,\bfzero]$.

One finds that
\begin{alignat*}{2}
	J_{30}(\ee^{\ii\bfk\cdot\bfx},\ee^{\ii\bfk\cdot\bfx},\ee^{-\ii\bfk\cdot\bfx})&=\mathtt{p}_{30,1}(\bfk)\ee^{\ii\bfk\cdot\bfx}, & & \qquad \bfk\neq\bfzero,\\
	J_{30}(\ee^{\ii\bfk\cdot\bfx},\ee^{\ii\bfl\cdot\bfx},\ee^{-\ii\bfl\cdot\bfx})&=\mathtt{p}_{30,2}(\bfk,\bfl)\ee^{\ii\bfk\cdot\bfx},& & \qquad \bfk\neq-\bfl,
\end{alignat*}
where
\begin{align*}
\mathtt{p}_{30,1}(\bfk)&= \frac{2}{3}\mathtt{T}_{10}(\bfk)\cdot \mathtt{T}_{20,1}(\bfk)+\frac{1}{3}\mathtt{T}_{10}(\bfk)\cdot\mathtt{T}_{20,2}(\bfk,\bfk)-\frac{1}{3}\alpha(\bfc_0\cdot\bfk)(\bfc_0^\perp\cdot\bfk)-\frac{1}{3}(\bfc_0\cdot\bfk)(\mathtt{T}_{10}(\bfk)\cdot\bfk)\\
&\qquad\mbox{} -\frac{\alpha^2}{2}\mathtt{T}_{10}(\bfk)\cdot\bfc_0+\frac{\alpha}{3}\left(2\bfc_0^\perp\cdot\mathtt{T}_{20,1}(\bfk)+\bfc_0^\perp\cdot\mathtt{T}_{20,2}(\bfk,\bfk)\right)+\mathtt{T}_{30,1}(\bfk)\cdot\bfc_0-\frac{\beta}{2}|\bfk|^4,\\
\mathtt{T}_{30,1}(\bfk)&=-\frac{1}{3}\bfr_1(\bfk)\frac{c_0\cdot\bfk}{|\bfk|^2}\mathtt{c}(|\bfk|)\mathtt{c}(2|\bfk|)-\frac{1}{12}\bfr_2(2\bfk)\frac{\bfc\cdot\bfk}{|\bfk|^2}\mathtt{c}(|\bfk|)-\frac{1}{6}\bfk(\bfc_0\cdot\bfk)\mathtt{c}(|\bfk|)\\
&\qquad\mbox{} -\frac{1}{2}\bfr_1(\bfk)(\alpha^2-|\bfk|^2)\frac{\bfc_0\cdot\bfk}{|\bfk|^2}+\frac{\alpha}{3}\bfr_2(\bfk)^\perp\frac{\bfc_0\cdot\bfk}{|\bfk|^2}\\
&\qquad\mbox{}+\frac{\alpha}{12}\bfr_1(2\bfk)^\perp\frac{\bfc_0\cdot\bfk}{|\bfk|^2}\mathtt{c}(|\bfk|)+\frac{\alpha}{6}\bfr_2(\bfk)^\perp\frac{\bfc_0\cdot\bfk}{|\bfk|^2}+\frac{\alpha}{6}\bfr_1(\bfk)\frac{\bfc_0^\perp\cdot\bfk}{|\bfk|^2}\mathtt{c}(|\bfk|)\\
&\qquad\mbox{}+\frac{\alpha}{12}\bfr_2(2\bfk)\frac{\bfc_0^\perp\cdot\bfk}{|\bfk|^2}+\frac{\alpha}{6}\bfk_1(\bfc_0^\perp\cdot\bfk)+\frac{\alpha^2}{6}\bfr_1(\bfk)\frac{\bfc_0\cdot\bfk}{|\bfk|^2}
\end{align*}
and
\begin{align*}
\mathtt{p}_{30,2}(\bfk,\bfl)&=\frac{1}{3}\left(\mathtt{T}_{10}(\bfk)\cdot\mathtt{T}_{20,1}(\bfk)+\mathtt{T}_{10}(\bfl)\cdot\mathtt{T}_{20,2}(\bfk,-\bfl)+\mathtt{T}_{10}(\bfl)\cdot\mathtt{T}_{20,2}(\bfk,\bfl)\right)-\frac{\alpha}{3}(\bfc_0\cdot\bfl)(\bfc_0^\perp\cdot\bfl)\\
&\qquad\mbox{}-\frac{1}{3}(\bfc_0\cdot\bfl)(\mathtt{T}_{10}(\bfk)\cdot\bfl)-\frac{\alpha^2}{6}\left(\mathtt{T}_{10}(\bfk)\cdot\bfc_0+2\mathtt{T}_{10}(\bfl)\cdot\bfc_0\right)+\frac{\alpha}{3}\big(\bfc_0^\perp\cdot\mathtt{T}_{20,1}(\bfl)\\
&\qquad\mbox{}+\bfc_0^\perp\cdot\mathtt{T}_{20,2}(\bfk,-\bfl)+\bfc_0^\perp\cdot\mathtt{T}_{20,2}(\bfk,\bfl)\big)+\mathtt{T}_{30,2}(\bfk,\bfl)\cdot\bfc_0-\frac{\beta}{6}\left(|\bfk|^2|\bfl|^2+2(\bfk\cdot\bfl)^2\right),\\
\mathtt{T}_{30,2}(\bfk,\bfl)&=-\frac{1}{6}\bfr_1(\bfk)\frac{\bfk}{|\bfk|^2}\cdot\bigg(\bfr_1(\bfk-\bfl)\bigg[\frac{\bfk-\bfl}{|\bfk-\bfl|^2} \cdot \bfr_3(\bfk,\bfl)\bigg]+\bfr_1(\bfk+\bfl) \bigg[\frac{\bfk+\bfl}{|\bfk+\bfl|^2}\cdot\bfr_3(\bfk,\bfl)\bigg]\bigg)\\
&\qquad\mbox{}-\frac{1}{12}\bfr_2(\bfk-\bfl)\bigg[\frac{\bfk-\bfl}{|\bfk-\bfl|^2}\cdot\bfr_3(\bfk,\bfl)\bigg] -\frac{1}{12}\bfr_2(\bfk+\bfl)\bigg[\frac{\bfk+\bfl}{|\bfk+\bfl|^2}\cdot\bfr_3(\bfk,\bfl)\bigg]\\
&\qquad\mbox{} -\frac{1}{6}\bfl\bigg[\bfl\cdot\bfr_3(\bfk,\bfl)\bigg]-\frac{1}{6}\bfr_1(\bfk)\frac{\bfk}{|\bfk|^2} \cdot \bigg[2\bfr_2(\bfl)\frac{\bfc_0\cdot\bfl}{|\bfl|^2}+\bfr_2(\bfk)\frac{\bfc_0\cdot\bfk}{|\bfk|^2}\bigg]\\
&\qquad\mbox{}+\frac{1}{6}\bfk\bigg[\bfk\cdot\bfr_1(\bfl)\frac{\bfc_0\cdot\bfl}{|\bfl|^2}\bigg]+\frac{\alpha}{6}\bfr_2(\bfl)^\perp\frac{\bfc_0\cdot\bfl}{|\bfl|^2}\\
&\qquad\mbox{}+\frac{\alpha}{12}\bigg(\bfr_1(\bfk-\bfl)^\perp\bigg[\frac{\bfk-\bfl}{|\bfk-\bfl|^2}\cdot\bfr_3(\bfk,\bfl)\bigg]
+\bfr_2(\bfl)^\perp\frac{\bfc_0\cdot\bfl}{|\bfl|^2}+\bfr_2(\bfk)^\perp\frac{\bfc_0\cdot\bfk}{|\bfk|^2}\bigg)\\
&\qquad\mbox{}+\frac{\alpha}{12}\bigg(\bfr_1(\bfk+\bfl)^\perp\bigg[\frac{\bfk+\bfl}{|\bfk+\bfl|^2}\cdot\bfr_3(\bfk,\bfl)\bigg]
+\frac{1}{|\bfk|^2}\bfr_2(\bfk)^\perp\frac{\bfc_0\cdot\bfk}{|\bfk|^2}+\bfr_2(\bfl)^\perp\frac{\bfc_0\cdot\bfl}{|\bfl|^2}\bigg)\\
&\qquad\mbox{}+\frac{\alpha}{6}\bfr_1(\bfk) \frac{\bfk}{|\bfk|^2}\cdot\bigg(\bfr_1(\bfk-\bfl)\frac{\bfc_0^\perp\cdot(\bfk-\bfl)}{|\bfk-\bfl|^2}
+\bfr_1(\bfk+\bfl)\frac{\bfc_0^\perp\cdot(\bfk+\bfl)}{|\bfk+\bfl|^2}\bigg) \\
&\qquad\mbox{}+\frac{\alpha}{6}\bfr_2(\bfk-\bfl)\frac{\bfc_0\cdot(\bfk-\bfl)}{|\bfk-\bfl|^2}+\frac{\alpha}{6}\bfr_2(\bfk+\bfl)\frac{\bfc_0\cdot(\bfk+\bfl)}{|\bfk+\bfl|^2}\\
&\qquad\mbox{}+\frac{\alpha}{3}\bfl(\bfc_0^\perp\cdot\bfl)+\frac{\alpha^2}{6}\bfr_1(\bfk)\frac{\bfc_0\cdot\bfk}{|\bfk|^2}-\frac{\alpha}{6}\bfr_1(\bfk)\bigg[\frac{\bfk}{|\bfk|^2}\cdot\bfr_1(\bfl)^\perp\bigg]\frac{\bfc_0\cdot\bfl}{|\bfl|^2} -\frac{\alpha^2}{6}\bfr_1(\bfl)\frac{\bfc_0\cdot\bfl}{|\bfl|^2},
\end{align*}
with\pagebreak
\begin{align*}
\bfr_1(\bfk)&=\alpha\bfk^\perp+\bfk\mathtt{c}(|\bfk|),\\
\bfr_2(\bfk)&=\bfk(\alpha^2-|\bfk|^2)-\alpha\bfk^\perp\mathtt{c}(|\bfk|),\\
\bfr_3(\bfk,\bfl)&=\bfr_1(\bfk)\frac{\bfc_0\cdot\bfk}{|\bfk|^2}+\bfr_1(\bfl)\frac{\bfc_0\cdot\bfl}{|\bfl|^2}.
\end{align*}

\end{itemize}

\enlargethispage{1cm}

\section*{Acknowledgements}

This research was supported by the Swedish Research Council under grant no.\ 2021-06594 while the authors were in residence at Institut Mittag-Leffler in Djursholm, Sweden in Autumn 2023. Additional support was
provided by the Swedish Research Council (grant no.\ 2020-00440), the European Research Council (grant agreement no.\ 678698) and the Knut and Alice Wallenberg Foundation. The project has also received funding from the European Union's Horizon 2020
research and innovation programme under the Marie Sk\l{}odowska-Curie grant agreement no.\ 101034255.

\vspace{-4mm}\begin{figure}[!h]
\includegraphics[width = .1\textwidth]{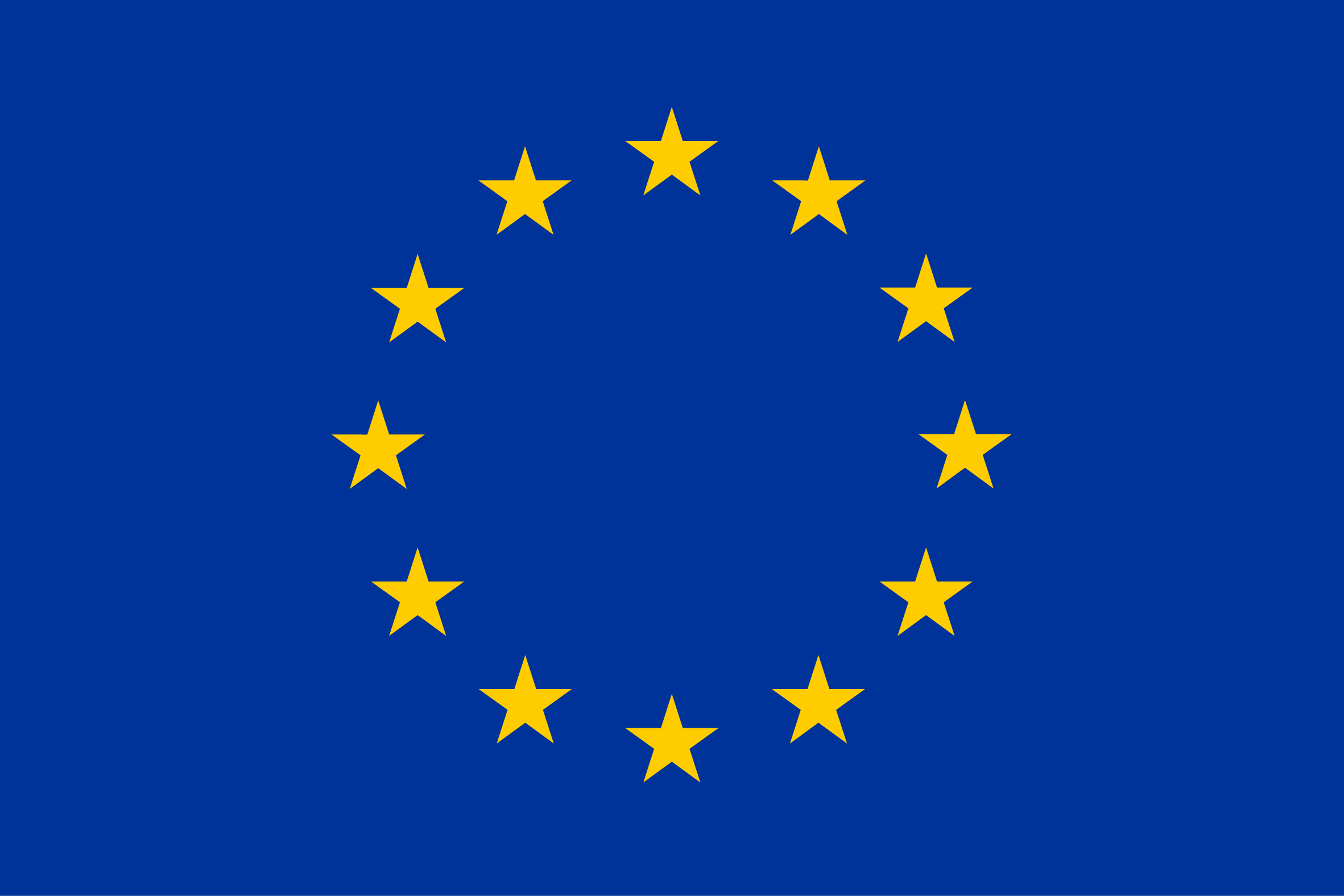}
\end{figure}

\vspace{-6mm}\bibliography{mdg}
\bibliographystyle{standard}

\end{document}